\documentclass{article}
\usepackage{amsmath,amssymb,amsthm,mathrsfs,verbatim}
\usepackage[all]{xy}
\usepackage[dvips]{graphicx}
\usepackage{comment}
\usepackage{MnSymbol}

\newtheorem{Thm}{Theorem}[section]
\newtheorem{Prop}[Thm]{Proposition}
\newtheorem{Cor}[Thm]{Corollary}
\newtheorem{Lem}[Thm]{Lemma}
\newtheorem{Def}[Thm]{Definition}

\theoremstyle{definition}
\newtheorem{Ex}[Thm]{Example}
\newtheorem{Rem}[Thm]{Remark}
\newtheorem{Assum}[Thm]{Assumption}

\newcommand{\bfa}{{\bf a}}
\newcommand{\bb}{{\bf b}}
\newcommand{\bi}{{\bf i}}
\newcommand{\bI}{{\bf I}}
\newcommand{\nn}{{\bf n}}
\newcommand{\zz}{{\bf z}}

\newcommand{\C}{\mathbb{C}}
\newcommand{\D}{\mathbb{D}}

\newcommand{\HH}{\mathbb{H}}
\newcommand{\R}{\mathbb{R}}
\newcommand{\Z}{\mathbb{Z}}
\newcommand{\N}{\mathbb{N}}
\newcommand{\Q}{\mathbb{Q}}
\newcommand{\PP}{\mathbb{P}}
\newcommand{\bS}{\mathbb{S}}

\newcommand{\cC}{\mathscr{C}}
\newcommand{\cb}{\mathfrak{b}}
\newcommand{\Hol}{\mathop{\mathrm{Hol}}\nolimits}

\newcommand{\vh}{\vec{h}}

\newcommand{\ve}{\vec{e}}

\newcommand{\SL}{\mathop{\mathrm{SL}}\nolimits}

\newcommand{\GL}{\mathop{\mathrm{GL}}\nolimits}
\newcommand{\U}{\mathop{\mathrm{U}}\nolimits}
\newcommand{\Sp}{\mathop{\mathrm{Sp}}\nolimits}

\newcommand{\open}{\mathrm{open}}
\newcommand{\Conv}{\mathrm{Conv}}
\newcommand{\Span}{\mathrm{Span}}

\begin{document}
\title{Local Calabi--Yau manifolds of type $\widetilde{A}$ via SYZ mirror symmetry}  %SYZ mirror symmetry and modular forms
\author{Atsushi Kanazawa \ \ \ Siu-Cheong Lau}
\date{}
\maketitle

% Any holomorphic sphere in a non-compact divisor $D$ must be contained in a toric divisor of the toric variety $D$

\begin{abstract}
We carry out the SYZ program for the local Calabi--Yau manifolds of type $\widetilde{A}$ by developing an equivariant SYZ theory for the toric Calabi--Yau manifolds of infinite-type. 
Mirror geometry is shown to be expressed in terms of the Riemann theta functions and generating functions of open Gromov--Witten invariants, 
whose modular properties are found and studied in this article.  
%We also derive explicit formulae for the generating functions which are open analogs of the Yau--Zaslow formula and show that they have nice modular properties. 
Our work also provides a mathematical justification for a mirror symmetry assertion of the physicists Hollowood--Iqbal--Vafa \cite{HIV}. 
\end{abstract}

%\tableofcontents % for convenience, can deleter in the end 

%%%%%%%%%%%%%%%%%%%%%%%%%%%%%%%%%%%%%%%%%%%%%%%%%%%%%%%%%%%%%%%%%%%%%%%%%%%%%%%%%%%%%%%%%%%%%%%%%%%%%%%
%%%%%%%%%%%%%%%%%%%%%%%%%%%%%%%%%%%%%%%%%%%%%%%%%%%%%%%%%%%%%%%%%%%%%%%%%%%%%%%%%%%%%%%%%%%%%%%%%%%%%%%

\section{Introduction} \label{section: Intro}

The aim of the present article is to investigate SYZ mirror symmetry of the local Calabi--Yau manifolds of type $\widetilde{A}$.  
This class of Calabi--Yau manifolds serves as local models of higher dimensional versions of Schoen's Calabi--Yau 3-folds \cite{Sch}, 
whose mirror symmetry was partially verified by Hosono--Saito--Stienstra \cite{HSS}.
We construct their SYZ mirror manifolds and deduce their modular properties by developing an equivariant SYZ theory for the toric Calabi--Yau manifolds of infinite-type. 
Mirror symmetry for this class of Calabi--Yau manifolds was studied in the beautiful works of Gross--Siebert \cite{GS2} in algebraic geometry, 
while we focus on symplectic aspects and open Gromov--Witten invariants in this article. 
The main contribution of the present article is threefold. 
The first is to provides a new class of SYZ mirror pairs of local Calabi--Yau manifolds, which will be useful in the study of the compact case, 
by the precise calculation of mirror maps in terms of open Gromov--Witten invariants and the relationship to the Riemann theta functions.  
The second is to reveal some interesting links between modular properties and quantum corrections in SYZ mirror symmetry in these special but important cases. 
The third is to provide a mathematical justification for a mirror symmetry assertion of the physicists Hollowood--Iqbal--Vafa \cite{HIV}. 
%Another important feature of our work is to demonstrate by example how powerful general theories such as the open mirror theorem and the theory of normalized slab functions of Gross-Siebert \cite{GS} are 
%in explicit computations. 
%We will also see that general theories such as the open mirror theorem and normalized slab functions of Gross-Siebert \cite{GS} are very useful to make explicit computations (Section \ref{sec:renorm}). 
%The modular property of Schoen's Calabi--Yau 3-fold was partially verified by Hosono--Saito--Stienstra \cite{HSS}.  
%More concretely we can express the SYZ mirrors in terms of the Riemann theta functions and Siegel modular forms. 
%Our work presents the first example of modularity of the generating function of open Gromov--Witten invariants of Calabi--Yau manifolds in high dimensions. 
%We focus on the open Gromov--Witten invariants and follow the symplectic formulation of the SYZ construction of Auroux \cite{Aur} and Chan--Lau--Leung \cite{CLL}. 
%The open mirror theorem and normalized slab functions of Gross-Siebert \cite{GS} will be very useful to make explicit computations (Section \ref{sec:renorm}). 

\subsection*{Main target geometries}
The easiest case of the main target geometries is the local Calabi--Yau surface of type $\widetilde{A}_{d-1}$ for $d \geq 1$. 
It is the total space of the elliptic fibration over the unit disc $\D=\{|z|<1\} \subset \C$, 
where all fibers are smooth except for the central fiber, which is a nodal union of $d$ rational curves forming a cycle. 
\begin{figure}[htbp]
 \begin{center} 
  \includegraphics[width=80mm]{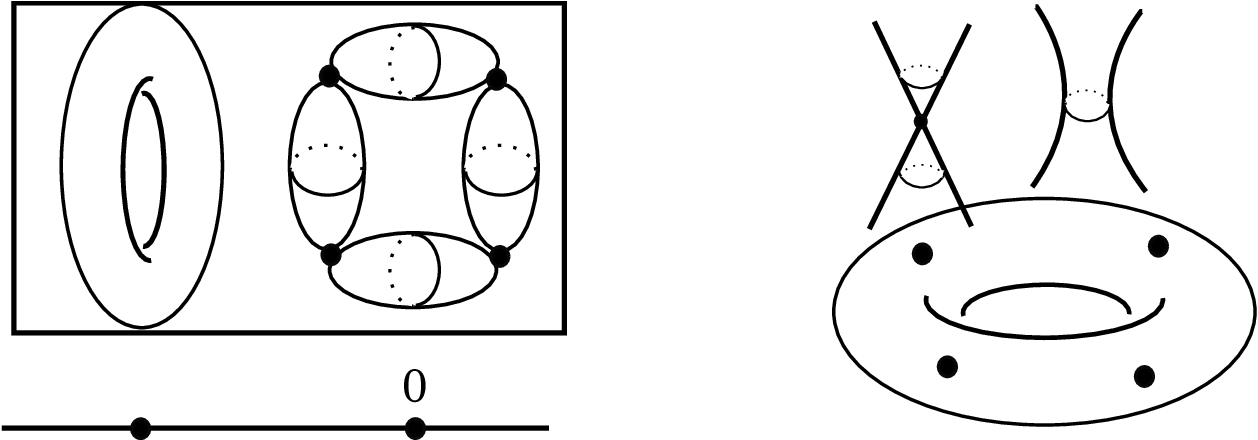}
 \end{center}
  \caption{Local Calabi--Yau surface $\widetilde{A}_{d-1}$ and its SYZ mirror} 
\label{fig:K3}
\end{figure}
%By Kodaira's classification, the singular fibers of elliptic fibrations have a nice correspondence with the affine Dynkin diagrams of ADE type. 
%which describes the intersection matrices of the singular fibers.  
%The Calabi--Yau surface of affine Dynkin type $\widetilde{A}_{d-1}$ serves as one of the fundamental building blocks of a compact elliptic fibration. 
%For instance, a generic elliptic Calabi--Yau surface has precisely $24$ singular fibers of type $\widetilde{A}_0$.  
We construct the SYZ mirrors of the local Calabi--Yau surfaces $\widetilde{A}_{d-1}$ and express them by modular objects. 
For instance, the most fundamental surface $\widetilde{A}_0$ has the SYZ mirror given in terms of the Jacobi theta function $\vartheta$:
$$ uv = \phi(q) \cdot \vartheta \left( \zeta - \frac{\tau}{2}; \tau \right) $$
where $q = e^{2\pi i \tau}$ and $\zeta \in \C/\langle 1, \tau \rangle$ for $\tau \in \HH$, and $u$ and $v$ are sections of certain line bundles on the elliptic curve $\C/\langle 1, \tau \rangle$. 
In other words, the SYZ mirror is the total space of a conic fibration over the elliptic curve $\C/\langle 1, \tau \rangle$.  
The function $\phi(q)$ is the generating function of the open Gromov--Witten invariants of a Lagrangian torus fiber of a SYZ fibration of $\widetilde{A}_0$. 
%$$ \phi(q) := \sum_{k=0}^\infty n_{\beta_0 + k F} \cdot q^k $$
%where $\beta_0$ is the basic disc class and $F$ is the elliptic fiber class of the $\widetilde{A}_0$ surface (the right figure of Figure \ref{fig:K3}).  
%\begin{figure}[htbp]
% \begin{center} 
%  \includegraphics[width=20mm]{A0Disc.eps}
% \end{center}
%  \caption{Stable discs in the $\tilde{A}_0$ surface with base disc class $\beta_0$.} 
%\label{fig:A0_disc}
%\end{figure}
It is a crucial object giving the quantum corrections in the framework of SYZ mirror symmetry \cite{SYZ} and 
we will prove an open analog of the Yau--Zaslow formula \cite{YZ}, namely $\phi(q) =  \frac{q^{\frac{1}{24}}}{\eta(q)}$ 
where $\eta(q)$ is the Dedekind eta function (Corollary \ref{cor:rootYZ}).  
%The following is an open analog of the Yau--Zaslow formula. 
%We will derive the closed formula $\phi(q) =  \frac{q^{\frac{1}{24}}}{\eta(q)}$ ($24$-th root of Yau--Zaslow formula). %which is exactly the $24$-th root of the Yau--Zaslow formula.  
%which gives a geometric reason why the Yau--Zaslow formula holds by dividing an elliptic K3 surface into $24$ copies of $\widetilde{A}_0$ surface.

%\begin{Thm}[Corollary \ref{cor:rootYZ}, $24$-th root of Yau--Zaslow formula] \label{thm:root-YZ-intro}
%$$\phi(q) =  \frac{q^{\frac{1}{24}}}{\eta(q)}.$$
%\end{Thm}

%The above equation defines a conic fibration over the elliptic curve $\C/\langle 1, \tau \rangle$.  In general the SYZ mirror of $\widetilde{A}_{d-1}$ is a conic fibration, degenerating at $d$ points, which are the zeros of a theta function with character (the right figure of Figure \ref{fig:TypeAd}).  
The above surface geometry has a natural extension to the higher dimensions.  
Namely, for $({\bf d})=(d_1,\ldots,d_l) \in \Z^{l}_{\ge 1}$, 
we consider a crepant resolution $X_{({\bf d})}$ of the multiple fiber product $\widetilde{A}_{d_1-1} \times_\D \ldots \times_\D \widetilde{A}_{d_l-1}$. 
%where $S_{d_k} \rightarrow \D$ denotes the local elliptic surface of type $\widetilde{A}_{d_k-1}$.   
We call such a manifold a {\it local Calabi--Yau manifold of type $\widetilde{A}$}, which serves as a local model of a higher dimensional analog of Schoen's Calabi--Yau 3-fold \cite{Sch}.  
%, whose existence is guaranteed by the existence of the maximal subdivision of the hypercube. 
We will prove that the SYZ mirror of a local Calabi--Yau manifold of type $\widetilde{A}$ has a beautiful expression in terms of the Riemann theta functions with characteristics. 
To illustrate this, let us state the main theorem for the local Calabi--Yau 3-fold $X_{(1,1)}$.

\begin{Thm}[Theorem \ref{thm:CY3/Z^2}] \label{thm:mirIntro}
The SYZ mirror of the local Calabi--Yau 3-fold $X_{(1,1)}$ is given by, for $(\zeta_1, \zeta_2) \in \C^2/(\Z^2\oplus\Omega \Z^2)$, 
\begin{equation} \label{eq:mir11}
uv = \Delta(\Omega) \cdot 
\Theta_2
\left(\zeta_1- \frac{\tau}{2}, \zeta_2- \frac{\rho}{2}; \Omega\right)
\end{equation}
where $\Theta_2$ is the genus $2$ Riemann theta function, $\Omega:=\begin{bmatrix}
                \rho   &  \sigma  \\
                 \sigma& \tau  \\
\end{bmatrix}$ lies in the Siegel upper half-plane, and $\Delta(\Omega)$ is the generating function of the open Gromov--Witten invariants of a Lagrangian torus fiber of a SYZ fibration of  $X_{(1,1)}$. 
%$$\Delta(\Omega)
% = \exp \left(\sum_{j \geq 2} \frac{(-1)^j }{j } 
%\sum_{\substack{({\bf l}_i=(l_i^1, l_i^2)\in \Z^{2} \setminus 0)_{i=1}^j  \\ \textrm{with } \sum_{i=1}^j  {\bf l}_i = 0}}
%\exp\left(\sum_{k=1}^j \pi i {\bf l}_k\cdot \Omega \cdot {\bf l}^T_k\right) \right). $$  
\end{Thm}

As before $u$ and $v$ are sections of suitable line bundles on the abelian surface $\C^2/(\Z^2\oplus\Omega \Z^2)$. 
Thus Equation \eqref{eq:mir11} defines a conic fibration over $\C^2/(\Z^2\oplus\Omega \Z^2)$ with discriminant being the genus $2$ curve, 
called the {\it mirror curve}, given by the zero locus of $\Theta_2$.
The function $\Delta(\Omega)$ is a 3-dimensional analog of the generating function $\phi(q)$ 
and its higher dimensional generalizations are also considered in Theorem \ref{thm:mir_X_1}. 
They have interesting modular properties (Proposition \ref{prop:modular}), and we anticipate that they are closely related with the higher genus Siegel modular forms. 
%generalizing the Dedekind eta function. 
%It follows from this construction that the mirror curve is given by the natural embedding of a genus two curve into its Jacobian surface by the Abel--Jacobi map, endowing the abelian surface with the principal polarization. 

%The mirror Equation \eqref{eq:mir11} embeds the mirror curve into its Jacobian surface as the theta divisor by the Abel--Jacobi map. 
%We observe that the complex moduli space of the SYZ mirror has a rich structure, which can be studied via the toroidal Torelli map 
%$\bar{\mathfrak{t}} : \overline{\mathcal{M}}_2 \rightarrow \overline{\mathcal{A}}_2$, 
%where $\overline{\mathcal{M}}_2$ is the the moduli space of stable genus $2$ curves 
%and $\overline{\mathcal{A}}_2$ is the Voronoi compactification of the moduli space $\mathfrak{H}_2/\Sp_4(\Z)$ 
%of principally polarized abelian surfaces. 
%For instance, the large complex structure limit corresponds to the stable genus $2$ curve obtained by gluing $2$ copies of $\PP^1$'s at 3 points. 
%It is contained in the maximal degeneration of the abelian surfaces (gluing of $2$ copies of $\PP^2$'s along 3 $\PP^1$'s) as the theta divisor 
%(the left figure of Figure \ref{fig:LCSL}). 
% \begin{figure}[htbp]
% \begin{center} 
%  \includegraphics[width=60mm]{LCSL.eps}
% \end{center}
%  \caption{Large complex structure limit and Crit$(X_{(1,1)} \to \D)$}
%\label{fig:LCSL}
%\end{figure}

%Mirror symmetry provided in Theorem \ref{thm:mirIntro} can be understood geometrically as follows. 

In the physics literature \cite{HIV}, Hollowood--Iqbal--Vafa constructed the local Calabi--Yau 3-fold $X_{(1,1)}$ in a physical way 
and asserted that its mirror is given by the zero locus of the genus $2$ Riemann theta function. 
Their supporting arguments are based on $3$ different techniques: matrix models, geometric engineering and instanton calculus. 
They relate $X_{(1,1)}$ with the geometry of the resolved conifold $\mathcal{O}_{\PP^1}(-1)^{\oplus2}$, 
whose mirror curve is a $4$-punctured $\PP^1$ (Hori--Iqbal--Vafa \cite{HIV2}).   
 \begin{figure}[htbp]
 \begin{center} 
  \includegraphics[width=130mm]{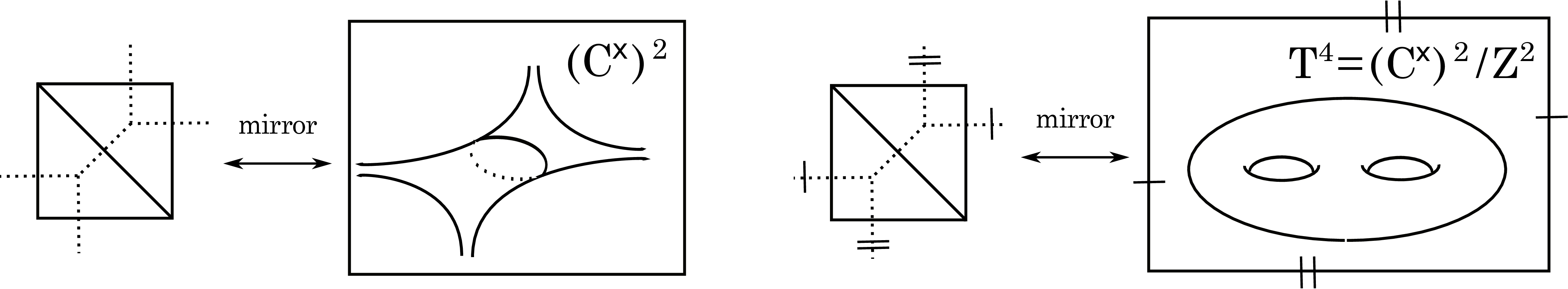}
 \end{center}
  \caption{Mirror correspondence for $\mathcal{O}_{\PP^1}(-1)^{\oplus2}$ and $X_{(1,1)}$}
\label{fig:MirrorCorresp4}
\end{figure}
As we will see in Section \ref{section: 3-fold}, $X_{(1,1)}$ is intuitively constructed by gluing the opposite sides of the toric web diagram of $\mathcal{O}_{\PP^1}(-1)^{\oplus2}$.    
Accordingly its mirror curve becomes a genus 2 curve constructed by gluing the punctures of the $4$-punctured $\PP^1$ in pairs.  
%The above `gluing' can be justified by considering a quotient of a toric Calabi--Yau manifolds of infinite-type.  

In this article we take the geometric SYZ approach to derive the mirrors, and our result agrees with the work \cite{HIV} of Hollowood--Iqbal--Vafa for $X_{(1,1)}$. 
The crucial advantage of our mathematical work is that not only the mirror is constructed geometrically, 
but also we obtain a closed formula of the generating function $\Delta(\Omega)$ of the open Gromov--Witten invariants, 
which has interesting modular properties but was not captured in the physics argument. 
%The function $\Delta(\Omega)$ depends on moduli parameters, and thus will be important when we study the mirror in a family. 

\begin{Rem} \label{rem:AAK}
The SYZ mirror construction in the reverse direction was carried out by Abouzaid--Auroux--Katzarkov \cite[Theorem 10.4]{AAK}. 
Namely, they consider the zero locus $H$ of the Riemann theta function in an abelian variety $V$ and take the blowup of $V \times \C$ along $H \times \{0\}$ 
(which plays the role of the conic fibration above).  
Their general theory applied to this example provides a Lagrangian fibration on the total space and its SYZ mirror, 
which turns out to be the local Calabi--Yau manifold of type $\widetilde{A}$.  
In their situation there are no non-constant holomorphic spheres in the conic fibration. 
On the other hand, there are non-trivial holomorphic spheres in our local Calabi-Yau manifolds, which lead us to the open Gromov--Witten generating function $\Delta(\Omega)$.
\end{Rem}

\subsection*{Varieties of general type}
A version of homological mirror symmetry for the varieties of general-type was formulated by Kapustin--Katzarkov--Orlov--Yotov \cite{KKOY} and 
it was proven for genus $2$ and higher curves by Seidel \cite{Sei} and Efimov \cite{Efi} respectively. 
As an application of the SYZ mirror pairs, we obtain the Landau--Ginzburg mirrors of general-type hypersurfaces in the polarized abelian varieties.  
They were speculated in \cite{Sei-Spec}, and they were also derived from the SYZ construction of Abouzaid--Auroux--Katzarkov \cite[Theorem 10.4]{AAK}.   

First, observe that the critical locus of the canonical holomorphic function $w:X_{(1,1)} \to \D$, 
which is the $1$-skeleton of the local Calabi--Yau 3-fold $X_{(1,1)}$, 
is the union of $3$ $\PP^1$'s forming a `$\theta$'-shape. %(the right figure of Figure \ref{fig:LCSL}).  
This is in agreement with the work of Seidel \cite{Sei}\footnote{
Note that A- and B-models are swapped in our work. Namely the genus $2$ curve appear in the B-model.} 
stating that the mirror of a genus $2$ curve is exactly such a Landau--Ginzburg model $w:X_{(1,1)} \to \D$.
 
In general the SYZ mirror of the local Calabi--Yau $(l+1)$-fold $X_{({\bf d})}$ is a conic fibration $X_{({\bf d})}^\vee \rightarrow A^\vee$ over a $({\bf d})$-polarized abelian variety $A^\vee$,  
with discriminant locus being a general-type hypersurface defined by the $({\bf d})$-polarization. 
Then it is shown that the generic fiber $A$ of $w:X_{({\bf d})} \to \D$ is mirror symmetric to the base $A^\vee$ of the conic fibration $X_{({\bf d})}^\vee \rightarrow A^\vee$ 
(Section \ref{section:fiber-base mirror dual}). 
Moreover, the Landau--Ginzburg model $w:X_{({\bf d})} \to \D$ (or its critical locus furnished with perverse structure explained by Gross--Katzarkov--Ruddat \cite{GKR})  
serves as a mirror of the discriminant locus of the conic fibration $X_{({\bf d})}^\vee \rightarrow A^\vee$.   
\begin{center}
 \begin{tabular}{c|c|c} 
 $(l+1)$-dim  & total space CY manifold $X_{({\bf d})}$   & total space CY manifold $X_{({\bf d})}^\vee$ \\ \hline
 $l$-dim        & fiber abelian variety $A$ &  base abelian variety $A^\vee$ \\ \hline
$(l-1)$-dim & perverse critical locus & dicsriminant locus \\ 
 \end{tabular}
 \end{center}
We anticipate that this mirror correspondence between the perverse critical locus of a family near the large complex structure limit 
and the discriminant locus of a conic fibration holds in a more general setting (Section \ref{section: speculation}). 
%It conjecturally produces a large class of examples of mirror symmetry for varieties of general type. 
Mirror pairs of perverse curves in Hodge-theoretic aspects can be found in the work of Ruddat \cite{Rud}. 
%This mirror correspondence was developed by Ruddat \cite{Rud2} for curves in the Batyrev mirror pairs of Calabi-Yau 3-folds from the Hodge-theoretical aspects.  
%He also studied the link with the Gross--Siebert program \cite{Rud}. 
%Mirror symmetry for the varieties of general type is still lurking and only partially explored area, and deserves further explorations.

%As in the $2$-dimensional case, the function $\Delta(\Omega)$ is the generating function of open Gromov--Witten invariants of certain Lagrangian torus in $X_{(1,1)}$, which encodes the quantum corrections of the SYZ mirror.  
%We expect that $\Delta(\Omega)$ is related to a genus $2$ Siegel modular form.  An explicit expression for $\Delta(\Omega)$ is derived in Theorem \ref{thm:CY3/Z^2}. 
%We can show that $\Delta(\Omega)$ obeys nice modular properties (Proposition \ref{Trans}). 
%The crucial property which is still currently under investigation is the behavior of $\Delta$ under the inversion $\Omega \mapsto -\Omega^{-1}$.
%Actually, our construction works in all dimensions and we conjecture that in all dimensions the generating functions of open Gromov--Witten invariants would give modular objects in all genera $g$. 
%which is a subject under rapid developments nowadays.

\subsection*{SYZ mirror symmetry}
The main technique in the present work is the SYZ mirror construction. 
It was introduced in the celebrated work of Strominger--Yau--Zaslow \cite{SYZ} and gave a geometric recipe to construct a mirror manifold of a Calabi--Yau manifold 
by taking the fiberwise torus dual of a Lagrangian torus fibration.   
The quantum corrections are captured by counting holomorphic discs bounded by the Lagrangian torus fibers. 
However, there are in general several major difficulties in realizing this SYZ construction. 
First, it requires the existence of a (special) Lagrangian fibration of the Calabi--Yau manifold. 
Second, in general the moduli spaces of holomorphic discs have highly technical obstructions \cite{FOOO}, and wall-crossing of the disc counting occurs \cite{Aur}.
The work of Gross--Siebert \cite{GS} gives a reconstruction of the mirror using tropical geometry, 
which provides a combinatorial recipe to compute the quantum corrections order-by-order. 
On the other hand we shall use the symplectic formulation in this article which uses disc enumeration and Gromov--Witten theory.

%The main technique we use for constructing the SYZ mirrors and computing the quantum corrections comes from toric Calabi--Yau manifolds. 
%Toric Calabi--Yau manifolds provide local models around curves or surfaces in general Calabi--Yau manifolds. For instance, the total space of the canonical line bundle of a compact toric Fano variety $Z$ is a toric Calabi--Yau manifold. 
%Its mirror has a close relation with the mirror of a compact Calabi--Yau hypersurface in $Z$.  
One crucial observation is that a local Calabi--Yau manifold of type $\widetilde{A}$ can be realized 
as the quotients of an open subset of a toric Calabi--Yau manifold of infinite-type by the discrete group $\Z^l$.  
In fact, this is the classical toric realization of degenerations of abelian varieties 
discussed by Mumford \cite{Mum}, Ash--Mumford--Rapoport--Tai \cite{AMRT} and Gross--Siebert \cite{GS2}, which is also explained in \cite[Section 8.4]{ABC}.  
Thanks to this realization, we can carry out an equivariant SYZ theory for the toric Calabi--Yau manifolds of infinite-type,  
generalizing the previous SYZ construction for the toric Calabi--Yau manifolds of finite-type \cite{CLL,LLW,CLT,CCLT}.

%Then we apply the theory to the specific geometries $X_{(d_1,\ldots,d_l)}$.

%Mirror symmetry for the 3-folds $X_{(1,1)},X_{1,2}$ and $X_{2,2}$ was first studied by Hollowood--Iqbal--Vafa \cite{HIV} using three different physical methods, 
%namely matrix models, geometric engineering and instanton calculus. 
%They discovered that the mirrors are given by the Riemann Theta functions. 
%In this paper we use SYZ and obtain more: namely the mirrors we construct are quantum-corrected. 
%The coefficients of their defining equations have interesting modular properties.  

\subsection*{Related works}
As mentioned above in Remark \ref{rem:AAK}, the SYZ mirror construction in the reversed direction was carried out by Abouzaid--Auroux--Katzarkov \cite{AAK}. 
Also mirror symmetry for the local Calabi--Yau threefold $X_{(1,1)}$ was derived physically by Hollowood--Iqbal--Vafa \cite{HIV}. 
The additional features in this article are open Gromov--Witten generating functions and mirror maps.

The appearance of the theta functions in SYZ mirror symmetry is natural and coherent with the previous literature.
From the early works of Fukaya \cite{Fuk} and Gross--Siebert \cite{GS2}, 
it is well understood that the theta functions appear naturally in mirror symmetry for the abelian varieties. 
The relation follows from SYZ and homological mirror symmetry \cite{Kon}. 
Namely, holomorphic line bundles over an abelian variety are mirror to Lagrangian sections of the mirror abelian variety,  
and theta functions are mirror to intersection points of the Lagrangian sections.  
This principle is greatly generalized in the work of Gross--Hacking--Keel--Siebert \cite{GHKS} to obtain canonical functions from toric degenerations 
by mirroring a combinatorial version of the intersection theory between the Lagrangian sections (which receives quantum corrections from scattering diagrams).
In the present article we take another perspective. 
Instead of mirror symmetry for an abelian variety itself, we consider the total space of a degeneration of abelian varieties as our target manifold and construct its SYZ mirror. 
In addition to the theta functions, we obtain the generating functions of certain open Gromov--Witten invariants 
which have important meanings to mirror maps of the total space.

%While this paper focuses on modularity in type $\widetilde{A}$ geometries, 
%there are interesting existing literature on modularity of type $\widetilde{D}$ and $\widetilde{E}$ geometries in the $1$-dimensional case. 
%For the elliptic orbifolds, Milanov--Ruan \cite{MR}, Satake--Takahashi \cite{ST}, and Shen--Zhou \cite{SZ,SZ2} derived modularity of the Gromov--Witten potentials. 
%The work \cite{LZ} of the second author with Zhou showed 
%that the quantum-corrected mirrors (and hence the open Gromov--Witten invariants) 
%of the elliptic orbifolds constructed in \cite{CHL,CHKL} of Type $\widetilde{D}_4,\widetilde{E}_6,\widetilde{E}_7$ are modular objects. 
%Modularity of the mirror of the elliptic orbifold of Type $\widetilde{E}_8$ is more subtle and was studied by Bringmann--Rolen--Zwegers \cite{BRZ}.  

%%%%%%%%%%%%%%%%%%%%%%%%%%%%%%%%%%%%%%%%%%%%%%%%%%%%%%%%%%%%%%%%%%%%%%%%%%%%%%%%%%%%%%%%%%%%%%%%%%%%%%%
%%%%%%%%%%%%%%%%%%%%%%%%%%%%%%%%%%%%%%%%%%%%%%%%%%%%%%%%%%%%%%%%%%%%%%%%%%%%%%%%%%%%%%%%%%%%%%%%%%%%%%%

%%%%%%%%%%%%%%%%%%%%%%%%%%%%%%%%%%%%%%%%%%%%%%%%%%%%%%%%%%%%%%%%%%%%%%%%%%%%%%%%%%%%%%%%%%%%%%%%%%%%%%%

\subsection*{Structure of Article}
Section \ref{section: toric infinite-type} lays foundations of the toric Calabi--Yau manifolds of infinite-type. 
Section \ref{section: SYZ mirror} develops (equivariant) SYZ mirror symmetry for the toric Calabi--Yau manifolds of infinite-type and their quotients by symmetries.  
Sections \ref{section: Calabi--Yau}-\ref{section: higher-dim} apply the above theories to the local Calabi--Yau manifolds of type $\widetilde{A}$, 
and also discuss mirror symmetry for varieties of general type. 
At the first reading, the reader can start with Sections \ref{section: Calabi--Yau}-\ref{section: higher-dim} for concrete examples and come back to the first few sections for the mathematical foundation.

%%%%%%%%%%%%%%%%%%%%%%%%%%%%%%%%%%%%%%%%%%%%%%%%%%%%%%%%%%%%%%%%%%%%%%%%%%%%%%%%%%%%%%%%%%%%%%%%%%%%%%%

\subsection*{Acknowledgements} 
The authors are grateful to Naichung Conan Leung and Shing-Tung Yau for useful discussions and encouragement. 
The first author benefited from many conversations with Amer Iqbal, Charles Doran, Shinobu Hosono and Yuecheng Zhu.  
The second author is very thankful to Mark Gross and Eric Zaslow for useful discussions and explaining the toric description of $\widetilde{A}_0$ dated back to 2011.
He also thanks Helge Ruddat for explaining his work on mirrors of general-type varieties using toric Calabi--Yau geometries and discussing the $3$-dimensional honeycomb tiling also in 2011. 
The first author was supported by the Harvard CMSA and the second author is supported by Boston University.

%%%%%%%%%%%%%%%%%%%%%%%%%%%%%%%%%%%%%%%%%%%%%%%%%%%%%%%%%%%%%%%%%%%%%%%%%%%%%%%%%%%%%%%%%%%%%%%%%%%%%%%
%%%%%%%%%%%%%%%%%%%%%%%%%%%%%%%%%%%%%%%%%%%%%%%%%%%%%%%%%%%%%%%%%%%%%%%%%%%%%%%%%%%%%%%%%%%%%%%%%%%%%%%

\section{Toric Calabi--Yau manifolds of infinite-type and symmetries} \label{section: toric infinite-type}
%Toric Calabi--Yau manifolds are toric manifolds with trivial canonical bundle. 
%They serve as important local building blocks of compact Calabi--Yau manifolds.  
%%Moreover, the quantum cohomology of a Calabi--Yau hypersurface $Z$ in a compact toric manifold $X$ has a close relation with the quantum cohomology of the total space $K_X$ of the canonical line bundle, which is a toric Calabi--Yau manifold. %, see for instance \cite{***}.  
%The SYZ mirrors of toric Calabi--Yau manifolds were constructed in \cite{CLL} which involve highly nontrivial quantum corrections.  
%Explicit expressions of the quantum corrections in terms of the mirror maps were derived in \cite{CCLT} in full generality. 
%%The readers are referred to \cite{Ful,Gui,CLS} for details about toric geometries and their K\"ahler structures.  
We shall make the Calabi-Yau geometry of type $\tilde{A}$ explicit by constructing it as a quotient of toric Calabi--Yau manifolds of infinite-type.
 Toric Calabi--Yau geometries discussed in most existing literature are of finite-type, namely they have only finitely many toric coordinate charts. 
Indeed the SYZ construction in \cite{CLL,CCLT} has natural generalizations to the toric Calabi--Yau manifolds of infinite-type (Definition \ref{def: infinite-type}).  

In this section we build the foundations for the toric manifolds of infinite-type, 
preparing for the next section where we extend the techniques in \cite{CLL,CCLT} to construct the SYZ mirrors of infinite-type and their quotients by toric symmetries.  Some parts of this section may appear to be technical.  The readers may want to skim over this section and quickly move to Section \ref{section: Calabi--Yau} for the local CY geometry of type $\tilde{A}$.

There are several subtle points compared with toric geometries of finite-type.  
First, there is a convergence issue for defining a toric K\"ahler metric. 
Indeed the toric K\"ahler metric is well-defined only in a neighborhood of the toric divisors. 
Thus more precisely the SYZ mirror is constructed for this neighborhood rather than the whole space. 
Second, the K\"ahler moduli space is of $\infty$-dimensions since there are infinitely many linearly independent toric curve classes. 
Third, as we shall see in Section \ref{section: SYZ mirror}, there are infinitely many terms in the equation of the SYZ mirror, since there are infinitely many toric divisors. 
We thus need to carefully make sense of the SYZ mirrors by working over rings of formal Laurent series.

%%%%%%%%%%%%%%%%%%%%%%%%%%%%%%%%%%%%%%%%%%%%%%%%%%%%%%%%%%%%%%%%%%%%%%%%%%%%%%%%%%%%%%%%%%%%%%%%%%%%%%%

\subsection{Toric manifolds of infinite-type and K\"ahler metrics}
Let $M$ and $N$ be dual lattices, and $N_{\R}:=N \otimes_\Z \R$ and similar for $M_\R$.  Let $n=\dim N_\R \geq 2$. 
%For a fan $\Sigma \subset N_\R$, we denote by $\Sigma(k)$ the set of $k$-dimensional cones of $\Sigma \subset N_\R$. 
We denote by $X_\Sigma$ the toric variety associated to a rational fan $\Sigma \subset N_\R$. 
We denote by $\Sigma(1)$ the set of primitive generator of rays in $\Sigma$. 
Throughout the article we assume the following.

\begin{Assum}
The support $|\Sigma| \subset N_\R$ is convex.  Also every cone of $\Sigma$ is contained in a maximal cone.
The toric manifold $X_\Sigma$ is smooth of dimension $n$, which is equivalent to say that each maximal cone of $\Sigma$ is generated by a basis of $N$.  
%For the purpose of defining a K\"ahler metric, we will assume that $\Sigma$ has an exhaustion by finite-type closed convex fans in the sense of Definition \ref{def:exhaustion}.
%Also without loss of generality we assume that $\Sigma$ is simple, that is $\Sigma \not= \Sigma_1 \times \Sigma_2$ for any two non-trivial fans $\Sigma_1,\Sigma_2$. 
\end{Assum}

\begin{Def}[Toric manifold of infinite-type] \label{def: infinite-type}
A toric manifold $X_\Sigma$ is said to be of finite-type if its fan $\Sigma \subset N_\R$ consists of finitely many cones.  Otherwise it is said to be of infinite-type.  
%We will always denote $\dim X$ by $n$, the lattice which supports $\Sigma$ by $N$, the generator of a ray by $v$, and a toric prime divisor by $D$.
\end{Def}

%Let $\Sigma(1)$ be the set rays (1-cones) of $\Sigma$, and $\{v_k\}_k$ be the set of primitive generators of the rays. 
%It is clear that $X_\Sigma$ is of infinite-type if and only if $\Sigma$ has infinitely many rays.

\begin{Ex}
Take $N=\Z^3$.  Let $\Sigma$ be the fan whose maximal cones are given by, for $n \in \Z$, 
$$
\R_{\geq 0}\langle (0,0,1), (n,1,1), (n+1,1,1) \rangle.
$$  
Then $X_\Sigma$ is smooth and the support $|\Sigma|$ is given by $\R_{\geq 0} \left(\{(x,1,1):x\in\R\} \cup (0,0,1) \right)$, which is convex (while $|\Sigma| - \{0\}$ is not open).
\end{Ex}

\begin{Ex}
Take $N=\Z^3$.  Let $\Sigma$ be the fan whose maximal cones are given, for $n \in \Z_{\geq 0}$, 
$$
\R_{\geq 0}\langle (-1,0,1), (0,-1,n), (0,-1,n+1) \rangle, \ \ \ \R_{\geq 0}\langle (1,0,1), (0,-1,n), (0,-1,n+1) \rangle. 
$$ 
Then $X_\Sigma$ is smooth and the support $|\Sigma|$ is given by  
$$
|\Sigma| = \R_{\geq 0} \langle (-1,0,1), (0,-1,1), (1,0,1)\rangle - \R_{> 0} \langle (-1,0,1), (1,0,1) \rangle,
$$
which is not convex.
\end{Ex}

\begin{Ex}
Take $N=\Z^2$. Let $\Sigma$ the fan such that $|\Sigma|=N_\R$ and $\Sigma(1)$ consists of 
$(0,1)$, $(0,-1)$, $(1,n)$, $(-1,m)$ for $m,n \in \Z$. 
Then the rays $\R_{> 0}(0,1)$ and $\R_{> 0}(0,-1)$ are not contained in any maximal cones. 
\end{Ex}

We make the following useful observations.

\begin{Lem} \label{lem:finite}
For a finite subset $\{v_i\}_{i=1}^k \subset \Sigma(1)$, there are only finitely many rays (and hence finitely many cones) of $\Sigma$ 
contained in the cone $\R_{\geq 0}\langle v_1,\ldots,v_k \rangle \subset N_\R$.
\end{Lem}

\begin{proof}
Suppose not. The rays are in one-to-one correspondence with points in the unit sphere $\bS \subset N_\R$ (with respect to an arbitrary metric). 
Then there are infinitely many points $\{p_i\}_{i=1}^\infty$ (corresponding to rays of $\Sigma$) contained in the compact region $\R_{\geq 0}\langle v_1,\ldots,v_k \rangle \cap \bS$. 
Thus there exists a subsequence $\{p_{i_j}\}$ converging to $p_0 \in \R_{\geq 0}\langle v_1,\ldots,v_k \rangle \cap \bS$.  
Since $\Sigma$ is convex, $p_0$ is contained in the support $|\Sigma|$.  In particular $p_0$ is contained in a certain cone $c$ of $\Sigma$.  Consider the union of all maximal cones containing $c$. 
By the assumption that every cone is in a maximal cone, for $j$ big enough, $p_{i_j}$ falls in the (relative) interior of one of these cones with $\dim > 1$. 
This is a contradiction. 
%It contradicts that $p_i$ spans a ray in $\Sigma$ which can only lie in the boundary of any cone with $\dim > 1$.  
\end{proof}

\begin{Lem} \label{lem:adj}
Suppose that $|\Sigma| - \{0\}$ is open.  Each ray of $\Sigma$ is adjacent to finitely many rays. 
Two rays are said to be adjacent if they are the boundaries of a common $2$-cone.  
\end{Lem}
\begin{proof}
Since $|\Sigma|-\{0\}$ is open, for every ray $l$ there exists a codimension $1$ ball $B$ transverse to $l$ with $l \subset \R_{\geq 0} \cdot B \subset |\Sigma|$.  For $B$ small enough it does not intersect with any rays other than $l$.  (Otherwise $l$ would be the limit of a sequence of rays not equal to $l$, which cannot be the case since a ray cannot lie in the relative interior of any cone other than itself.)  Suppose there are infinitely many distinct rays adjacent to $l$. 
The intersections of the corresponding $2$-cones with $\R_{\geq 0} \cdot B$ give infinitely many distinct points in the sphere $\partial B\cong S^{n-2}$.  
Then there exists a limit point $p$ in $\partial B$.  Since $|\Sigma|$ is convex, $\R_{\geq 0} (l \cup \{p\}) \subset \Sigma$.  Hence $\R_{\geq 0} (l \cup \{p\})$ lies in a certain cone of $\Sigma$. 
Then the sequence of $2$-cones limiting to $\R_{\geq 0} (l \cup \{p\})$ eventually falls into the relative interior of a certain cone, which cannot be the case.  
\end{proof}

%In particular we have the following.
\begin{Prop}
A toric manifold of infinite-type is non-compact.
\end{Prop}
\begin{proof}
Suppose the toric manifold is compact, which is equivalent to the condition that the fan is complete.  Let $v_1,\ldots,v_n$ be generators of the fan which is also a basis of $N_\R$.  Since the fan is complete, there exists a generator $v_{n+1}$ such that $-v_{n+1} \in \R_{>0} \langle v_1,\ldots,v_n\rangle$.  Then $\R_{\geq 0} \langle v_1,\ldots,v_{n+1} \rangle = N_\R$, and so $N_\R$ is the union of the cones generated by $v_{i_1},\ldots,v_{i_n}$, where $i_k \in \{1,\ldots,n+1\}$ are pairwise distinct.  By Lemma \ref{lem:finite} each of these cones only contains finitely many cones in $\Sigma$.  Thus $\Sigma$ only has finitely many cones.
\end{proof}

\begin{Prop}
There exists a set of real numbers $\{c_v \}_{v \in \Sigma(1)}$ such that there is an injective map from the set of cones of $\Sigma$ to the set of faces of
$$
P := \bigcap_{v\in \Sigma(1)} \{ y \in M_\R \ | \ l_v(y) := (v,y) - c_v \geq 0 \}. 
$$
Here a face of $P$ is defined to be a non-empty (closed) subset of $P$ cut out by a finitely many affine hyperplanes in $M_\R$. 
(A face can be $P$ itself, which corresponds to the $0$-cone of $\Sigma$.)  Such a $P$ is called a dual of $\Sigma$.
\end{Prop}

\begin{proof}
Let $\Sigma(1)=\{v_i\}_{i \in \Z_{>0}}$. First consider the cone $\R_{\geq 0} \langle v_1, v_2, v_3 \rangle$. 
By Lemma \ref{lem:finite}, there are only finitely many top-dimensional cones of $\Sigma$ contained in $\R_{\geq 0} \langle v_1, v_2, v_3 \rangle$.  
They combine to give a finite-type sub-fan $\Sigma'$ (which may not be convex), 
and so we can choose $c_v \in \R$, where $v$ are primitive generators of $\Sigma'$, 
such that there is an injective map from the set of cones of $\Sigma'$ to the set of faces of the polytope $P' := \{l_v \geq 0: v \in \Sigma'(1)\}$. 

Now consider the cone $\R_{\geq 0} \langle v_1, v_2, v_3, v_4 \rangle$, which again contains only finitely many top-dimensional cones of $\Sigma$. They combine to give a finite-type sub-fan $\Sigma''$ and $\Sigma'$ above is a subfan of $\Sigma''$. 
We can choose $c_v \in \R$ for $v$ being primitive generators of $\Sigma''$ but not of $\Sigma'$, 
such that there is an injective map from the set of cones of $\Sigma''$ to the set of faces of the polytope $P''$, 
which extends the above injective map.  Inductively all the $\{c_v\}_{v\in \Sigma(1)}$ are fixed.
It is easy to see that the map from cones of $\Sigma$ to faces of $P$ is injective.  
%By definition, a face of $P$ is of the form $\bigcap_{v \in S} \{l_v = 0\} \subset P$, and by the injectivity $S$ is unique. Thus it defines the inverse map from faces of $P$ to cones of $\Sigma$, and hence the map is a bijection.
\end{proof}

In the above choice of $c_v$ it may happen that the correspondence between the cones of $\Sigma$ and the faces of $P$ is not bijective. 
The following exhaustion condition helps to derive good properties of $P$, which are important for constructing K\"ahler metrics.

\begin{Def}[Exhaustion condition] \label{def:exhaustion}
An exhaustion of a fan $\Sigma$ by finite-type closed convex fans is a sequence of fans $\Sigma_1 \subset \Sigma_2 \subset \ldots$ such that $\bigcup_{i=1}^\infty \Sigma_i = \Sigma$, each $\Sigma_i$ has finitely many rays and $|\Sigma_i|$ is closed and convex.
\end{Def}

\begin{Prop} \label{prop: open covering}
Suppose that $\Sigma$ admits an exhaustion by finite-type closed convex fans $\{\Sigma_i\}$.  Then the dual $P$ can be made to satisfy the following conditions:
\begin{enumerate}
\item For every compact subset $R \subset M_\R$, there are only finitely many facets of $P$ which intersect $R$.
\item Each boundary point $p \in \partial P$ belongs to a facet of $P$, where a facet is a codimension $1$ face.
\item There exists open covering $\{U_v\}_{v \in \Sigma(1)}$ of the facets of $P$, where $U_v$ is an open neighborhood of the boundary stratum $P \cap \{l_v = 0\}$, 
such that each $p \in P$ intersects only finitely many $U_v$'s. 
\end{enumerate}
Such a $P$ is called a dual polyhedral set of $\Sigma$.
\end{Prop}

\begin{proof}
The numbers $c_v$ in the definition of $P$ are chosen by induction on the exhausting finite closed convex fans $\Sigma_i$. 
We fix an arbitrary linear metric on $M_\R$. 
First we have the dual polyhedral set $P_1$ for $\Sigma_1$, and fix one of the vertices $p_0 \in P_1$. 
Since $\Sigma_1$ is closed and convex, the faces of $P_1$ are one-to-one corresponding to the cones of $\Sigma_1$. 
Then we choose $c_v$ for generators $v$ of rays in $\Sigma_2 - \Sigma_1$ to get the dual polyhedral set $P_2$. 
Since $\Sigma_1$ is convex, we can require that the distances of the newly added facets from $p_0$ are at least $1$ by taking $c_v << 0$. 
Similarly we choose $c_v$ for generators $v$ of $\Sigma_k - \Sigma_{k-1}$ in so that the distances of the newly added facets from $p_0$ are at least $k-1$. 
Inductively we obtain the desired dual $P$.

For any point $p \in P = \bigcap_{k=1}^\infty P_k$, the distances of facets in $P_k - P_{k-1}$ tend to infinity as $k \to \infty$.
Now suppose there is a point $p \in \partial P$ which does not belong to any facet of $P$. Then $p$ does not belong to any facet of $P_k$ for any $k$.  Since $p \in \partial P$, there exists a sequence of facets of $P_k$ (where $k$ varies) whose distances with $p$ tend to $0$, which is impossible by construction.  Thus (2) is satisfied.

For each facet whose normal is $v$, take $U_v$ to be all the points in $M_\R$ whose distance with the facet is less than $\epsilon = 1/2$.  Then for every $p \in P$, since there are only finitely many facets whose distance with $p$ is less than $\epsilon$, $p$ intersects only finitely many $U_v$.
\end{proof}

Up to this point, $X = X_\Sigma$ is merely a complex manifold. 
For the purpose of SYZ construction, we need a K\"ahler structure and a Lagrangian fibration on $X$. 
In \cite{Gui}, Guillemin constructed for a toric manifold $X$ of finite-type a toric K\"ahler metric via the potential 
$$
G(y) := \frac{1}{2} \sum_{v\in \Sigma(1)} \left(l_v(y) \log l_v(y)\right)
$$
on the interior $P^\circ$, where $c_v$ are constants involved in the definition of the dual polytope $P$ (see \cite[Section 4.9]{Gui} and \cite[Section 3]{Abreu} for details).  
Then $P^\circ \cong \R^n$ by the Legendre transform sending $y \in P^\circ$ to $\partial_y G(y) \in \R^n$, 
and $G$ becomes a function $F$ on $\R^n$. 
The torus-invariant K\"ahler metric on $X$ is given by $\omega = 2 i \partial \bar{\partial} F$ 
when restricted to the open torus $(\C^\times)^n \subset X$, where $F$ is treated as a function on $(\C^\times)^n$ via pullback by $\log|\cdot|: (\C^\times)^n \to \R^n$.  
The asymptotic behavior of $G$ ensures that $\omega$ extends to a K\"ahler form on the whole space $X$. 
Furthermore the torus action gives a moment map $X \to P$ which serves as a Lagrangian fibration. 

Unfortunately the function $G$ does not make sense for a toric manifold of infinite-type since the series on the RHS does not converge. 
Instead, we define a toric K\"ahler form around a toric neighborhood of the toric divisors. 

\begin{Def}[K\"ahler potential] \label{def:metric}
Assume the exhaustion condition. 
Given an open covering $\{U_v\}$ of $P$ as in Prop \ref{prop: open covering}, 
we choose a non-negative function $\rho_v$ on $\R^n$ which is supported on $U_v$ and equals to $1$ in a smaller neighborhood of the boundary stratum $P \cap \{l_v = 0\}$.  Now define
$$
\widetilde{G}(y) := \frac{1}{2} \sum_{v \in \Sigma(1)} \rho_v(y) (l_v(y) \log l_v(y)), 
$$
which is a finite sum for each fixed $y \in P^\circ$. 
\end{Def}
Since $\widetilde{G}$ has the same asymptotic behavior as $G$ at each boundary point, %(only finitely many terms of $G$ are relevant to each boundary point and all the rest are thrown away), 
it gives the desired K\"ahler potential on a toric neighborhood $X^o$ of the union of divisors $\bigcup_{v\in \Sigma(1)} D_v \subset X$ through the Legendre transform. 
%(Note that $\widetilde{G} = 0$ on $P^\circ - \bigcup_v U_v$ and does not define a K\"ahler potential there.)
Moreover, with respect to the K\"ahler metric, we have the moment map $X^o \to P$, whose image is a neighborhood of $\partial P$ in $P$. 
From now on we always equip $X^o$ with the above K\"ahler metric. 
We define $H^*(X,\Z)$ to be the dual of $H_*(X,\Z)$, which could be of $\infty$-rank. 

\begin{Def}
We denote by $H_i(X,\Z)$ the singular homology of $X$.  
Then $H^i(X,\Z) := \mathrm{Hom}(H_i(X,\Z), \Z)$.  
Similarly $H^i(X,T) := \mathrm{Hom}(H_i(X,T),\Z)$ is defined as the dual of the relative homology, where $T$ is a moment map fiber of $X$.
\end{Def}

We have the exact sequence 
$$0 \longrightarrow H_2(X,\C) \longrightarrow H_2(X,T)_\C \longrightarrow H_1(T,\C) \longrightarrow 0$$
and its dual 
$$ 0 \longrightarrow H^1(T,\C) \longrightarrow H^2(X,T)_\C \longrightarrow H^2(X,\C) \longrightarrow 0. $$

In the context of toric manifolds of infinite-type, the vector spaces $H_2(X,\C)$, $H_2(X,T)_\C$ and their duals are $\infty$-dimensional. 
In fact, we have $H_2(X,T)_\C = \bigoplus_i \C \cdot \beta_i$ where $\beta_i$ are the basic disc classes (the readers are referred to \cite{CO} for basic holomorphic discs in a toric manifold).  
It consists of finite linear combinations of basic disc classes. 
On the other hand, the dual $H^2(X,T)_\C = \prod_i \C \cdot D_i$ consists of formal infinite linear combinations of the toric prime divisors $D_i$.
The basis $\{\beta_i\}$ and $\{D_i\}$ are dual to each other. 
By a toric divisor we mean an element of $H^2(X,T)$. 
For instance, the toric anti-canonical divisor of $X$ is $K_X = \sum_i D_i$, which has infinitely many terms if $X$ is of infinite-type. 

Next we shall define the K\"ahler moduli space of a toric manifold (of infinite-type). 

\begin{Prop} \label{prop:H_2^eff}
$H_2(X,\Z)$ is spanned by the toric curve classes, which are represented by rational curves given by the toric strata of $X$.  
\end{Prop}

\begin{proof}
It is well-known that the statement holds when $\Sigma$ is complete (in which case $X_\Sigma$ is of finite-type). 
Now consider the case when $\Sigma$ is not complete.
Since $|\Sigma|$ is convex, it is contained in a closed half space in $\R^n$ (or otherwise it is complete).  Consider a ray which is contained in the open half space.  By taking limit of its corresponding $\C^\times$-action, any rational curve in $X$ can be moved to the union of toric divisors $\bigcup_i D_i$.

Thus it suffices to consider a curve contained in a toric divisor $D$. 
The toric fan of $D$ is obtained by taking the quotient of all the cones containing $v$ along the direction of $v$, where $v$ is the generator of $\Sigma$ corresponding to $D$. 
Thus the support of the toric fan of $D$ is still convex. 
Then we can run the above argument again for the toric variety $D$.
Inductively, we must end up with a complete fan in which the statement holds.
\end{proof}

Given a basis $\{\alpha_l: l \in \Z_{>0}\}$ of $H_2(X,\Z)$, an element of $H_2(X,\Z)$ is a finite linear combination of $\alpha_l$.  On the other hand, $H^2(X,\Z)$ consists of infinite linear combination of $T_l$ where $\{T_l: l \in \Z_{>0}\}$ is the dual basis.

Let $\{C_i: i \geq 0\}$ be the set of all irreducible toric rational curves which correspond to $(n-1)$-dimensional cones in $\Sigma$ that do not lie in the boundary of the support $|\Sigma|$. 
We define a formal variable $q_i$ corresponding to each $C_i$, which can be interpreted as the exponential of the complexified symplectic area of $C_i$.  
Then $q_i = \exp 2\pi \bi (C_i,\cdot)$ defines a function on $H^2(X,\C)/H^2(X,\Z)$.
By Proposition \ref{prop:H_2^eff} we have $H_2(X,\Z) = \mathrm{Span}(\{C_i:i\geq 0\}) / R$ where $R$ is spanned by the linear relations among $\{C_i: i \geq 0\}$.  
%For each linear relation $\sum_{i=1}^k a_i C_i = 0$ in $H_2(X,\Z)$ we define a multiplicative relation $\prod_{a_i > 0} q_i^{a_i} -  \prod_{a_i < 0} q_i^{a_i} = 0$.

\begin{Def}[K\"ahler moduli space] \label{def:K-mod}
We define the multiplicative relation
$$
\prod_{a_i > 0} q_i^{a_i} \sim \prod_{a_j < 0} q_j^{a_j}
$$
associated to each linear relation $\sum_{i=1}^k a_i C_i \in R$. 
Let $\C[[q_1,\ldots]]^f \subset \C[[q_1,\ldots]]$ be the subring consisting of formal series having finitely many terms in each equivalent class under the above multiplicative relations. 
Define $I \subset \C[[q_1,\ldots]]^f$ to be the ideal generated by all the multiplicative relations. 
The K\"ahler moduli space of $X$ is defined to be $\mathrm{Spec} (\C[[q_1,\ldots]]^f/I)$. 
The variables $\{q_i\}_i$ are called the K\"ahler parameters of $X$. 
\end{Def}

The K\"ahler moduli space can be interpreted as a formal neighborhood of the large volume limit in a certain partial compactification of $H^2(X,\C)/H^2(X,\Z)$.  When there are finitely many K\"ahler parameters, $\C[[q_1,\ldots]]^f$ coincides with $\C[[q_1,\ldots]]$. 
For toric manifolds of infinite-type there might be infinitely many toric rational curves in a given homology class.  We shall restrict to $\C[[q_1,\ldots]]^f/I$ in order to talk about convergence of formal series. 

\begin{Lem} \label{lem:neg-int}
Let $X$ be a toric manifold of infinite-type with $|\Sigma| - \{0\}$ entirely contained in an open half space of $N_\R$. 
For any non-zero class $\alpha \in H_2(X,\Z)$, there exists a toric divisor with $D \cdot \alpha < 0$.  In particular, any non-constant holomorphic sphere in $X$ is contained in a toric divisor. 
\end{Lem}

\begin{proof}
Assume not.  Then $\alpha = \sum_{j} a_j \beta_j$ for a finite collection of basic disc classes $\beta_j$ and $a_j \in \Z_{>0}$.  Thus $\sum_{j} a_j v_j = 0$ where $v_j$ are the corresponding primitive generators.  The cone generated by all the $v_j$ forms a vector space, and hence is contained in the boundary of the half space.  But this contradicts that $|\Sigma| - \{0\}$ is entirely contained in the open half space.

For a non-constant holomorphic sphere, it has negative intersection with a toric divisor implies that it is contained in that divisor.
\end{proof}

In particular, if $|\Sigma| - \{0\}$ is open, the condition in the above lemma is satisfied.  It ensures that any curve cannot escape to infinity as shown in the proposition below.  It is important for having a well-defined Gromov--Witten theory.

\begin{Prop} \label{prop:cpt}
Let $X$ be a toric manifold of infinite-type with $|\Sigma| - \{0\}$ being open. 
Let $\alpha$ be a rational curve class in $X$.  There exists a compact subset $S \subset X$ such that any rational curves in $\alpha$ is contained in $S$.
\end{Prop}

\begin{proof}
By Lemma \ref{lem:neg-int}, any rational curve is contained in the toric divisors, and in particular they are contained in the neighborhood $X^o$ where K\"ahler metric is defined.
Moreover $\alpha$ has a negative intersection number with a certain toric divisor $D$, and hence any rational curve in $\alpha$ has a non-constant sphere component contained in $D$.  Since $|\Sigma| - \{0\}$ is open, the primitive generator corresponding to $D$ lies in the interior and so $D$ is compact.  Since the symplectic area of any rational curve in $\alpha$ is a fixed constant, it is contained a certain fixed compact set containing $D$.
\end{proof}

\begin{Cor} \label{cor:fin}
Let $X$ be a toric manifold of infinite-type with $|\Sigma| - \{0\}$ being open.  There are only finitely many toric irreducible curves in the same class.  Thus $\C[[q_1,\ldots]]^f = \C[[q_1,\ldots]]$.
\end{Cor}

\begin{proof}
By Proposition \ref{prop:cpt}, any rational curves in the same class are contained in a compact subset of $X$, which only contains finitely many toric irreducible curves.
\end{proof}

Given an inclusion $X \subset X'$ of toric manifolds via a toric morphism, 
then the collection of K\"ahler parameters $\{q_i\}_i$ of $X$ is a subset of the collection of K\"ahler parameters $\{q_i\}_i\cup \{q_j'\}_j$ of $X'$.  

\begin{Lem} \label{lem:rel}
A linear relation for curve classes in $X$ is also a linear relation for toric curve classes in $X'$.  Conversely, if a linear relation for $X'$ only involves toric curve classes in $X$, then it is also a linear relation for $X$.
\end{Lem}
\begin{proof}
The first statement is obvious.  Now suppose $\sum_{i=1}^k a_i C_i$ is a linear relation in $X'$, where $C_i$ are toric curve classes in $X$.  From the exact sequence
$$ 0 \longrightarrow H_2(X,\Z) \longrightarrow H_2(X,T) \longrightarrow H_1(T,\Z) \longrightarrow 0 $$
every class in $H_2(X,\Z)$ is determined by the set of its intersection numbers with all toric divisors of $X$.  
Now $\sum_{i=1}^k a_i C_i$ has intersection number $0$ with any toric divisor of $X'$, and in particular with any toric divisor of $X$. 
Hence the class is $0$ in $H_2(X,\Z)$.
\end{proof}

Let $I$ and $I'$ be the ideals generated by the multiplicative relations given for $X$ and $X'$ respectively (Definition \ref{def:K-mod}). 
Since $I$ is a subset of $I'$, we have a natural map $\C[[q_1,\ldots]]/I \to \C[[q_1,\ldots,q_1',\ldots]]/I'$ which gives a fibration 
$$
\mathrm{Spec} (\C[[q_1,\ldots,q_1',\ldots]]/I') \longrightarrow \mathrm{Spec} (\C[[q_1,\ldots]]/I).
$$  
By Lemma \ref{lem:rel}, a multiplicative relation in $I'$ which only involves $q_1,\ldots$ is also a relation in $I$. 
Hence we also have the map $\C[[q_1,\ldots]]^f/I \to \C[[q_1,\ldots,q_1',\ldots]]^f/I'$.

We can define a section of this fibration as follows.

\begin{Def} \label{def:restr}
Let $X \subset X'$ be toric manifolds where the inclusion is a toric morphism.
Define a linear map 
$$
\C[[q_1,\ldots,q_1',\ldots]]^f/I' \longrightarrow \C[[q_1,\ldots]]^f/I
$$
as follows. 
For a monomial in $\C[[q_1,\ldots,q_1',\ldots]]/I'$, if it is equivalent to a monomial $Q(q_1,\ldots)$ merely in $\{q_i\}$, 
then its image is defined to be $[Q] \in \C[[q_1,\ldots]]/I$; otherwise its image is $0$.  
\end{Def}

The above map is essentially the operation of setting $q_i' \mapsto 0$. 
However we need to write each term in merely $q_1,\ldots$ whenever possible before taking $q_i' \mapsto 0$. 
We denote the map by $(\cdot)|_{(q_i')=0}$.

%%%%%%%%%%%%%%%%%%%%%%%%%%%%%%%%%%%%%%%%%%%%%%%%%%%%%%%%%%%%%%%%%%%%%%%%%%%%%%%%%%%%%%%%%%%%%%%%%%%%%%%

\subsection{Toric Calabi--Yau manifolds of infinite-type} \label{sec:torCYinf}
Now we focus on a toric {\it Calabi--Yau} manifold $X=X_\Sigma$, whose anti-canonical divisor $K_X=\sum_i D_i$ is linearly equivalent to $0$.   
%In this paper we define a divisor to be a (possibly infinite) formal linear combination of irreducible varieties of codimension $1$. 
A toric Calabi--Yau manifold is necessarily non-compact\footnote{There exists a holomorphic function whose zero set is exactly the anti-canonical divisor $\sum_i D_i$.}.  
%As mentiond before the support of the fan is assumed to be convex. 
The setup can be made as follows. 

\begin{Def}[Toric Calabi--Yau manifolds] \label{def:torCY}
Let $N = N' \times \Z$ for a lattice $N'$ of rank $n-1$. 
Let $P \subset N_{\R}'$ be a lattice polyhedral set (which can be non-compact) containing $0 \in N'$, 
and fix a lattice triangulation of $P$, each of whose triangles is standard\footnote{
A standard triangle is isomorphic to the convex hull $\Conv(\{0,e_1,\ldots,e_{n-1}\})$ for a basis $\{e_i\}_{i=1}^{n-1}$ of $N'$ under an integral translation of $N'$.}. 
Coning over $P \times \{1\} \subset N$ produces an $n$-dimensional fan $\Sigma$. 
Then $X = X_\Sigma$ defines an $n$-dimensional toric Calabi--Yau manifold.
\end{Def}
For the purpose of defining a K\"ahler metric, from now on we always assume the exhaustion condition (Definition \ref{def:exhaustion}). 
We begin with the following simple observation.

\begin{Lem} \label{lem:neg-int-CY}
Let $X$ be a toric Calabi--Yau manifold. Then for any class $\alpha \in H_2(X,\Z)$, there exists a toric prime divisor $D$ with $D \cdot \alpha < 0$.
\end{Lem}
\begin{proof}
Since $X$ is Calabi--Yau, $\sum_{i} D_i = 0$ in $H^2(X,\Z)$.  Thus $\alpha \cdot \sum_{i} D_i = 0$. 
On the other hand $\alpha \cdot D_i$ are not all $0$, and hence there exists some $i$ with $\alpha \cdot D_i < 0$.

Alternatively, from definition $|\Sigma|-\{0\}$ is contained in the open half space $N'_\R \times \R_{>0}$.  Then result follows from Lemma \ref{lem:neg-int}.
\end{proof}

\begin{Rem} \label{rem:base-pt}
For later purpose we fix an identification $N' \cong \Z^{n-1}$ in such a way that for the standard basis $\{e_i\}_{i=1}^{n-1}$ of $\Z^{n-1}$, 
the cone $\R_{\geq 0}\langle (0,1),(e_1,1),\ldots,(e_{n-1},1)\rangle$ is a cone in $\Sigma$.
The choice of a splitting $N = N' \times \Z$ and an isomorphism $N' \cong \Z^{n-1}$ fixes a base point of $X$ which we use to carry out the SYZ construction. 
Namely, we take the base point to be the toric fixed point which corresponds to the maximal cone $\R_{\geq 0}(\{(0,1),(e_1,1),\ldots,(e_{n-1},1)\})$.
\end{Rem}

For a toric Calabi--Yau manifold of finite-type, a Lagrangian fibration with codimension $2$ discriminant locus was constructed by Goldstein \cite{Gol} and Gross \cite{Gro}, 
extending the construction of Harvey--Lawson \cite{HL} on $\C^3$. 
The fibration has played a crucial role in the SYZ construction for toric Calabi--Yau manifolds of finite-type \cite{CLL}. 
Using the modified toric K\"ahler metric given in Definition \ref{def:metric}, the construction can be extended naturally to a toric Calabi--Yau manifold of infinite-type. 

Let $X$ be a toric Calabi--Yau manifold and 
$X^o$ a toric neighborhood of the anti-canonical divisor $\sum_i D_i$, over which a toric K\"ahler metric and the corresponding moment map $\mu: X^o \to M'_\R \times \R$ are defined. 
Let $\mu': X^o \to M'_\R$ be the first component of $\mu$, where $M'$ is the dual lattice of $N'$.   
Let $w$ be the toric holomorphic function on $X$ corresponding to the lattice point $(0,1) \in M' \times \Z$. 

\begin{Def}[Lagrangian fibration] \label{def:Lag-fib}
We define $\pi: X^o \to M'_\R \times \R$ to be $\pi = (\mu', |w - \delta|)$, 
where $\delta$ is taken sufficiently close to $0$ so that $\{p \in X: |w(p) - \delta| < 2 \delta\} \subset X^o$. 
%Then $\{p \in X: |w(p) - \delta| < 2 \delta\}$ is a symplectic manifold with boundary $\{p \in X: |w(p) - \delta| = 2 \delta\}$. 
%Let $\overline{X^o}$ be the symplectic cut $\overline{X}_{|w-\delta|<2\delta}$, which is a symplectic manifold obtained from $|w- \delta|^{-1}([0, 2 \delta])$ but with each circle fiber of the Hamiltonian $|w - \delta|$ collapsed to a point \cite{Ler}. 
By shrinking $X^o$, we can assume
$$
X^o = \{p \in X: |w(p) - \delta| < 2 \delta\}.
$$ 
%Then $\overline{X^o}$ is a partial compactification of $X^o$.  
Then $\pi$ defines a Lagrangian fibrations of $X^o$ over $B = M'_\R \times [0, 2\delta)$. 
\end{Def}

The proof of the following proposition is almost identical to the finite-type case and is omitted. 
%The readers are referred to \cite{Gro} and \cite[Prop. 4.9]{CLL} for details.

\begin{Prop}[{Discriminant loci \cite{Gro},\cite[Prop. 4.9]{CLL}}] \label{prop:disc}
The discriminant locus of $\pi: \overline{X^o} \to \overline{B}$ consists of $3$ components, 
namely the boundary 
$$
\partial \overline{B} = M'_\R \times \{0\} \sqcup M'_\R \times \{2\delta\}, 
$$
and $\Gamma \times \{\delta\}$, 
where $\Gamma \subset M'_\R$ is the image of the union of codimension $2$ toric strata under $\mu': X^o \to M'_\R$.

Over each point $p \in \partial \overline{B}$, the fiber of $\pi$ is an $(n-1)$-dimensional torus which can be identified with $N'_\R / N'$. 
The fiber over $p \in \Gamma \times \{\delta\}$ is the total space of a torus fibration over $\U(1)$, 
whose fiber is $T^{n-1}$ at $e^{i \theta} \not= 1$, and is $T^k$ ($k \leq n-2$) at $1 \in \U(1)$ when $p$ corresponds to a point in a $k$-dimensional toric stratum but not in any $(k-1)$-dimensional stratum.
\end{Prop}

%%%%%%%%%%%%%%%%%%%%%%%%%%%%%%%%%%%%%%%%%%%%%%%%%%%%%%%%%%%%%%%%%%%%%%%%%%%%%%%%%%%%%%%%%%%%%%%%%%%%%%%

\subsection{GKZ system for toric manifolds of infinite-type}

%Closed-string mirror symmetry for the complete intersections in the toric Fano manifolds was proved by Givental \cite{Giv} and Lian--Liu--Yau \cite{LLY}.  
Mirror symmetry for toric manifolds has been extensively studied, and the mirror theorem was proved for semi-projective toric stacks \cite{CCIT}. 
In this section, we shall show many of the results naturally extend to the toric manifolds of infinite-type. 

It will be conceptually clearer to introduce another set of variables $\{y_i\}$, known as the complex parameters of the mirror, 
which is in one-to-one correspondence with the set of K\"ahler parameters $\{q_i\}$.   
The mirror complex moduli is defined as $\C[[y_1,\ldots]]^f/I$ by replacing $q_i$ by $y_i$ in Definition \ref{def:K-mod}. 
However, the identification between the mirror complex moduli and the K\"ahler moduli is given by the highly-nontrivial mirror map explained below. %rather than the very naive map $y_i \mapsto q_i$.

Recall that $\{D_i\}$ denotes the set of toric prime divisors.  
By the map $H^2(X,T) \to H^2(X,\Z)$, a toric divisor $D$ can be identified with an element in $H^2(X,\Z)$, 
which can then be regarded as a first-order differential operator $\widehat{D} $ acting on $\C[[y_1,\ldots]]^f/I$, namely 
$$
\widehat{D}  \cdot y_i = (D \cdot C_i) \, y_i.
$$
The operator $\widehat{D} $ is explicitly expressed in terms of a basis as follows. 
First, $D$ (set theoretically) intersects finitely many irreducible toric curves, say $C_j$ for $j \in J$.   
Consider the image of $\Span \{C_j\}_{j \in J} \subset H_2(X,\Q)$ and choose a basis $\{\alpha_1,\ldots,\alpha_k\}$. 
Let $y^{\alpha_1},\ldots,y^{\alpha_k}$ be the corresponding K\"ahler parameters.  
Then
$$
\widehat{D}  = \sum_{l=1}^k (D \cdot \alpha_l) \, y^{\alpha_l} \frac{\partial}{\partial y^{\alpha_l}}.
$$
The following GKZ system was introduced by Gelfand, Kapranov and Zelevinskii \cite{GKZ1,GKZ2}.  
It plays a crucial role in the study of toric mirror symmetry as Gromov--Witten invariants can be extracted from their solutions.

\begin{Def}[GKZ system] \label{def:GKZ}
For each $d \in H_2(X,\Z)$, we define the differential operator
\begin{equation}
\Box_d := \prod_{i: (D_i, d) > 0} \prod_{k=0}^{(D_i,d)-1} (\widehat{D} _i - kz) - y^d \prod_{i: (D_i, d) < 0} \prod_{k=0}^{-(D_i,d)-1} (\widehat{D} _i - kz).
\end{equation}
The GKZ system is the system of differential equations $\Box_d \cdot h = 0$ for all $d \in H_2(X,\Z)$, where $h \in \C[[y_1,\ldots]]^f / I$.

Let $\{\alpha_l\}_{l \in \Z_{>0}}$ be a basis of $H_2(X,\Q)$ and $y^{\alpha_l}$ the corresponding K\"ahler parameters. 
The GKZ module over $H^2(X,\C)/H^2(X,\Z) \times \C_z$ is define to be 
$$\C[z,y^{\pm \alpha_1},\ldots] \left\langle z \frac{\partial}{\partial \log y^{\alpha_l}}: l \in \Z_{>0} \right\rangle \Big/ \langle  \Box_d: d \in H_2(X,\Z) \rangle.$$
\end{Def}

Here the variable $z$ is not particularly important for the purpose of this article; its power records the degree of a differential operator. 
Solutions to the GKZ system are given by the coefficients of the celebrated $I$-function (c.f. \cite{Iri}). 

\begin{Def}[$I$-function]
Let $\{\alpha_l\}_{l \in \Z_{>0}}$ be a basis of $H_2(X,\Z)$ and $\{T_l\}_{l \in \Z_{>0}}$ its dual basis.
We define the following formal $H^{\mathrm{even}}(X)$-valued series
\begin{align}
\bI(z;y) &:= e^{z^{-1} \sum_{l=1}^\infty T_l \log y^{\alpha_l}} \, \bI_{\mathrm{main}}(z;y)  \notag \\
&:= e^{z^{-1} \sum_{l=1}^\infty T_l \log y^{\alpha_l}} \sum_{d \in H_2^{\mathrm{eff}}(X,\Z)} y^d \prod_{i} \frac{\prod_{m=-\infty}^0 (D_i + mz)}{\prod_{m=-\infty}^{d \cdot D_i} (D_i + mz)}.\notag
\end{align}
It is understood that the above is written in terms of the cup product of cohomology classes and series expansions of the exponential and log functions.
\end{Def}

%Note that in this setting the $I$-function can have infinitely many components since $H^{\mathrm{even}}(X)$ is of infinite dimensions.  

\begin{Lem}
The coefficient of each component of
$\bI_{\mathrm{main}}$ belongs to $\C[[y_1,\ldots]]^f/I$.
\end{Lem}

\begin{proof}
For each $d\in H_2^{\mathrm{eff}}$, note that $\frac{\prod_{m=-\infty}^0 (D_i + mz)}{\prod_{m=-\infty}^{d \cdot D_i} (D_i + mz)} = 1$ if $d \cdot D_i = 0$, which is the case for all but finitely many $D_i$.  Thus $\prod_{i} \frac{\prod_{m=-\infty}^0 (D_i + mz)}{\prod_{m=-\infty}^{d \cdot D_i} (D_i + mz)}$ is indeed a finite product, which has only finitely many terms in its expansions. 
Hence $y^d$ only appears only finitely many times in the coefficient of each cohomology class in $\bI$.
\end{proof}

The proof of the following proposition is almost identical to that for the finite-type (see for instance \cite[Lemma 4.6]{Iri}).

\begin{Prop}
We have $\Box_d \cdot \bI(z;y) = 0 $ for all $d \in H_2(X,\Z)$.  
\end{Prop}

\begin{Def}[Mirror map]
The mirror map is defined as the coefficient of $1/z$ of the $I$-function $\bI(z;y)$.
\end{Def}

The following is obtained by direct computation.
\begin{Prop} \label{prop:mir-map}
The mirror map is $H^2(X)$-valued.  The coefficient of $T_l \in H^2(X,\Z)$ is defined to be $\log q^{\alpha_l}$, which equals to
$$ \log y^{\alpha_l} - \sum_{i}(D_i \cdot \alpha_l) g_i(y) $$
where
$$ g_i(y) := \sum_{d} \frac{(-1)^{(D_i \cdot d)} (-(D_i \cdot d)-1)!}{\prod_{p\not=i}(D_p \cdot d)!} y^d, $$
the summation is taken over all $d \in H_2^{\mathrm{eff}}(X,\Z)$ satisfying the condition that 
$$ -K_X \cdot d = 0, D_i \cdot d < 0 \textrm{ and } D_p \cdot d \geq 0 \textrm{ for all } p \not= i.$$
We denote the mirror map by $q(y) = (q^{\alpha_l})_{l=1}^\infty$ whose entries are given by the above expression.
\end{Prop}

For the sake of completeness we quote the toric mirror theorem for the compact toric orbifolds.  
We shall not directly use this theorem in this article (although the proof of the open mirror theorem for toric Calabi--Yau orbifolds in \cite{CCLT} uses it).  

\begin{Thm}[Mirror theorem for toric orbifolds of finite-type \cite{CCIT}]
Let $X$ be a compact semi-Fano toric K\"ahler orbifold.
Fix a basis $\{\alpha_l\}$ of $H_2(X,\Z)$ and denote its dual basis by $\{T_l\} \subset H^2(X,\Z)$.  
Similarly fix a homogeneous basis $\{\phi_a\}$ of $H^*(X,\Z)$ and denote its dual basis by $\{\phi^a\} \subset H^*(X,\Z)$ with respect to the Poincar\'e duality.  Define
$$
J(z;q) := e^{z^{-1} \sum_{l=1}^\infty T_l \log q^{\alpha_l}} \left(1+\sum_a\sum_{d \in H_2^\mathrm{eff}(X,\Z)\setminus\{0\}} q^d \Big\langle1,\frac{\phi_a}{z-\psi}\Big\rangle_{0,2,d}\phi^a\right)
$$
where $\langle \cdots \rangle_{g,k,d}$ denotes the genus $g$, degree $d$ descendent Gromov--Witten invariant of $X$ with $k$ insertions. 
Then we have $J(z;q(y)) = \bI(z;y)$. 
\end{Thm}

\begin{comment}
Similarly it is easy to see that
\begin{Lem}
The coefficient of each component of
$\sum_a\sum_{d \in H_2^\mathrm{eff}(X,,\Z)\setminus\{0\}} q^d \Big\langle1,\frac{\phi_a}{z-\psi}\Big\rangle_{0,2,d}\phi^a$
belongs to $\C[[q_1,\ldots]]^f/I$.
\end{Lem}
\end{comment}

%%%%%%%%%%%%%%%%%%%%%%%%%%%%%%%%%%%%%%%%%%%%%%%%%%%%%%%%%%%%%%%%%%%%%%%%%%%%%%%%%%%%%%%%%%%%%%%%%%%%%%%

\subsection{Symmetries and quotients} \label{subset: free group}

In this subsection, we consider an effective free discrete group action of $G$ on a toric manifold $X=X_\Sigma$ of infinite-type by toric morphisms.  For simplicity we assume that the action has only finitely many orbits. 
%such that $X/G$ is still a smooth complex manifold (which is no longer toric in general).  
We give an explicit description of the K\"ahler structures, mirror maps and Lagrangian fibrations. 

The setting is the following.  Let $\Sigma$ be a fan with an exhaustion by finite-type closed convex fans, and assume $|\Sigma|-\{0\}$ is open. 
Consider a discrete group $G < \GL(N)$ whose action on $N_\R$ preserves $\Sigma$, mapping the $k$-cones to the $k$-cones.  
We assume that the induced action on $\Sigma - \{0\}$ is free and has only finitely many orbits.

Notice that a $G$-invariant K\"ahler metric may not exist, shown by the example below.

\begin{Ex} \label{ex:no-metric}
Consider the fan $\Sigma$ consisting of the maximal cones for $(m,n) \in \Z^2$
\begin{align*}
&\langle(2m,n,1),(2m+1,n,1),(2m,n+1,1)\rangle, \\
&\langle(2m+1,n,1),(2m,n+1,1),(2m+1,n+1,1)\rangle,\\
&\langle(2m+1,n,1),(2m+1,n+1,1),(2m+2,n+1,1)\rangle, \\
&\langle(2m+1,n,1),(2m+2,n,1),(2m+2,n+1,1)\rangle
\end{align*} 
 in $N_\R = \R^3$.  It is depicted by the right figure in Figure \ref{fig:Domino2}. 
 It admits a group action by $G=\Z^2$, where the standard basis acts on $N = \Z^3$ by $e_1 \cdot (a_1,a_2,a_3) = (a_1+2,a_2,a_3)$ and $e_2 \cdot (a_1,a_2,a_3) = (a_1,a_2+1,a_3)$. 
 \begin{figure}[htbp]
 \begin{center} 
  \includegraphics[width=100mm]{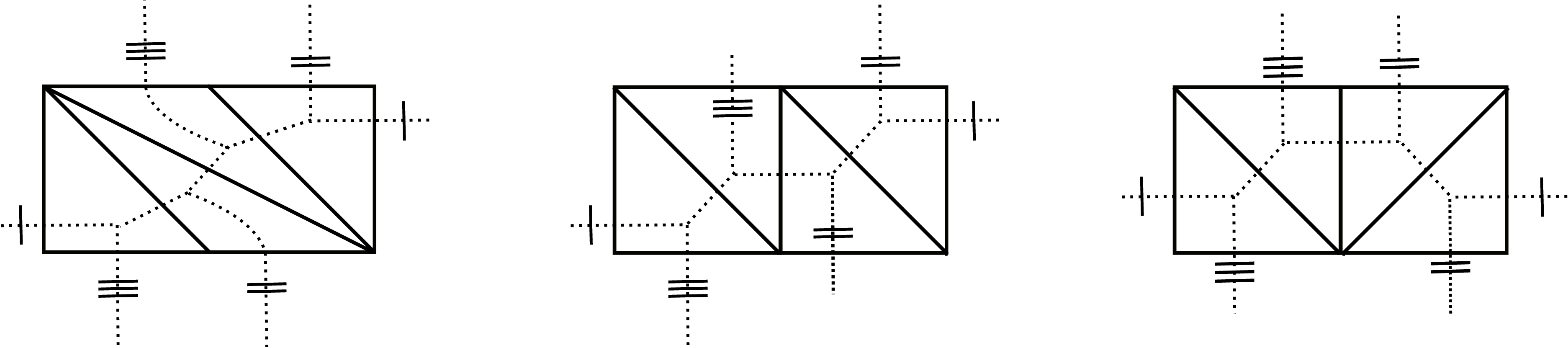}
 \end{center}
 \caption{Fundamental domain with its flops}
\label{fig:Domino2}
\end{figure}
 Then $X_\Sigma$ does not admit a $G$-invariant K\"ahler metric.  The reason is as follows.

Consider the toric invariant curves $C_1,C_2,C_3,C_4$ corresponding to the cones 
$$\langle(0,1,1),(1,0,1)\rangle,\langle(1,0,1),(1,1,1)\rangle,\langle(1,0,1),(2,1,1)\rangle,\langle(1,-1,1),(1,0,1)\rangle$$ 
respectively.  We have $[C_1]+[C_2]+[C_3]=[C_4]$ in $H_2(X_\Sigma,\Z)$.  Moreover $(-e_2) \cdot [C_2] = [C_4]$. 
Suppose there exists a $G$-invariant K\"ahler metric.  Then $C_2$ and $C_4$ have the same symplectic area since they are in the same $G$-orbit.  
This forces the symplectic area of $C_1 + C_3$ to be $0$, which contradicts to that fact that $C_1$ and $C_3$ are non-zero holomorphic curves which have positive symplectic areas.
\end{Ex}

As the above example illustrates, the $G$-action may not respect the stratification of the cone of effective classes, in the sense that the $G$-orbit closure of an interior point of the effective cone may hit the origin. 
In such a case $G$-invariant K\"ahler metric cannot exist.
As it turns out, the key obstruction to the existence is the compatibility between the $G$-action and the dual polytope of $\Sigma$. 
We consider the induced linear action of $G$ on the dual lattice $M$.  

\begin{Prop} \label{prop:G-metric}
Assume that the constants $c_v$ for $v \in \Sigma(1)$ can be chosen in such a way that the corresponding polytope 
$$
P := \bigcap_{v\in \Sigma(1)} \{ y \in M_\R \ | \ l_v(y) := (v,y) - c_v \geq 0 \}
$$
is invariant under the $G$-action on $M_\R$ up to translation, 
that is, there is an action of $G$ on $M_\R$ by affine linear transformations with the linear parts given by taking dual of the given $G$-action on $N$ under which $P$ is invariant. 
Then there exists a $G$-invariant toric K\"ahler metric on a toric neighborhood $X^o$ of the toric divisors.
\end{Prop}

\begin{proof}
%First we fix a toric K\"ahler metric on $X^o$ as in Definition \ref{def:metric}.  By Lemma \ref{lem:G-equiv}, we have a free $G$-action on the corresponding moment-map image $\mu(X^o)$. 
Without loss of generality we may assume that $P$ has no translational symmetry, namely $P + a = P$ for some $a \in M_\R$ implies $a=0$.  We define an affine linear action of $G$ on $M_\R$ by $x * g = x \cdot g + a_g$, where $x \cdot g$ is the original linear action of $G$ on $M_\R$ and $a_g\in M_\R$ is fixed by the equality $P \cdot g = P + a_g$.  Note that $a_g \not= 0$ whenever $g \not= \mathrm{id}$, and hence the affine linear action is properly discontinuous.

In each $G$-orbit of rays of $\Sigma$ we fix a representative $\R \cdot v$. 
Since $|\Sigma|-\{0\}$ is open, each ray is in the interior, and by Lemma \ref{lem:adj} it is adjacent to finitely many rays.  
As a result the facet $H_v$ is compact.
Then we fix an open neighborhood $U_v$ of $H_v$ whose closure is compact. 
 We also fix a non-negative function $\rho_v$ which is supported in $U_v$ and equals to $1$ in a smaller neighborhood of $H_v$. 
 By the action of $G$, we obtain the corresponding open neighborhoods $U_v * g$ of $H_v * g$ and support functions $(g^{-1})^* \rho_v$.  We do this for every $G$-orbit.  
Since the affine linear $G$-action is properly discontinuous, $U_v \cap (U_v * g) \not= \emptyset$ only for finitely many $g$. 
By assumption there are only finitely many orbits, and hence for each $p \in P$, there are only finitely many open sets $U_v \ni p$.  
Then as in Definition \ref{def:metric} we define
$$
\widetilde{G}(y) = \frac{1}{2} \sum_{v} \sum_{g \in G} \rho_v(y * g) \cdot (l_v(y * g) \log l_v(y * g))
$$
on $\bigcup_{v} \bigcup_{g \in G} U_v * g$, where $v$ runs over the primitive generators of the representatives of the finitely many $G$-orbits of rays.  By definition $\widetilde{G}$ is invariant.  Moreover since $P$ is invariant under the affine linear action, its defining linear functions $l_v(y)$ are $G$-equivariant.  Since $\rho_v = 1$ in a neighborhood of $H_v$,
The above has the correct boundary behavior and defines the toric K\"ahler potential by Legendre transform.
\end{proof}

By Proposition \ref{prop:G-metric}, we observe that an open neighborhood $X^o$ of the toric divisors in the toric Calabi--Yau manifold of infinite type 
associated to the left or central figures of Figure \ref{fig:Domino2} 
admits a $G$-invariant toric K\"ahler metric (c.f. Section \ref{section: 3-fold}).  
 
Recall that the K\"ahler moduli of $X$ is $\C[[q_1,\ldots]]^f/I$ (Definition \ref{def:K-mod}), where $q_i$ is a formal K\"ahler parameter associated to each toric irreducible curve $C_i$.  
Since the K\"ahler structure is invariant under $G$, and $q_i$ measures the K\"ahler size of the curve $C_i$, the variable $q_i$ should be invariant under the $G$-action. 
This motivates the following definition of the K\"ahler moduli for the quotient $X^o / G$. 

\begin{Def}[K\"ahler moduli for the quotient] \label{def:K-mod-quot}
The K\"ahler moduli for the quotient $X^o / G$ is defined to be (the Spectrum of) the quotient ring $(\C[[q_1,\ldots]]^f/I)/G$, 
where the $G$-quotient means the quotient by the ideal generated by $g \cdot q_i - q_j$ for all $C_j \cdot g = C_i$. 
%(Note that $g \in G$ maps an irreducible toric curve to another irreducible toric curve.) 
\end{Def}

The K\"ahler moduli defined above can be regarded as a formal neighborhood of a limit point of 
$$
H^2(X^o/G,\C)/H^2(X^o/G,\Z) = \mathrm{Spec}[q^{\pm \alpha_1},\ldots] / G,
$$
where $\{\alpha_i\}$ is a basis of $H_2(X^o/G,\Q)$, and $g \cdot q^{\alpha_i} = q^{\alpha_i \cdot g^{-1}}$. 
Similarly the mirror complex moduli is defined as $(\C[[y_1,\ldots]]^f/I)/G$ by replacing every formal variable $q_i$ by $y_i$. 
Recall from Corollary \ref{cor:fin} that the superscript $f$ (which stands for finitely many terms in the same class) can be dropped if we assume $|\Sigma| - \{0\}$ is open.

Recall that we have the GKZ system defined by the differential operators $\Box_d$ in Definition \ref{def:GKZ}. 
We have an action of $G$ on the ring of differential operators 
$$
\C[z,y^{\pm \alpha_1},\ldots] \left\langle z \frac{\partial}{\partial \log y^{\alpha_l}}: l \in \Z_{>0} \right\rangle
$$
given by
$$
g \cdot \frac{\partial}{\partial \log y^{\alpha_l}} = \frac{\partial}{\partial \log y^{\alpha_l \cdot g^{-1}}}.
$$

\begin{Prop}
If $h \in \C[[q_1,\ldots]]^f/I$ satisfies $\Box_d \cdot h = 0$ for all $d \in H_2(X,\Z)$, then so does $g \cdot h$ for each $g \in G$. 
Thus the GKZ module in Definition \ref{def:GKZ} is preserved under $G$.
\end{Prop}

\begin{proof}
The $G$-action on the differential operator $\widehat{D} $ is given by $g \cdot \widehat{D}  = \widehat{D \cdot g^{-1}}$ for each toric divisor $D$.  Then for $C_j = C_i \cdot g$, we have
$$ (g \cdot \widehat{D} ) \cdot (g \cdot y_j) =  (\widehat{D \cdot g^{-1}}) \cdot y_i = ((D \cdot g^{-1}) \cdot C_i) \cdot y_i = (D \cdot C_j) \cdot (g \cdot y_j) = g \cdot (\widehat{D}  \cdot y_j). $$
Hence $(g \cdot \widehat{D} )(g \cdot h) = g \cdot (\widehat{D}  \cdot h)$ for all $h \in \C[[q_1,\ldots]]^f/I$.  Recall that
$$\Box_d = \prod_{i: (D_i, d) > 0} \prod_{k=0}^{(D_i,d)-1} (\widehat{D} _i - kz) - y^d \prod_{i: (D_i, d) < 0} \prod_{k=0}^{-(D_i,d)-1} (\widehat{D} _i - kz).$$
Thus $(g \cdot \Box_d) \cdot (g \cdot h) = g \cdot (\Box_d \cdot h)$.
On the other hand
\begin{align*}
g \cdot \Box_d &= \prod_{i: (D_i, d) > 0} \prod_{k=0}^{(D_i,d)-1} (\widehat{D_i \cdot g^{-1}} - kz) - (g \cdot y^d) \prod_{i: (D_i, d) < 0} \prod_{k=0}^{-(D_i,d)-1} (\widehat{D_i \cdot g^{-1}} - kz) \\
&= \prod_{i: (D_i \cdot g, d) > 0} \prod_{k=0}^{(D_i \cdot g,d)-1} (\widehat{D} _i - kz) - y^{d \cdot g^{-1}} \prod_{i: (D_i \cdot g, d) < 0} \prod_{k=0}^{-(D_i \cdot g,d)-1} (\widehat{D} _i - kz)\\
&= \prod_{i: (D_i, d \cdot g^{-1}) > 0} \prod_{k=0}^{(D_i,d \cdot g^{-1})-1} (\widehat{D} _i - kz) - y^{d \cdot g^{-1}} \prod_{i: (D_i, d \cdot g^{-1}) < 0} \prod_{k=0}^{-(D_i,d \cdot g^{-1})-1} (\widehat{D} _i - kz)\\
&= \Box_{d \cdot g^{-1}}.
\end{align*}

Thus $\Box_d \cdot h = 0$ if and only if $\Box_{d \cdot g^{-1}} \cdot (g \cdot h) = 0$. 
Since $d \cdot g^{-1}$ runs over the whole $H_2(X,\Z)$ as $d$ runs over the whole $H_2(X,\Z)$, it follows that $g \cdot h$ satisfies the same system of differential equations.
Also it follows from $g \cdot \Box_d = \Box_{d \cdot g^{-1}}$ that the GKZ ideal $\langle  \Box_d: d \in H_2(X,\Z) \rangle$ is preserved under $G$.
\end{proof}

Next we consider the $G$-action on the $I$-function.  For this purpose, fix a homogeneous basis $\{T_{l,p}\}$ of $H^{\mathrm{even}}(X,\Z)$, where $p$ records the cohomological degree and $T_{l,2} = T_l$.  An $H^{\mathrm{even}}(X)$-valued series is the formal sum $\sum_{p,l} h_{l,p} \, T_{l,p}$, 
where each $h_{l,p} \in \C[[y_1,\ldots]]^f/I$.  Define the $G$-action by 
$$g \cdot \sum_{p,l} h_{l,p} \, T_{l,p} := \sum_{p,l} (g \cdot h_{l,p}) \, (g \cdot T_{l,p}).$$
It is easy to see that the definition is independent of the choice of a basis.  Note that $C \cdot (g \cdot T_{l,p}) = (C \cdot g) \cdot T_{l,p}$ for $C \in H_*(X,\Z)$.

\begin{Prop}
We have $ g \cdot \bI_{\mathrm{main}} = \bI_{\mathrm{main}}$. 
\end{Prop}

\begin{proof}
The assertion follows from direct computation. 
\begin{align*}
g \cdot \sum_{d \in H_2^{\mathrm{eff}}(X,\Z)} y^d \prod_{i} \frac{\prod_{m=-\infty}^0 (D_i + mz)}{\prod_{m=-\infty}^{d \cdot D_i} (D_i + mz)} &= \sum_{d \in H_2^{\mathrm{eff}}(X,\Z)} y^{d\cdot g^{-1}} \prod_{i} \frac{\prod_{m=-\infty}^0 (g \cdot D_i + mz)}{\prod_{m=-\infty}^{d \cdot D_i} (g \cdot D_i + mz)} \\
&= \sum_{d \in H_2^{\mathrm{eff}}(X,\Z)} y^{d} \prod_{i} \frac{\prod_{m=-\infty}^0 (D_i + mz)}{\prod_{m=-\infty}^{(d \cdot g) \cdot (g^{-1} \cdot D_i)} (D_i + mz)} \\
&= \sum_{d \in H_2^{\mathrm{eff}}(X,\Z)} y^d \prod_{i} \frac{\prod_{m=-\infty}^0 (D_i + mz)}{\prod_{m=-\infty}^{d \cdot D_i} (D_i + mz)}.
\end{align*}
\end{proof}

We have the following corollary for the mirror map.

\begin{Cor} \label{cor:G-mir-map}
Recall that the mirror map for $X$ is given by $\log q^{\alpha_l}(y)$ where $\alpha_l \in H_2(X,\Z)$ is the dual basis of $T_l \in H^2(X,\Z)$.  Then
$$ g \cdot (\log q^{\alpha_l}(y)) = \log q^{\alpha_l \cdot g^{-1}}(y).$$
In particular the inverse mirror map $\log y^{\alpha_l}(q)$ has the same property:
$$ g \cdot (\log y^{\alpha_l}(q)) = \log y^{\alpha_l \cdot g^{-1}}(q).$$
\end{Cor}

With the above corollary, we can define the mirror map for $X^o/G$ as follows.

\begin{Def}
The mirror map for $X^o/G$ is
$$
(\C[[q_1,\ldots]]^f/I)/G \to (\C[[y_1,\ldots]]^f/I)/G, \ \ \ q^{\alpha_l} \mapsto q^{\alpha_l}(y)
$$
where $\log q^{\alpha_l}(y)$ is the mirror map of $X$.  By Corollary \ref{cor:G-mir-map} the map is $G$-equivariant and hence is well-defined.
\end{Def}

Now we consider a toric Calabi--Yau manifold $X$ of infinite-type.  
We require that $G$ is a subgroup of $\SL(N)$ so that the quotient is still Calabi--Yau.

\begin{Lem} \label{lem:torCY}
Let $\nu = (0,1) \in M' \times \Z = M$ and $G \subset \SL(N)$. Then $\nu$ is invariant under $G$. 
A toric holomorphic volume form of $X$ descends to the quotient $X^o/G$, and hence $X^o/G$ is Calabi--Yau.
\end{Lem}
\begin{proof}
The vector $\nu$ is characterized by the property that it maps every primitive generator to $1$. 
Since $g \in G$ preserves the fan, it maps a primitive generator to another primitive generator. 
In particular $(g^* \nu) (v) = 1$ for every primitive generator, and hence $g^* \nu = \nu$.

A toric holomorphic volume form of $X^o$ takes the form (for a fixed $c \in \C$)
$$ c \, d z_1 \wedge \ldots \wedge d z_n = c \, w \, d \log z_1 \wedge \ldots d \log z_n$$
on every toric coordinate system $(z_1,\ldots,z_n)$ corresponding to a maximal cone of $\Sigma$, where $w$ is the holomorphic function corresponding to $\nu \in M$. 
Since $g \in \SL(N)$, $g^* d \log z_1 \wedge \ldots d \log z_n = d \log z'_1 \wedge \ldots d \log z'_n$ where $(z'_1,\ldots,z'_n)$ is the coordinate system for the image maximal cone of $g$. 
Moreover $g^* w = w$.  Thus the toric holomorphic volume form is preserved.
\end{proof}

Then we define a Lagrangian torus fibration on $X^o / G$ as follows.

\begin{Prop} \label{prop:fib-quot}
Let $\mu$ be the moment map with respect to a $G$-invariant toric K\"ahler metric given in Proposition \ref{prop:G-metric}. 
Let $\mu'$ be the composition of $\mu$ with the projection to the first factor $M_\R \to M'_\R$.  Then 
$$(\mu',|w-\delta|): X^o / G \to (M'_\R / G) \times [0,2\delta)$$ 
is a Lagrangian torus fibration. 
This fibration is special with respect to $\frac{\Omega}{w - \delta}$ (descended to $X^o / G$) where $\Omega$ is a toric holomorphic volume form of $X$.
\end{Prop}

\begin{proof}
First note that $\mu'$ is the moment map for the action of the subtorus $N'_\R/N' \subset N_\R / N$, which commutes with the $G$-action since $G$ acts as toric morphisms. 
Hence $\mu'$ is $G$-equivariant and gives a map $X^o / G \to M'_\R / G$.  This is indeed a surjective map, since for every $u \in M'_\R$, there exists $c \gg 0$ such that $u + c \nu \in P$.

Then $\frac{\Omega}{w - \delta}$ defines a meromorphic $n$-form on $X^o$ which is nowhere zero and whose pole set is the divisor $\{w = \delta\} \subset X^o$.  Since both $\Omega$ and $w$ are $G$-invariant, $\frac{\Omega}{w - \delta}$ descends to the quotient $X^o/G$.  The proof that the fibration $X^o \to M'_\R \times [0,2\delta)$ is special Lagrangian is the same as in the finite-type case given by \cite{Gro} using symplectic reduction.  Since everything is $G$-equivariant, the special Lagrangian fibration descends to one on the $G$-quotient.
\end{proof}

Note that the boundary divisor $\pi^{-1}(M'_\R\times \{0\})$ is the anti-canonical divisor $K_{X^o}=\sum_{i=1}D_i$.   

%By considering a symplectic cut, we can also partially compactify the above special Lgrangian fibration (see Definition \ref{def:Lag-fib}): $$(\mu',|w-\delta|): \overline{X^o / G} \to \overline{B}/G=(M'_\R / G) \times [0,2\delta].$$ When $M'_\R / G$ is compact, the symplectic cut $ \overline{X^o / G}$ is a full compactification of $X^o / G$.  The discriminant loci of $(\mu',|w-\delta|)$ consists of three components just as described in Proposition \ref{prop:disc}, namely two boundary components $\partial \overline{B}/G$ and $\Lambda/G \times \{\delta\}$. Then the boundary divisor is the anti-canonical divisor $D=\sum_{i=1}D_i+D_\infty$ of $\overline{X}^o$. 

%%%%%%%%%%%%%%%%%%%%%%%%%%%%%%%%%%%%%%%%%%%%%%%%%%%%%%%%%%%%%%%%%%%%%%%%%%%%%%%%%%%%%%%%%%%%%%%%%%%%%%%
%%%%%%%%%%%%%%%%%%%%%%%%%%%%%%%%%%%%%%%%%%%%%%%%%%%%%%%%%%%%%%%%%%%%%%%%%%%%%%%%%%%%%%%%%%%%%%%%%%%%%%%

\section{SYZ mirrors of toric Calabi--Yau manifolds of infinite-type} \label{section: SYZ mirror}

%There are two approaches to define the SYZ mirror of an infinite-type toric CY manifold $X$.  The first approach is to construct a Lagrangian fibration over a neighborhood of the toric divisors, take T-duality and consider their open Gromov--Witten invariants.  The second approach is to take limit of the SYZ mirrors of the finite-type toric CY manifolds $X'$ contained in $X$.  These two approaches produce the same mirror, because the open Gromov--Witten invariant of each disc class $\beta$ in $X$ is the same as that of $X' \subset X$ when $X'$ is big enough (Proposition \ref{}  *** the open moduli spaces are the same, since bubbling curves in $X'$ cannot escape to infinity ***).  We will focus on the second approach, since we want to make use of the known results for finite-type toric CY manifolds \cite{CLL,CCLT}.

The Strominger--Yau--Zaslow (SYZ) conjecture \cite{SYZ} provides a foundational geometric understanding of mirror symmetry. 
It asserts that, for a mirror pair of Calabi--Yau manifolds $X$ and $Y$, 
there exist Lagrangian torus fibrations $\pi:X\to B$ and $\pi^\vee:Y\to B$ which are ideally fiberwise-dual to each other. 
It suggests a geometric construction of the mirror $Y$ by fiberwise dualizing a Lagrangian torus fibration on $X$.  
In this article, we will use the formulation of an SYZ mirror (with quantum corrections) given in \cite[Section 2]{CLL}.
%The motivation comes from T-duality in string theory.  

The SYZ mirror of a toric Calabi--Yau manifold of finite-type was constructed in \cite{CLL}. 
In this section we consider the SYZ mirror of a toric Calabi--Yau manifold of infinite-type.  The construction is similar and so we shall be brief. 
More precisely we construct the SYZ mirror of a neighborhood of an anti-canonical toric divisor in a toric Calabi--Yau manifolds of infinite-type, 
which contains the essential information of open Gromov--Witten theory.

Technically the SYZ mirror of a toric manifold of infinite-type involves infinitely many K\"ahler parameters, 
and also the mirror Laurent series in $z$ contains infinitely many terms. 
We shall need the topological ring of formal series $\C[u,v][[z_1^{\pm 1}, \ldots, z_{n-1}^{\pm 1}]][[q_1,\ldots]]^f/I$ where the mirror Laurent series lives in.

As in Definition \ref{def:Lag-fib}, we shall always assume that the corresponding lattice polyhedral set $P$ of the toric Calabi--Yau manifold admits an exhaustion by compact lattice polytopes, so that a K\"ahler metric and a Lagrangian fibration are defined (Definition \ref{def:Lag-fib}).  We show that the SYZ mirror formulated in \cite[Section 2]{CLL} equals to the following.

\begin{Thm} \label{thm:SYZ}
The SYZ mirror of a toric Calabi--Yau manifold $X$ of infinite-type is 
$$\check{X} = \mathrm{Spec} \big((\C[u,v][[z_1^{\pm 1}, \ldots, z_{n-1}^{\pm 1}]][[q_1,\ldots]]^f/I) / (uv - F^\open(q_1,\ldots; z_1,\ldots,z_{n-1}))\big)$$
where $F^\open$ is given below as a formal series in $\C[[z_1^{\pm 1}, \ldots, z_{n-1}^{\pm 1}]][[q_1,\ldots]]^f/I$:
\begin{equation} \label{eq:F^open}
F^\open = \sum_{v} \left( \sum_{\alpha \in H_2^\mathrm{eff} (X,\Z)} n_{\beta_v + \alpha} q^\alpha \right) q^{A_v} z^v.
\end{equation}
In the above expression, the sum is over all $v \in N'$ where $(v,1)=(a_1,\ldots,a_{n-1},1)$ are generators of the fan.  $\beta_v$ are the basic disc classes corresponding to the rays generated by $(v,1)$.
$n_{\beta_{v} + \alpha}$ is the open Gromov--Witten invariant of the disc class $\beta_{v} + \alpha$ of a regular moment-map fiber.  $z^v = z_1^{a_1} \ldots z_{n-1}^{a_{n-1}}$ where $a_i$ are the coefficients in the expression $\partial (\beta_v-\beta_0) = \sum_{k=1}^{n-1} a_k \cdot \partial(\beta_{e_k} -  \beta_{0})$.  $A_v$ is the curve class
$$
A_v := \beta_{v} - \beta_{0} - \sum_{k=1}^{n-1} a_k \cdot (\beta_{e_k} -  \beta_{0}).
$$
$q^C = q_{1}^{k_1} \ldots q_{p}^{k_p}$ for an effective curve class $C = \sum_{j=1}^p k_j C_j$ for $k_j \geq 0$.  ($\{C_j\}$ is the set of irreducible toric curves from Definition \ref{def:K-mod}.)  
\end{Thm}

%%%%%%%%%%%%%%%%%%%%%%%%%%%%%%%%%%%%%%%%%%%%%%%%%%%%%%%%%%%%%%%%%%%%%%%%%%%%%%%%%%%%%%%%%%%%%%%%%%%%%%%

\subsection{Construction}

%%%%%%%%%%%%%%%%%%%%%%%%%%%%%%%%%%%%%%%%%%%%%%%%%%%%%%%%%%%%%%%%%%%%%%%%%%%%%%%%%%%%%%%%%%%%%%%%%%%%%%%

\subsubsection{Semi-flat SYZ mirrors} \label{sec:semi-flat}
Recall that we have a Lagrangian fibration $\pi: X^o \to B = M'_\R \times [0,2\delta)$ from Definition \ref{def:Lag-fib} 
with discriminant loci given by $\partial B = M'_\R \times \{0\}$ and $\Gamma \times \{\delta\}$. 
Denote the complement by $B_0 := B - \partial B - (\Gamma \times \{\delta\})$ and a fiber of $\pi$ at $r \in B_0$ by $L_r$. 
We have the dual torus bundle over $B_0$
$$ \check{\pi}: \check{X}_0 := \{(L_r,\nabla):r \in B_0, \nabla \textrm{ is a flat } \U(1) \textrm{-connection on } L_r\} \longrightarrow B_0. $$
The total space $\check{X}_0$ is called the semi-flat mirror of $X_0:=\pi^{-1}(B_0)$.  

We shall first show that $\check{X}_0$ has semi-flat complex coordinates. 
%We fix the following contractible subset in $\check{X}_0$ to describe them. 
Recall from Remark \ref{rem:base-pt} that we fix an identification $N' \cong \Z^{n-1}$ such that $\R_{\geq 0} \langle (0,1),(e_1,1),\ldots,(e_{n-1},1)\rangle$ is a cone in $\Sigma$. 
Then we fix the connected component of $M'_\R - \Gamma$ corresponding to the primitive generator $(0,1) \in N$, and denote its complement in $M'_\R$ by $S$.  Then take
$ U := B_0 - S \times \{\delta\} $
which is a contractible open set.  

Let us consider a point $r_0 = (r_1,r_2) \in U$ where $r_1 \in M'_\R$ and $r_2 > \delta$. 
Then $L_{r_0}$ is isotopic to a moment-map fiber via $\{(\mu',|w-a|) = r_0\} \subset X - \bigcup_i D_i$ for $a \in [0,\delta]$. 
In particular $\pi_2(X,L_{r_0})$ can be identified with $\pi_2(X,T)$. 
Since the torus bundle $\pi^{-1}(U) \to U$ is trivial, this gives an identification of $\pi_2(X,L_r)$ with $\pi_2(X,T)$ for all $r \in U$.

We have the disc classes $\beta_{(0,1)}(r), \beta_{(e_i,1)}(r) \in \pi_2(X,L_r)$ corresponding to the primitive generators $(0,1)$ and $(e_i,1)$ for $i=1,\ldots,n-1$ respectively. 
Let us denote $\bar{\beta}_{e_i}(r) := \beta_{(e_i,1)}(r) - \beta_{(0,1)}(r)$.
Then semi-flat complex coordinates are given by
$$ z_i(L_r,\nabla) := \exp \left(- \int_{\bar{\beta}_{e_i}(r)} \omega \right) \mathrm{Hol}_\nabla (\partial \beta_{(e_i,1)} - \partial \beta_{(0,1)}) $$
for $i=1,\ldots,n-1$ and
$$ z_0(L_r,\nabla) := \exp \left(- \int_{\beta_{(0,1)}(r)} \omega \right) \mathrm{Hol}_\nabla (\partial \beta_{(0,1)})$$
where $(L_r,\nabla) \in \check{\pi}^{-1}(U)$, where $\omega$ is the K\"ahler form in Definition \ref{def:metric}.

%%%%%%%%%%%%%%%%%%%%%%%%%%%%%%%%%%%%%%%%%%%%%%%%%%%%%%%%%%%%%%%%%%%%%%%%%%%%%%%%%%%%%%%%%%%%%%%%%%%%%%%

\subsubsection{Wall-crossing of open Gromov--Witten invariants}

A key ingredient in SYZ construction is open Gromov--Witten invariant, definition of which is rather involved.  We refer to \cite[Section 2.1.2]{FOOO} for the moduli spaces of stable discs, and \cite[Chapter 7]{FOOO} for transversality issues.  We shall restrict to the situation that no disc bubbling occurs, so that the disc moduli does not have boundary and can be treated as the moduli spaces in closed Gromov--Witten theory.  This avoids the ambiguity of defining the invariants.

First of all, recall that the wall consists of the locations of Lagrangian fibers where disc bubbling occurs and hence the invariants are not well-defined.  We will remove the wall in order to talk about open Gromov--Witten invariants.

\begin{Def} \label{def:wall}
The wall $H$ of the Lagrangian fibration $\pi$ is
$$ H := \{r \in B_0: L_r \textrm{ bounds a non-constant holomorphic disc of Maslov index $\le 0$} \} $$
\end{Def}

In the infinite-type case, we need to make sure that stable discs of the same class stay in a compact region, so that the disc moduli is compact.

\begin{Lem}
Let $\beta \in \pi_2(X,L_r)$ have Maslov index $2$, where $r \in B_0 - H$.  Then any stable disc (with any fixed number of boundary marked points) in $\beta$ lies in a compact region.
\end{Lem}

\begin{proof}
$\beta$ belongs to $H_2(X^{(k)},L_r)$ for $k \gg 0$. 
Since any non-constant holomorphic disc bounded by $L_r$ has at least Maslov index $2$, the class $\beta$ takes the form $\beta_i + \alpha$ where $\beta_i$ is a basic disc class and $\alpha$ is a curve class.  Any rational curve in $\alpha$ is contained in the toric divisors and has zero intersection with any toric divisor not belonging to $X^{(k)}$.  A holomorphic disc in $\beta_i$ intersects only the toric irreducible divisor $D_i$.  As a stable disc in $\beta$ is connected, the curve component must intersect $D_i$.  Since $X^{(k)}$ is a toric Calabi--Yau manifold of finite-type whose fan has convex support, any such curve in $\alpha$ is contained in $X^{(k)}$ and lies in a compact region.
\end{proof}

The definition of an open Gromov--Witten invariant is briefly recalled as follows.

\begin{Def}[Open Gromov--Witten invariant] \label{def:oGW}
Let $L_r$ be a Lagrangian torus which bounds no non-constant holomorphic disc of Maslov index $\le 0$.  
Let $\beta \in \pi_2(X,L)$, and denote by $\mathcal{M}_1(\beta)$ the moduli space of stable discs with one boundary marked point representing $\beta$.  The open Gromov--Witten invariant associated to $\beta$ is $n_\beta := \int_{\mathcal{M}_1(\beta)} \mathrm{ev}^*[\mathrm{pt}]$, 
where $\mathrm{ev}:\mathcal{M}_1(\beta) \to L$ is the evaluation map at the boundary marked point.
\end{Def}

For dimension reason $n_\beta$ is non-zero only when $\beta$ has Maslov index $2$. 
The condition that $L$ bounds no non-constant holomorphic disc of Maslov index $\le 0$ makes sure that disc-bubbling does not occur and so $\mathcal{M}_1(\beta)$ has no codimension $1$ boundary.  
Then $n_\beta$ is well-defined and in particular does not depend on the choice of perturbations in Kuranishi structure.

The following proposition corresponds to \cite[Lemma 4.27, Propositions 4.30 \& 4.32]{CLL}, describing the wall-crossing of the open Gromov--Witten invariants. 
The proof is parallel to the finite-type case and is omitted here. 

\begin{Prop} \label{prop:wc}
The wall is given by $H = M'_\R \times \{\delta\}$. 
We have $B_0 - H = B_+ \coprod B_-$ where $B_+ = \{(r_1,r_2) \in B_0: r_2 > \delta\}$ and $B_- = \{(r_1,r_2) \in B_0: r_2 < \delta\}$. 
For $r \in B_+$, we have $n_\beta^{L_r} = n_\beta^{T}$ where $T$ denotes a regular moment-map fiber. 
For $r \in B_-$, $n_\beta^{L_r}$ equals to $1$ when $\beta = \beta_{(0,1)}$ and $0$ otherwise.

Moreover $n_\beta^T \not= 0$ only when $\beta = \beta_{(v,1)} + \alpha$ where $\alpha \in H_2(X,\Z)$ is a rational curve class, and $\beta_{(v,1)}$ is a basic disc class corresponding to the primitive generator $(v,1)$ of the fan.  We have $n_{\beta_{(v,1)}} = 1$.
\end{Prop}
The last assertion in the proposition is due to the result of Cho--Oh \cite{CO}.

Thus the only non-trivial open Gromov--Witten invariants are $n_{\beta_{(v,1)}+\alpha}^{L_r}$ where $r \in B_+$ and $\alpha \not=0$. 
The corresponding disc moduli space has sphere-bubbling contributions which lead to non-trivial obstructions. 
It turns out that these invariants exactly correspond to the instanton corrections in the mirror map and in particular can be extracted from solutions of the GKZ system.  
This will be done in the next subsection.

%The SYZ mirror will have infinitely many terms, and we need to work with the ring 
%$\C[[z_1^{\pm 1}, \ldots, z_{n-1}^{\pm 1}]][[q_1,\ldots]]^f.$ (see Definition \ref{def:K-mod}). 

%%%%%%%%%%%%%%%%%%%%%%%%%%%%%%%%%%%%%%%%%%%%%%%%%%%%%%%%%%%%%%%%%%%%%%%%%%%%%%%%%%%%%%%%%%%%%%%%%%%%%%%

\subsubsection{SYZ mirrors with quantum corrections}

We obtain the following generating function of open Gromov--Witten invariants corresponding to the boundary divisor $D_u = \{w = \delta\}$:
for $(L_r,\nabla) \in B_0 - H$, 
\begin{equation}
u (L_r,\nabla) := \sum_{\substack{\beta \in \pi_2(X,L_r) \\ \beta \cdot D_u = 1, \beta\pitchfork D_u}} n_\beta \cdot \mathrm{e}^{-\int_{\beta}\omega}\Hol_{\nabla}(\partial \beta). 
\label{eq:gen}
\end{equation}
It follows from Proposition \ref{prop:wc} that $u$ can be expressed in terms of the semi-flat complex coordinates $z_0,\ldots,z_n$ as follows. 
The detail can be found in \cite[Proposition 4.39]{CLL}.

\begin{Prop} \label{prop: gen fnct}
With the same notation as in Theorem \ref{thm:SYZ}, we have
$$
u (L_r,\nabla) = \left\{
\begin{array}{ll}
z_0 \left(\sum_J \left(\sum_{\alpha \in H_2^{\mathrm{eff}}(X,\Z)} n_{\beta_{(J,1)}+\alpha} q^\alpha\right) q^{A_J} z^J\right) & \textrm{ for }  r \in B_+ \\
z_0 & \textrm{ for }  r \in B_-
\end{array}
\right.
$$
\end{Prop}
%Explanations of the notations can be found in Theorem \ref{thm:SYZ}.

According to \cite{CLL}, the SYZ mirror is given by the equation $uv = F^\open$, where $F^\open$ is the wall-crossing factor of the generating function $u$\footnote{
More precisely, the $v$ variable is obtained as the generating function of open Gromov--Witten invariants corresponding to the divisor $D_\infty$ 
coming from a symplectic cut \cite{Ler} by the Hamiltonian circle action of $(0,-1) \in N = N' \times \Z$. 
The mirror equation $uv = F^\open$ is the relation between the $2$ functions $u$ and $v$.}. 
It follows from Proposition \ref{prop: gen fnct} that the SYZ mirror is the one given in Theorem \ref{thm:SYZ}.

%In order to construct the SYZ mirror of a toric Calabi--Yau manifold $X^o$, we consider its partial compactification given by a symplectic cut $\overline{X^o}=\overline{X}_{|w-\delta|<2\delta}$ (see Definition \ref{def:Lag-fib}).  The procedure is summarized as follows: 

%$n_\beta$ is called the one-pointed open Gromov--Witten invariant associated to the disc class $\beta$, see Definition \cite{def:oGW} below.  
%The coordinate $Z_i$ can change dramatically from one chamber to another of $(\pi^\vee)^{-1}(B_0\setminus H)$, and this jump is called the wall-crossing phenomenon.

By definition $F^\open$ belongs to the ring $\C[[z_1^{\pm 1}, \ldots, z_{n-1}^{\pm 1}]][[q_1,\ldots]]^f / I$, 
where we recall that $\C[[q_1,\ldots]]^f / I$ is the K\"ahler moduli given in Definition \ref{def:K-mod}. 
Alternatively, $F^\open$ can be deduced as the limit of the corresponding wall-crossing factors of members of the exhaustion $X_k$, the toric Calabi-Yau manifolds corresponding to the fans $\Sigma_k$ in Definition \ref{def:exhaustion}.
For this we recall the natural topology for power series ring.

\begin{Def}
We define a topology on $\C[[z_1^{\pm 1}, \ldots, z_{n-1}^{\pm 1}]][[q_1,\ldots]]^f$ as follows. 
The basic open sets takes the form 
$$
\left\{\sum_{I,J} a_{I,J} z^I q^J: \sum_{J:[q^{J}] = h_1} a_{I_1,J} = c_1, \ldots, \,\, \sum_{J:[q^{J}] = h_p} a_{I_p,J} = c_p \right\}
$$ 
for some fixed $p \geq 0$, indices $I_i \in \Z^{n-1}$, monomial classes $h_i$ in $\C[[q_1,\ldots]]^f/I$, and $c_i \in \C$ for $i=1,\ldots,p$.  $\C[[z_1^{\pm 1}, \ldots, z_{n-1}^{\pm 1}]][[q_1,\ldots]]^f/I$ is equipped with the quotient topology. 
%Open sets in the quotient have the same description if we choose a representative for each equivalent class of monomials and write a series in terms of the representatives (and this is valid because for a series in $\C[[z_1^{\pm 1}, \ldots, z_{n-1}^{\pm 1}]][[q_1,\ldots]]^f$, there are only finitely many terms in the same equivalent class).
\end{Def}

It is easy to see the following.
\begin{Prop}
$\C[[z_1^{\pm 1}, \ldots, z_{n-1}^{\pm 1}]][[q_1,\ldots]]^f/I$ is Hausdorff in the above topology.  A sequence $(f_i)$ in $\C[[z_1^{\pm 1}, \ldots, z_{n-1}^{\pm 1}]][[q_1,\ldots]]^f/I$ is convergent if and only if each term eventually stabilizes.
\end{Prop}

The SYZ mirror can be obtained by taking the limit of the exhausting finite-type Calabi--Yau manifolds of $X$.

\begin{Prop} \label{prop:lim}
Let $X$ be a toric Calabi--Yau $n$-fold of infinite-type. 
Let $X_1\subset X_2 \subset \ldots \subset X$ be the sequence of toric Calabi--Yau $n$-folds of finite-type corresponding to the compact exhaustion $P_1\subset P_2 \subset \ldots \subset P$. 
Let 
$$
\{((u,v),x_1,\ldots,x_{n-1}) \in \C^2 \times (\C^\times)^{n-1}: uv =  F^\open_k (q_1,\ldots,q_{N_k}; x_1,\ldots,x_{n-1})\}
$$
be the SYZ mirrors of $X_k$. %(defined using the same based point corresponding to the cone spanned by $\{(0,1),(e_i,1)\}_{i=1}^{n-1}$).  
Then $F^\open$ is the limit of $F^\open_k$ in $\C[[z_1^{\pm 1}, \ldots, z_{n-1}^{\pm 1}]][[q_1,\ldots]]^f/I$.
Moreover $F^\open_k$ is the image of $F$ under the restriction map 
$$\C[[z_1,\ldots,z_{n-1}]][[q_1,\ldots,q_{N_k},\ldots]]^f/I \longrightarrow \C[[z_1,\ldots,z_{n-1}]][[q_1,\ldots,q_{N_k}]]^f/I_k$$ 
given in Definition \ref{def:restr}.
\end{Prop}

\begin{proof}
$ F^\open_k$ takes the form
$$  F^\open_k = \sum_{v} \sum_{\alpha \in H_2^\mathrm{eff} (X_k)} n^{X_k}_{\beta_{v} + \alpha} q^\alpha q^{A_v} z^v $$
where $v \in N' = \Z^{n-1}$ such that $(v,1)$ are generators of the fan of $X_k$; see also Theorem \ref{thm:SYZ} for the notations.

To prove that $ F^\open_k$ limits to $F^\open$, it suffices to see that each coefficient $n^{X_k}_{\beta_{(J,1)} + \alpha}$ equals to $n^{X}_{\beta_{(J,1)} + \alpha}$ for $k$ large enough.
Let $\beta$ be a basic disc class of a regular moment map fiber of $X$ and $\alpha \in H_2^{\mathrm{eff}}(X,\Z)$.  For $k$ large enough, $\beta+\alpha$ is a disc class of a regular moment map fiber of $X_k$.  Consider the moduli space $\mathcal{M}^{X}_{0,1,\beta + \alpha}$. The elements consist of a nodal union of a basic disc in $\beta$ and a curve in $\alpha$.  In order to have a non-empty intersection with a basic disc in $\beta$, all such curves are contained in a compact subset of $X_k \subset X$.  Thus we have $\mathcal{M}^{X_k}_{0,1,\beta + \alpha} = \mathcal{M}^{X}_{0,1,\beta + \alpha}$, and hence $n_{\beta + \alpha}^X = n_{\beta + \alpha}^{X_k}$ for $k$ large enough.

Suppose $n^{X_k}_{\beta_{v} + \alpha} q^\alpha q^{A_v} z^v$ is a term in $F^\open$ but not in $ F^\open_k$.  Then either $(v,1)$ is not a ray of the fan of $X_k$, or $\alpha$ is not a curve class in $X_k$.  In the first case, $A_v$ is not a curve class in $X_k$, and hence any monomial equivalent to $q^{A_v}$ must contain K\"ahler parameters which do not belong to $X_k$.  In the second case $q^\alpha$ contain K\"ahler parameters not belonging to $X_k$.  Thus $ F^\open_k$ is obtained from $F^\open$ by setting the extra K\"ahler parameters to be $0$.
\end{proof}

%%%%%%%%%%%%%%%%%%%%%%%%%%%%%%%%%%%%%%%%%%%%%%%%%%%%%%%%%%%%%%%%%%%%%%%%%%%%%%%%%%%%%%%%%%%%%%%%%%%%%%%
%%%%%%%%%%%%%%%%%%%%%%%%%%%%%%%%%%%%%%%%%%%%%%%%%%%%%%%%%%%%%%%%%%%%%%%%%%%%%%%%%%%%%%%%%%%%%%%%%%%%%%%

\subsection{Open mirror theorem and Gross--Siebert normalization} \label{sec:renorm}

The open mirror theorem for toric Calabi--Yau manifolds of infinite-type can be deduced purely algebraically from the corresponding theorem for finite-type \cite{CCLT} and Proposition \ref{prop:lim}. 
This gives an explicit computation of all the coefficients of the SYZ mirror in Theorem \ref{thm:SYZ}. 

\begin{Thm}[Open mirror theorem] \label{thm:open-mir-thm}
Let $X$ be a toric Calabi--Yau manifold which could be of infinite-type.  Then we have
$$\sum_{\alpha} n_{\beta_l + \alpha} q^{\alpha}(\check{q}) = \exp(g_l(\check{q})),$$ where we recall from Proposition \ref{prop:mir-map} that
\begin{equation}\label{eqn:funcn_g}
g_l(\check{q}):=\sum_{d}\frac{(-1)^{(D_l\cdot d)}(-(D_l\cdot d)-1)!}{\prod_{p\neq l} (D_p\cdot d)!}\check{q}^d,
\end{equation}
the summation is over all effective curve classes $d\in H_2^\text{eff}(X,\Z)$ satisfying $-K_X\cdot d=0, D_l\cdot d<0 \text{ and } D_p\cdot d \geq 0 \text{ for all } p\neq l,$ 
and $q(\check{q})$ is the mirror map.
\end{Thm}

\begin{proof}
By Proposition \ref{prop:lim}, we have the convergence 
$$
F^\open_k = \sum_{l\in\Sigma_k^{(1)}} \left(\sum_{\alpha \in H_2^\mathrm{eff} (X_k,\Z)} n^{X_k}_{\beta_{l} + \alpha} q^\alpha \right) q^{A_l} z^{v_l} \longrightarrow 
F^\open = \sum_{l\in\Sigma^{(1)}} \left( \sum_{\alpha \in H_2^\mathrm{eff} (X,\Z)} n^X_{\beta_{l} + \alpha} q^\alpha \right) q^{A_l} z^{v_l}
$$
where $\Sigma_k$ is the exhaustion of $\Sigma$ in Definition \ref{def:exhaustion} and $\Sigma^{(1)}$ denotes the set of rays of $\Sigma$. 
Thus $\sum_{\alpha \in H_2^\mathrm{eff} (X_k,\Z)} n^{X_k}_{\beta_{l} + \alpha} q^\alpha$ converges to $\sum_{\alpha \in H_2^\mathrm{eff} (X,\Z)} n^X_{\beta_{l} + \alpha} q^\alpha$.

By the open mirror theorem \cite[Theorem 1.4]{CCLT} for the toric Calabi--Yau manifold $X_k$ (where $k$ is big enough so that $D_l$ is contained in $X_k$), we have
$$ \sum_{\alpha \in H_2^\mathrm{eff} (X_k,\Z)} n^{X_k}_{\beta_{l} + \alpha} q^\alpha_{X_k}(\check{q}) = \exp(g^{X_k}_l(\check{q})) $$
where $q_{X_k}(\check{q})$ is the mirror map for $X_k$, $g^{X_k}_l(\check{q})$ takes the same expression as in Equation \eqref{eqn:funcn_g}, where the summation is over all effective curve classes $d\in H_2^\text{eff}(X_k,\Z)$ satisfying the same conditions, and $D_p$ in the expression are required to be toric divisors of $X_k$.

For each $d \in H_2^\mathrm{eff}(X,\Z)$, curves in $d$ are contained in $X_k$ for $k \gg 0$.  In particular $D \cdot d = 0$ for every toric divisor not contained in $X_k$. 
Thus the summand of $g_l^{X_k}$ corresponding to $d$ agrees with that of $g_l^X$ for $k \gg 0$. 
This implies $g_l^{X_k}(\check{q})$ converges to $g_l(\check{q})$ in $\C[[\check{q}_1,\ldots]]^f/I$, and so the same is also true for the mirror maps, 
namely $q^\alpha_{X_k}(\check{q})$ converges to $q^\alpha(\check{q})$. 
By taking $k \to \infty$ it follows that $\sum_{\alpha} n_{\beta_l + \alpha} q^{\alpha}(\check{q}) = \exp(g_l(\check{q}))$.
\end{proof}

%%%%%%%%%%%%%%%%%%%%%%%%%%%%%%%%%%%%%%%%%%%%%%%%%%%%%%%%%%%%%%%%%%%%%%%%%%%%%%%%%%%%%%%%%%%%%%%%%%%%%%%
%%%%%%%%%%%%%%%%%%%%%%%%%%%%%%%%%%%%%%%%%%%%%%%%%%%%%%%%%%%%%%%%%%%%%%%%%%%%%%%%%%%%%%%%%%%%%%%%%%%%%%%

In the current SYZ construction, we have chosen the toric fixed point corresponding to the maximal cone $\cb=\R_{\geq 0} \cdot \langle (0,1),(e_1,1),\ldots,(e_{n-1},1)\rangle$ as a base point and the primitive generator $v=(0,1) \in \cb$ for a trivialization of the torus bundle (Remark \ref{rem:base-pt} and the beginning of Section \ref{sec:semi-flat}). 
We can carry out the same construction for the other choices of a maximal cone $\cb$ and a primitive generator $v$ as well and obtain the functions $F^{\open}_{\cb,v}(q;z)$.

It was shown in \cite{Lau} that $\{F^\open_{\cb,v_i}(q;z)\}$ satisfies the Gross--Siebert normalization condition for the toric Calabi--Yau manifolds of finite-type by using the open mirror theorem of \cite{CCLT}.  The normalization is an essential ingredient in the Gross-Siebert program of toric degenerations \cite{GS,GS3}.  ( $\cb$ is called a slab, which is a subset of the wall that we choose to pass through.)  Since we still have the open mirror theorem (Theorem \ref{thm:open-mir-thm}) for toric CY of infinite-type, the same proof as in \cite{Lau} goes through to show that Gross--Siebert normalization still holds in our context. 

\begin{Thm} \label{thm: GS norm}
	For a toric Calabi--Yau manifold of infinite-type, the collection of generating functions $F^{\open}_{\cb,v}(q;z)$ satisfies and is uniquely determined by the following Gross--Siebert normalization conditions.
	\begin{enumerate}
		\item The constant term of each $F^{\open}_{\cb,v}(q;z)$ (as a series in $q$ and $z$) is $1$.
		\item If $v_i$ and $v_j$ are adjacent vertices of $\cb$, then $ F^{\open}_{\cb,v_i}(q;z) = q^{A_{v_j}-A_{v_i}} z^{v_j-v_i} F^{\open}_{\cb,v_j}(q;z)$. 
		(See Theorem \ref{thm:SYZ} for the definition of $A_{v}$). 
		\item If $v \in \cb \cap \cb'$, then $F^{\open}_{\cb,v}(q;z) = F^{\open}_{\cb',v}(q;z)$.
		\item 	
		\begin{equation}
		\log F^\open_{\cb,v_i}(q;z) = \sum_{k=1}^\infty \frac{(-1)^{k-1}}{k} (F^\open_{\cb,v_i}(q;z) - 1)^k
		\end{equation}
		has no term of the form $a \cdot q^{C}$ where $a \in \C^\times$ and $C \in H_2(X,\Z)-\{0\}$.
	\end{enumerate}
\end{Thm}

The normalization will be useful for explicit computations of SYZ mirrors in Sections \ref{section: Calabi--Yau}, \ref{section: 3-fold} and \ref{section: higher-dim}.

%%%%%%%%%%%%%%%%%%%%%%%%%%%%%%%%%%%%%%%%%%%%%%%%%%%%%%%%%%%%%%%%%%%%%%%%%%%%%%%%%%%%%%%%%%%%%%%%%%%%%%%
%%%%%%%%%%%%%%%%%%%%%%%%%%%%%%%%%%%%%%%%%%%%%%%%%%%%%%%%%%%%%%%%%%%%%%%%%%%%%%%%%%%%%%%%%%%%%%%%%%%%%%%

\subsection{SYZ under free group actions} \label{subsection: group action}
%********************** Write down a general setup that we have a Lagrangian fibration equivariant under a proper free discrete $G$-action (and $G$-invariant Kaehler structure) which descends to a proper free $G$-action on the base.  Assume that we have a compactification such that we have enough quantum-corrected coordinates.  We pull back a simply connected chart from quotient to have a simply connected chart upstairs.  Then we have action on semi-flat coordinates, and hence quantum corrected coordinates.  The wall-crossing maps are preserved under $G$ up to an overall scaling by $z^{\beta_0 \cdot g - \beta_0}$ and hence we have an action of $G$ on $\check{X}$ ********************************
Now we consider a group action as in Section \ref{subset: free group}, where $G < \SL(N)$ (where $N$ denotes a lattice) acts on the fan $\Sigma - \{0\}$ supported in $N_\R$ freely, 
where $X_\Sigma$ is a toric Calabi--Yau manifold given in Definition \ref{def:torCY}. 
From Proposition \ref{prop:G-metric} and \ref{prop:fib-quot}, we have a $G$-invariant K\"ahler structure on $X^o$ and a Lagrangian fibration $X^o \to M_\R' \times [0,2\delta)$ which descends to the quotient.  

In this subsection, we construct the SYZ mirror of the quotient $X^o/G$ via taking a $G$-quotient of $\check{X}$. 
In general, to construct the SYZ mirror for a Lagrangian fibration $Y \to B$ of a K\"ahler manifold $Y$ where the base $B$ is not simply connected, 
one needs to cover $B$ by simply connected open sets, study the wall-crossing phenomenon and construct the SYZ mirror of each open set, 
and argue that they can be glued together to give a global mirror of $Y \to B$. 
On the other hand, it is conceptually cleaner by pulling back the Lagrangian fibration as $Z:=Y\times_B \widetilde{B} \to \widetilde{B}$ over the universal cover $\widetilde{B}$, 
construct the SYZ mirror $\check{Z}$ of $Z \to \widetilde{B}$, and take the quotient of $\check{Z}$ by the deck transformation group to define the SYZ mirror of $Y \to B$. 
We shall take this approach in this section.  

In our case, $Y = X^o/G$ and the base $B = (M_\R'/G) \times [0,2\delta)$ where $G$ acts on $M_\R'$ freely. 
Hence the universal cover is $\widetilde{B} = M_\R' \times [0,2\delta)$, the deck transformation group is $G$ itself and the pull-back is exactly the original Lagrangian fibration $X^o \to M_\R' \times [0,2\delta)$. 
The task is to construct a natural $G$-action on the SYZ mirror $\check{Z}$, induced from the action of $G$ on $X^o$.

Recall that the semi-flat mirror $\check{X}^o$ of $X^o$ is the space of pairs $(L_r,\nabla)$ where $L_r \subset X^o$ is a non-singular Lagrangian fiber equipped with a flat $\U(1)$-connection $\nabla$.  

\begin{Lem} \label{lem:G-wall}
The group $G$ takes a non-singular Lagrangian fiber to another non-singular Lagrangian fiber.  Moreover if a non-singular fiber $L_r$ bounds a non-constant holomorphic disc of Maslov index $0$, then the same holds for $L_r \cdot g$ for any $g \in G$.  In particular $G$ has an action on the base $B$ which preserves the discriminant locus and the wall of the Lagrangian fibration.
\end{Lem}

\begin{proof}
Since $G$ acts as toric morphisms, it maps toric orbits to toric orbits. 
In particular there is a unique $G$-action on the moment map image of $X^o$ such that the moment map $\mu$ is $G$-equivariant. 
Since $\nu = (0,1) \in M$ is invariant under $G$ (Lemma \ref{lem:torCY}), the $G$-action descends to $M_\R / \R \cdot \nu$ and 
$\mu'$, the composition of $\mu$ with the projection $M_\R \to M_\R / \R \cdot \nu$, is also $G$-equivariant.  
The holomorphic function $w$ corresponding to $\nu$ is also $G$-invariant, and hence the fibration map $(\mu',|w-\delta|)$ is $G$-equivariant. 
In particular $G$ maps fibers to fibers.

Moreover the $G$-action preserves the toric stratification.  A fiber is singular if and only if it hits a codimension-$2$ toric strata. 
Hence singular fibers are mapped to singular fibers under $G$.  Thus the $G$-action on $B$ preserves the discriminant locus.

Since $G$ preserves the whole K\"ahler structure, it maps a non-constant holomorphic disc of Maslov index $0$ bounded by $L_r$ to that bounded by $L_r \cdot g$. 
Hence the wall is also preserved by $G$.
\end{proof}

In order to understand the $G$-action on the semi-flat complex structure, it would be easier to use a chart of $\check{X}_0$ which is preserved by $G$.  Unfortunately the chart $U$ taken in Section \ref{sec:semi-flat} is not preserved by $G$, since $G$ can map the connected component of $M'_\R - \Gamma$ corresponding to $(0,1) \in N$ to another component.  Instead we take the following $G$-invariant chart.

Recall that the discriminant loci in the base $B = M_\R = M_\R' \times \R$ are given by $\partial B = M_\R' \times \{0\}$ and $\Gamma \times \{\delta\}$.  The fundamental group $\pi_1(B)$ is generated by loops winding around the codimension-$2$ locus $\Gamma \times \{\delta\}$.  Now take a contractible open set $U' = B_0 - (\Gamma \times [0, \delta])$. 
$G$ preserves $B_0$ and $\Gamma$.  Moreover the last component $|w-\delta|$ of the fibration map is invariant under $G$, and hence $U'$ is preserved by $G$.

As in Section \ref{sec:semi-flat}, we pick a point $r_0 = (r_1,r_2) \in U'$ where $r_1 \in M'_\R - \Gamma$ and $r_2 > \delta$, 
and identify the fiber $L_{r_0}$ with a moment map fiber by the Lagrangian isotopy
$$
\{(\mu',|w-a|) = r_0\} \subset X - \bigcup_i D_i, \ \ \ a \in [0,\delta].
$$
Since $U'$ is contractible, any other fiber $L_r$ is identified with $L_{r_0}$. 
This gives identifications $H_1(L_{r}) \cong H_1(T)$ and $H_2(X,L_{r}) \cong H_2(X,T)$ where $T$ denotes a moment-map fiber.

There is a key difference between this identification and that in Section \ref{sec:semi-flat}. 
For different choices of $r_1$ in different chambers of $M'_\R - \Gamma$, the identifications $H_1(L_{r}) \cong H_1(T)$ and $H_2(X,L_{r}) \cong H_2(X,T)$ are different. 
In particular if we compose the identification $H_1(L_{r}) \cong H_1(T)$ by one choice of $r_1$, with the identification $H_1(T) \cong H_1(L_{r})$ by another choice of $r_1$, the resulting endomorphism on $H_1(L_{r})$ is a non-trivial monodromy (if the $2$ choices of $r_1$ live in different chambers of $M'_\R - \Gamma$).  This does not occur in Section \ref{sec:semi-flat} since a chamber (namely the one corresponding to $(0,1) \in N$) is fixed in the beginning in the definition of the contractible open set $U$. 
The chambers of $M'_\R - \Gamma$ are called slabs in the Gross--Siebert program \cite{GS}.

We fix the above choice of $r_1$ to be in the chamber of $M'_\R - \Gamma$ corresponding to $(0,1) \in N$.  Let
$$ z^\beta (L_r,\nabla) := \exp \left(- \int_{\beta(r)} \omega \right) \mathrm{Hol}_\nabla (\partial \beta)$$
be the semi-flat complex coordinate on $U'$ corresponding to a disc class $\beta \in H_2(X,T)$, which is identified with a disc class in $H_2(X,L_r)$.  
As in Section \ref{sec:semi-flat}, we have the disc classes $\beta_0$, $\beta_{v_i} - \beta_0$ for $i=1,\ldots,n-1$ 
(where $\{(0,1),(v_i,1) \textrm{ for } i=1,\ldots,n-1\}$ is a maximal cone of $\Sigma$) whose boundary classes form a basis of $N = H_1(T)$.  
Thus for any $\beta$, we have 
$$
\partial \beta = a_0 \partial \beta_0 + \sum_{i=1}^{n-1} a_i (\partial \beta_i - \partial \beta_0)
$$
for some $a_i \in \Z$, and hence $z^\beta$ can be written in terms of the coordinates $(z^{\beta_0},z^{\beta_1-\beta_0},\ldots,z^{\beta_{n-1}-\beta_0})$ as
$$
z^\beta = q^{\beta - a_0 \beta_0 - \sum_{i=1}^{n-1} a_i (\beta_i - \beta_0)} (z^{\beta_0})^{a_0} \prod_{i=1}^{n-1} (z^{\beta_i-\beta_0})^{a_i}
$$
where $\beta - a_0 \beta_0 - \sum_{i=1}^{n-1} a_i (\beta_i - \beta_0) \in H_2(X,\Z)$.  This is regarded as an element in 
$$
((\C[[q_1,\ldots]]^f/I)/G)[z^{\pm\beta_0}][[z^{\pm(\beta_1-\beta_0)},\ldots,z^{\pm(\beta_{n-1}-\beta_0)}]]. 
$$
We recall the reader that $\C[[q_1,\ldots]]^f/I)/G$ is defined as the K\"ahler moduli of the quotient $X^o/G$ (Definition \ref{def:K-mod-quot}).

The induced action of $G$ on the semi-flat coordinates is given by the following.

\begin{Lem} \label{lem:G-z}
We have $ g^* \, z^\beta = z^{\beta \cdot g^{-1}}$ for $g \in G$.  
\end{Lem}
\begin{proof}
By definition, we have 
$$
(g^* \, z^\beta) (L,\nabla) = z^\beta (L \cdot g, \nabla \cdot g) = \exp \left(- \int_{\beta^{L \cdot g}} \omega \right) \mathrm{Hol}_{\nabla \cdot g} (\partial \beta^{L \cdot g})
$$
where $\nabla \cdot g$ denotes the pull back of the flat connection $\nabla$ on $L$ to $L \cdot g$ by $g^{-1}$. 
Then we have 
$\beta^{L \cdot g} = (\beta \cdot g^{-1})^L \cdot g \in \pi_2(X,L \cdot g)$ 
and
$$
\mathrm{Hol}_{\nabla \cdot g} (\partial \beta^{L \cdot g}) = \mathrm{Hol}_{\nabla} (\partial \beta^{L \cdot g} \cdot g^{-1}) = \mathrm{Hol}_{\nabla} ((\partial {\beta \cdot g^{-1}})^{L}).
$$  
Also since the K\"ahler structure is $G$-invariant, we have 
$$
\int_{(\beta \cdot g^{-1})^L \cdot g} \omega = \int_{(\beta \cdot g^{-1})^L} (g^{-1})^* \omega = \int_{(\beta \cdot g^{-1})^L} \omega.
$$  
As a result, it follows that 
$$
(g^* \, z^\beta) (L,\nabla) = \exp \left(\int_{(\beta \cdot g^{-1})^L} \omega\right) \mathrm{Hol}_{\nabla} ((\partial {\beta \cdot g^{-1}})^{L}) = z^{\beta \cdot g^{-1}} (L,\nabla).
$$
\end{proof}

The wall divides $B_0$ into $2$ chambers, $B_0 - H = B_+ \cup B_-$.  Note that for this choice of $U'$, while $B_+$ is still connected, $B_- \cap U'$ consists of the connected components $B_{-,v} := C_v \times (0,\delta)$ where $C_v$ is a chamber of $M'_\R - \Gamma$ corresponding to a primitive generator $(v,1)$ of $\Sigma$ ($v \in N'$).  ($B_\pm$ are given in Proposition \ref{prop:wc}).  The chamber structure is preserved under $G$.

\begin{Lem} \label{lem:G-chamber}
We have $B_+ \cdot g = B_+$ and $B_- = B_- \cdot g$ for $g \in G$.
\end{Lem}
\begin{proof}
Since $G$ preserves the holomorphic volume form on $X$, it preserves the orientation of the base $B$. 
Moreover it preserves the wall $H$ by Lemma \ref{lem:G-wall}.  Hence it preserves the chambers above and below the wall. 
\end{proof}

Now we need to consider the $G$-action on generating functions of open Gromov--Witten invariants.  The following simple lemma would be useful.

\begin{Lem} \label{lem:G-oGW}
We have $ n_\beta^L = n_{\beta \cdot g}^{L}$ for a Lagrangian torus fiber $L$ over $B_0 - H$. 
\end{Lem}
\begin{proof}
Since the $G$-action preserves the K\"ahler structure of $X$, it gives an isomorphism between the moduli spaces $\mathcal{M}^L_1(\beta) \cong \mathcal{M}^{L \cdot g}_1(\beta \cdot g)$ for any $g \in G$.  As a result we have $n_\beta^L = n_{\beta \cdot g}^{L \cdot g}$.  By Lemma \ref{lem:G-chamber} $L \cdot g$ and $L$ belongs to the same chamber.  Thus $n_{\beta \cdot g}^{L \cdot g} = n_{\beta \cdot g}^{L}$.
\end{proof}

In parallel to Proposition \ref{prop: gen fnct}, we have the following expression of $u$ (c.f. Equation \eqref{eq:gen}).

\begin{Prop} \label{prop:u}
We have 
\begin{equation}
u (L_r,\nabla) = \left\{
\begin{array}{ll}
z^{\beta_0} \cdot F^\open & \textrm{ for }  r \in B_+ \\
z^{\beta_v} & \textrm{ for }  r \in B_{-,v}
\end{array}
\right.
\end{equation}
and
$$ v (L_r,\nabla) = \left\{
\begin{array}{ll}
z^{-\beta_0} & \textrm{ for }  r \in B_+ \\
z^{-\beta_v} \cdot F^\open & \textrm{ for }  r \in B_{-,v}.
\end{array}
\right.$$
where 
$$F^\open = \sum_v \left(\sum_{\alpha \in H_2^{\mathrm{eff}}(X,\Z)} n_{\beta_v +\alpha} q^\alpha\right) z^{\beta_v - \beta_0} \in ((\C[[q_1,\ldots]]^f/I)/G)[[z^{\pm(\beta_1-\beta_0)},\ldots,z^{\pm(\beta_{n-1}-\beta_0)}]].$$
\end{Prop}
\begin{proof}
For $r \in B_-$, by Proposition \ref{prop:wc}, $L_r$ has a unique holomorphic disc class $\beta(r)$ of Maslov index $2$ with the property that $\beta(r) \cdot D = 1$ where $D = \{w = \delta\}$ is the boundary divisor, and $\beta(r) \cdot D_v = 0$ for all toric divisors $D_v$.  Now suppose $r \in B_{-,v_0}$ and consider the Lagrangian isotopy obtained from moving the fiber $L_r$ to $L_{r_0}$ along a path from $r$ to $r_0$ in $V'$.  Under this isotopy, $\beta(r)$ is identified with a disc class $\beta(r_0)$ of $L_{r_0}$ with $\beta(r_0) \cdot D_{v_0} = 1$ and $\beta(r_0) \cdot D_v = 0$ for all $v \not= v_0$.  Hence $\beta(r_0)$ is identified with $\beta_{v_0} \in H_2(X,T)$.  As a result, $u (L_r,\nabla) = z^{\beta_{v_0}}$ if $r \in B_{-,v_0}$.

For $r \in B_+$ it is the same as Proposition \ref{prop: gen fnct}.  Namely $L_r$ can be identified with a toric fiber $T$, and stable disc classes of Maslov index $2$ are of the form $\beta_v + \alpha$ for some basic disc class $\beta_v$ and effective curve class $\alpha \in H_2^{\mathrm{eff}}(X,\Z)$.  
As a result 
$$
u (L_r,\nabla) = \sum_v \sum_{\alpha \in H_2^{\mathrm{eff}}(X,\Z)} n_{\beta_v +\alpha} q^{\alpha} z^{\beta_v}
$$
as stated. 
The expression for $v$ is deduced similarly.
\end{proof}

Note that the wall-crossing function is $z^{\beta_0 - \beta_v} F^\open$, which depends on the slab $C_v$ passed through going from $B_{-,v}$ to $B_+$.

Due to the invariance of open Gromov--Witten invariants under $G$-action, it turns out the action on the generating function $F^\open$ is simply an overall scaling given as follows.
\begin{Prop} \label{prop:G-F}
We have $g^* \, F^{\open} = z^{\beta_0 - \beta_0 \cdot g^{-1}} \cdot F^\open$ for $g \in G$. 
Moreover, we have 
$$ g^* u = \left\{
\begin{array}{ll}
u & \textrm{ for }  r \in B_+\\
z^{\beta_0 \cdot g^{-1}-\beta_0} \cdot u & \textrm{ for }  r \in B_-
\end{array}
\right.
$$
and
$$ g^* v = \left\{
\begin{array}{ll}
z^{\beta_0 \cdot g^{-1}-\beta_0} \cdot v & \textrm{ for }  r \in B_+\\
v & \textrm{ for }  r \in B_-
\end{array}
\right. $$
\end{Prop}
\begin{proof}
As an element in $(\C[[q_1,\ldots]]^f/I)/G$, $q^{\alpha}$ is $G$-invariant for any $\alpha \in H_2(X^o)$.  Moreover $n_\beta = n_{\beta \cdot g}$ by Lemma \ref{lem:G-oGW}.  Thus 
\begin{align*}
g^* F^\open &= \sum_v \left(\sum_{\alpha \in H_2^{\mathrm{eff}}(X^o)} n_{\beta_v +\alpha} q^\alpha\right) z^{(\beta_v - \beta_0) \cdot g^{-1}} \\
&= \sum_v \left(\sum_{\alpha \in H_2^{\mathrm{eff}}(X^o)} n_{\beta_v \cdot g +\alpha} q^\alpha\right) z^{\beta_v - \beta_0 \cdot g^{-1}}\\
&= z^{\beta_0 - \beta_0 \cdot g^{-1}} \cdot F^\open
\end{align*}
where in the second equality, we rename the dummy variable $v$ to $v \cdot g$.  The expressions for $g^*u$ and $g^*v$ follow from this, Proposition \ref{prop:u} and Lemma \ref{lem:G-z}.
\end{proof}

Note that the $G$-action on $u$ and $v$ also undergoes wall-crossing.  On the other hand, the $G$-action on the product $uv$ behaves well, namely $g^* (uv) = z^{\beta_0 \cdot g - \beta_0} (uv)$.  As a result, the SYZ mirror $\check{X}$ defined by $uv = F^\open$ is preserved by the $G$-action.  To make sense of the action in terms of coordinates, we need to choose a $G$-action on the individual coordinates $u$ and $v$ which satisfies the above equality for $g^* (uv)$.  On the other hand, the resulting quotient variety $\check{X}/G$ remains the same for different choices.  As explained in the very beginning of this section, $\check{X}/G$ is the SYZ mirror of $X^o/G$.  Since $G$ is assumed to act freely on the fan $\Sigma$ and in particular freely on the rays, it acts freely on $H_2(X,T)$.  Thus $G$ acts freely on $\C[u,v][[z^{\pm(\beta_1-\beta_0)},\ldots,z^{\pm(\beta_{n-1}-\beta_0})]]$.
 We conclude with the following.

\begin{Thm} \label{thm:G-SYZ}
The SYZ mirror $\check{X}^o = \{uv = F^{\open}(z_1,\ldots,z_{n-1})\}$ of $X^o$ admits an induced free $G$-action.  %(see Equation \eqref{eq:F^open} for the expression of $F^\open$). 
The quotient $\check{X}^o/G$ has a conic fibration structure
$$
\check{X}^o/G \longrightarrow \mathrm{Spec}(\C[z_1^{\pm 1},\ldots,z_{n-1}^{\pm 1}) / G
$$with discriminant locus $\{F^\open = 0\}/G$.
\end{Thm}

\section{Local Calabi--Yau surfaces of type $\widetilde{A}$} \label{section: Calabi--Yau}

%We begin with a $2$-dimensional example. 

In the rest of this article, we apply the theory developed in the previous sections to the local Calabi--Yau surfaces of type $\widetilde{A}$ and their fiber products. 
We shall see that their SYZ mirrors have beautiful expressions in terms of modular forms and theta functions.  
We refers the reader to Appendix for some basics and notations of abelian varieties and theta functions used in this section.     

\subsection{Toric geometry}

We define the $\widetilde{A}_{d-1}$ surface ($d \geq 1$) to be the total space of the local elliptic fibration over the disc $\D$ with only one singular fiber which is of type $I_d$ in Kodaira classification.  We may denote the surface simply by $\widetilde{A}_{d-1}$.
%It is also known as the Tate family of elliptic curves. %See Figure \ref{fig:TypeAd}. 
%A generic compact elliptic K3 surface has 24 singular fibers of type $I_1$.  

Lagrangian tori on the $\widetilde{A}_{d-1}$ surface can be constructed by taking the parallel transport of vanishing cycles of $\widetilde{A}_{d-1} \to \D$ along circles in $\D$.
To make the geometric structures more explicit, we exhibit the $\widetilde{A}_{d-1}$ surface as a quotient of a toric Calabi--Yau surface of infinite-type defined as follows.

Let $N = \Z^2$, $\sigma_i=\R_{0 \ge}[i,1]+\R_{0 \ge}[i+1,1]$ be a cone in $N_\R$ for $i \in \Z$,
and $\Sigma=\bigcup_{i \in \Z} \sigma_i \subset \R^2$ the fan given as the infinite collection of these cones (and their boundary cones). 
The corresponding toric surface $X = X_\Sigma$ is Calabi--Yau since all the primitive generators $(i,1) \in N$ have second coordinates being $1$. 
(If instead we take $\Sigma_k$ to be the fan consisting of the cones $\sigma_i$ for $i=p,\ldots,p+k$, where $p$ is any fixed integer, 
then the corresponding toric Calabi--Yau surface is the resolution of $A_k$ singularity and is denoted as $\widehat{A_k}$.)
We call $X$ the $\widehat{A_\infty}$ surface.

%We define $\widehat{A_\infty}:=X_\Sigma^o$.  
%It is a scheme of finite type and also written as $\widehat{A_\infty}=\varinjlim\widehat{A_n}$. 

The fan $\Sigma$ has an obvious symmetry of $\Z$ given by $k \cdot (a,b) = (a+k,b)$ for $k\in \Z$ and $(a,b) \in N$. It is straightforward to check the following.

\begin{Lem}
Take $c_k = -\frac{k(k-1)}{2}$ for $k \in \Z$.  
Then the polytope
$$
P := \bigcap_{k \in \Z} \{ (y_1,y_2) \in M_\R \ | \ k y_1 + y_2 \geq c_k \}
$$
is invariant under the $\Z$-action on $M_\R$ defined by $1 \cdot (y_1,y_2) = (y_1,y_2-y_1) + (-1,1)$. 
\end{Lem}

In particular the polytope $P$ is invariant under the subgroup $d \Z \subset \Z$, where $d$ is a fixed positive integer.
It follows from Proposition \ref{prop:G-metric} that there exists a toric neighborhood $X_\Sigma^o$ of the toric divisors which has a $(d \Z)$-invariant toric K\"ahler metric.  
As a result we obtain a K\"ahler metric on the quotient $X^o / (d \Z)$.

Recall from Section \ref{sec:torCYinf} that $w$ denotes the toric holomorphic function on $X^o$ that corresponds to the lattice point $(0,1)\in M'\times\Z$.  This lattice point is invariant under the $\Z$-action, and hence the holomorphic function $w$ descends to the quotient $X^o / (d \Z)$.  By taking suitable $X^o$ (which is a toric neighborhood of the toric divisors) we can assume that $w$ is valued in a small disc around the origin.  It is easy to check the following.

\begin{Prop}
Except the fiber over the origin, each fiber of $w$ on $X^o / (d \Z)$ is an elliptic curve.  The fiber over the origin is singular of type $I_d$.
\end{Prop}

Hence $X^o / (d \Z)$ is a $\widetilde{A}_{d-1}$ surface.
Since $H_2(X^o)$ is spanned by the irreducible toric rational curves $\{C_i\}_{i \in \Z}$, the K\"ahler moduli for $X^o$ is given simply by $\C[[q_i: i \in \Z]]$,  
%By the action of $d\Z$, the K\"ahler parameter $q_i$ is identified with $q_{i+d}$.  
and thus the K\"ahler moduli for $X^o/d\Z$ is
$$
\C[[q_i: i \in \Z]] / \langle q_i - q_{i+d} \rangle_{i\in \Z} \cong \C[[q_1,\ldots,q_{d}]]. 
$$
The elliptic fiber class $F = C_1 + \ldots + C_d$ would be of special interest to us.

%$X_\Sigma^o$ has a natural morphism $w:\widehat{A_\infty}\rightarrow \D$ 
%corresponding to the projection of the fan $\Sigma$ to the $y$-axis (Figure \ref{fig:A_infty}).  
%\begin{figure}[htbp]
% \begin{center} 
%  \includegraphics[width=65mm]{A_infty3.eps}
% \end{center}
% \caption{Projection from $\Sigma$ to the $y$-axis and $\phi:\widehat{A_\infty}\rightarrow \D$}
%\label{fig:A_infty}
%\end{figure}
%The central fiber $w^{-1}(0)$ consists of infinite chain of $\PP^1$'s while the general fiber of $\phi$ is isomorphic to $\C^\times$.  $w$ descends to the quotient $\overline{w}:\widetilde{A}_{d-1}:=\widehat{A_{\infty}}/\Z\rightarrow \D$, whose fiber $\bar{w}^{-1}(t)$ (for $|t| < 1$) is a smooth elliptic curve $\C^\times/t^{d\Z}$ for $t \not= 0$ and is a type $I_d$ singular fiber at $t=0$.
%
%****************** K\"ahler parameters *******************

%This example appears in the classical monograph \cite{AMRT}. 

%We have $\overline{\phi}^{-1}(t) \cong \C^\times$ over the unit circle $|t|=1$.   
%We denote $\widehat{A_{\infty}}/_1\Z$ by $\widehat{A_{\infty}}//\Z$ for simplicity. 

%%%%%%%%%%%%%%%%%%%%%%%%%%%%%%%%%%%%%%%%%%%%%%%%%%%%%%%%%%%%%%%%%%%%%%%%%%%%%%%%%%%%%%%%%%%%%%%%%%%%%%%
\subsection{SYZ mirror of $\widetilde{A}_{0}$ surface}
Let us first consider the case $d=1$.  
We have only one K\"ahler parameter $q$ for $X^o/\Z$ and, by Lemma \ref{lem:G-oGW}, 
$\sum_{\alpha \in H_2^{\mathrm{eff}}(X,\Z)} n_{\beta_v +\alpha} q^\alpha$ is independent of $v \in N' = \Z$. 
Theorem \ref{thm:G-SYZ} asserts that the SYZ mirror of the $\widetilde{A}_{0}$ surface takes the form 
$$
\left\{uv = F^\open = \left(\sum_{\alpha \in H_2^{\mathrm{eff}}(X,\Z)} n_{\beta_0 +\alpha} q^\alpha\right) \sum_v z^{\beta_v - \beta_0} \right\} / \Z
$$
where by Lemma \ref{lem:G-z} the generator $1 \in \Z$ takes $z := z^{\beta_1 - \beta_0}$ to $z^{\beta_0 - \beta_{-1}} = q^{-1} z$ (since $C_0 = -2\beta_0 + \beta_{-1} + \beta_1 \in H_2(X^o)$). 
Moreover, by Proposition \ref{prop:G-F}, the generator  $1 \in \Z$ takes $F^\open$ to $q^{-1} z F^\open$.

The open Gromov--Witten invariants can be computed by Theorem \ref{thm:open-mir-thm}.   In this case, we have more efficient methods.  The result is that $F^\open$ admits a nice factorization
$$
F^\open = \prod_{i=1}^\infty(1+q^iz^{-1})\prod_{j=0}^\infty(1+q^jz).
$$

One way is to use the classification of admissible discs (an open version of admissible curves in Bryan--Leung \cite{BL}).  This method was used in \cite{LLW} in computing open Gromov-Witten invariants for $A_d$ surfaces.  For $A_d$ surfaces we obtain essentially the same formula as above except that it is a finite product.  

Another proof is that the series expansion of 
$$
\log \left( \prod_{i=1}^\infty(1+q^iz^{-1})\prod_{j=0}^\infty(1+q^jz) \right) = \sum_{i=1}^\infty \log (1+q^iz^{-1}) + \sum_{j=0}^\infty \log (1+q^jz)
$$
clearly has no term of the form $a \cdot q^C$ for $a \in \C^\times$ and $C \in H_2(X,\Z) - \{0\}$, and hence satisfies the Gross--Siebert normalization condition. 
By Theorem \ref{thm: GS norm} it must be the expression for the open Gromov--Witten generating function.

The RHS can be expressed as the following beautiful form by Jacobi triple product formula:
$$ \prod_{i=1}^\infty(1+q^iz^{-1})\prod_{j=0}^\infty(1+q^jz) = \prod_{k=1}^\infty\frac{1}{1-q^k}\cdot\sum_{l=-\infty}^\infty q^{\frac{l(l-1)}{2}}z^l = \frac{e^{\frac{\pi i \tau}{12}}}{\eta(\tau)}\cdot \vartheta \left( \zeta - \frac{\tau}{2}; \tau \right)$$
where $q:=e^{2\pi i \tau}$, $z:=e^{2\pi i \zeta}$, $\eta$ is the Dedekind eta function, and $\vartheta$ is the Jacobi theta function.  We conclude as follows.
\begin{Thm} \label{thm:A/Z}
The SYZ mirror of the local Calabi--Yau surface $\widetilde{A}_{0}$ is given by
$
uv = F^\open (z;\tau) 
$
where the open Gromov--Witten potential $F^\open$ is given by
\begin{equation} \label{eq:F-A/Z}
F^\open (z;\tau)= \frac{e^{\frac{\pi i \tau}{12}}}{\eta(\tau)}\cdot \vartheta \left( \zeta - \frac{\tau}{2}; \tau \right),
\end{equation}
where $\tau$ is the volume of the central fiber.  
More precisely, the SYZ mirror is a conic fibration over the elliptic curve $\C^\times/e^{2\pi i \tau\Z}$ which degenerates over the divisor $F^\open (z;\tau) =0$. 
\end{Thm}

By direct computation we can check that $\sum_v z^{\beta_v - \beta_0} = \vartheta\left(\zeta-\frac{\tau}{2}; \tau\right)$. 
By comparing the above $2$ expressions of $F^\open$, we obtain the following 24-th root of Yau--Zaslow formula explained in Section \ref{section: Intro}. 

\begin{Cor}[Root of Yau--Zaslow formula] \label{cor:rootYZ}
We have the following identity. 
$$\sum_{\alpha \in H_2^{\mathrm{eff}}(X,\Z)} n_{\beta_0 +\alpha} q^\alpha = \frac{e^{\frac{\pi i \tau}{12}}}{\eta(\tau)}.$$
\end{Cor}

The above formula follows from Theorem \ref{thm:open-mir-thm} on the relation between open Gromov--Witten invariants and mirror maps for toric Calabi--Yau manifolds of infinite-type.  On the other hand, one can also prove the formula by establishing canonical isomorphisms between the disc moduli $\mathcal{M}^{\widehat{A}_\infty}_1(\beta_k + \alpha)$ of the covering $\widehat{A}_{\infty}$ surface (where $\alpha$ is a chain of $(-2)$-curves) and the genus $0$ curve moduli $\mathcal{M}^{Y_\alpha}_{0,1}(s + \alpha)$ of a local surface $Y_\alpha$ containing the chain of $(-2)$-curves $\alpha$ and a $(-1)$-curve $s$ intersecting with $\alpha$ at one appropriate point (determined by the intersection between $\beta_k$ and $\alpha$).  By the result of Bryan--Leung \cite{BL}, the corresponding Gromov--Witten invariants are either $1$ or $0$ depending on whether $s + \alpha$ is admissible or not (which is a purely combinatorial condition, independent of the choice of the local surface $Y_\alpha$).  Thus the above root of Yau--Zaslow formula and the (primitive case of) Yau--Zaslow formula can both be deduced by the technique of \cite{BL}.

%\begin{proof}
%Applying the theory developed in the previous section together with the result \cite{LLW}, the SYZ mirror of $\widetilde{A}_{0}$ is given by the following: in $\mathrm{Spec} \left(\C[[z,z^{-1}]][[q]]\right)^\Z$ where $n \in \Z$ acts by $q^l z^k \mapsto q^l (q^nz)^k$.
%By direct calculation we have
%$$F^\open (z;\tau) = \prod_{k=1}^\infty\frac{1}{1-q^k}\cdot\sum_{l=-\infty}^\infty q^{\frac{l(l-1)}{2}}z^l.$$
%\end{proof}

A priori the generating function $F^\open(z;\tau)$ is a formal series in both $z$ and $q$.   
From the above expression, $F^\open(z;\tau)$ extends over the global moduli $\HH / \SL_2(\Z)$ as a holomorphic section of the principal polarization.  
%See the Appendix for the modular transformation property of $\eta$ and $\vartheta$. 
The geometric interpretation of the transformation property of $F^\open(z;\tau)$ by
$
S:=\begin{bmatrix}
                0   &  -1 \\
                1 &  0
\end{bmatrix} \in \SL_2(\Z)
$
(which is explicit since we have a formula of $F^\open(z;\tau)$) remains mysterious to us, since $S(\tau) = -1/\tau$ takes the large volume limit $q=0$ to $q=1$ where Gromov--Witten theory no longer makes sense.

By Theorem \ref{thm: GS norm}, we have the following identity on the Dedekind eta function. 
We shall see a generalization of the RHS in higher dimensions. 
\begin{Cor}We have 
 $$ 
\log  \frac{e^{\frac{\pi i \tau}{12}}}{\eta(\tau)}=\sum_{k,l=1}^\infty\frac{q^{kl}}{l}
=\sum_{k \geq 2} \frac{(-1)^k}{k} \sum_{\substack{(l_1,\ldots,l_k) \in (\Z\setminus\{0\})^k \\ \sum_{i=1}^kl_i=0}} q^{\sum_{i=1}^k \frac{l_i^2}{2}}.
$$
\end{Cor}
%\begin{proof}
%The first equality follows from the product expansion of $\eta(\tau)$, and the second from the normalization condition discussed above. 
%\end{proof}
%\end{comment}
%%%%%%%%%%%%%%%%%%%%%%%%%%%%%%%%%%%%%%%%%%%%%%%%%%%%%%%%%%%%%%%%%%%%%%%%%%%%%%%%%%%%%%%%%%%

\subsection{SYZ mirror of $\widetilde{A}_{d-1}$ surface}

The result in the previous subsection has a natural generalization to the local Calabi--Yau surface of type $\widetilde{A}_{d-1}$, 
namely the quotient $X^o/d\Z$, for an arbitrary $d \in \N$. 
It involves the following generalization of the Jacobi theta function to several variables:
$$
\vartheta_d(u_1,\ldots,u_d;\tau) := \sum_{n_1,\ldots,n_d 
= -\infty}^\infty \prod_{1 \leq i \leq j \leq d} e^{2\pi i n_i n_j \tau} \prod_{k=1}^d e^{2\pi i n_k u_k}.
$$
The above definition of the multivariable theta function $\vartheta_d$ can be found in, for instance, the Bellman's book \cite[Section 61]{Bel}.  
Here we take the convention $e^{2\pi i n_i n_j \tau}$ instead of $e^{\pi i n_i n_j \tau}$ for our convenience.
Recall that we have the K\"ahler parameters $q_1,\ldots,q_d$ of the surface $\widetilde{A}_{d-1}$. 

\begin{Thm} \label{thm:Ad}
The open Gromov--Witten potential $F^\open$ of the local Calabi--Yau surface $\widetilde{A}_{d-1}$ is given by 
$$ F^\open = \frac{e^{\frac{d \pi i \rho}{12}}}{\eta(\rho)^{d}} 
\sum_{p=0}^{d-1} r^{\frac{p^2}{2} - \frac{p^2}{2d}} \left(\prod_{i=1}^{d-1} Q_i \right)^{\frac{-p}{d}} 
\cdot \vartheta_{d-1} (T_1 - p \rho,\ldots, T_{d-1} - p\rho; \rho) \cdot \Theta_1
\begin{bmatrix}
                \frac{p}{d} \\
                {\displaystyle\sum_{i=1}^{d-1}} T_i -\frac{d \rho}{2}
\end{bmatrix}
(d\xi; d\rho).$$
Here $Q_j = \exp (2\pi i T_j) = \prod_{l=1}^j q_l$ for $1 \le j \le d-1$, and $r = e^{2\pi i \rho} = \prod_{l=1}^d q_l$ is the K\"ahler parameter of the elliptic fiber class. 
%It give a section of the $(d)$-polarization of the elliptic curve $\C^\times/r^\Z$. 
Therefore the SYZ mirror of $\widetilde{A}_{d-1}$ 
is a conic fibration over the elliptic curve $\C^\times/r^\Z$ which degenerates over the divisor $\{F^\open (z;\tau) =0\}$, which gives a $(d)$-polarization. 
\end{Thm}

\begin{proof}
By the same argument as the $\widetilde{A}_{1}$ case, we have 
$$
F^\open(z;q)=\prod_{k=1}^d\prod_{i=1}^\infty \left(1+r^i \left(\prod_{l=1}^{k}q_{l-1}z \right)^{-1}\right)\prod_{j=0}^\infty \left(1+r^j\prod_{l=1}^{k}q_{l-1}z \right). 
$$
Here $z$ is the coordinate of the torus $\C^\times/r^\Z$ for $r=q_1q_2\dots q_d$. 
In other words, the mirror is a conic fibration over the elliptic curve $\C^\times/r^\Z$ which degenerates over the points $F^\open (z;q) =0$. 
By a straightforward calculation, $F^\open(z;q)$ can be written as 
\begin{align}
F^\open(z;q) &= \prod_{k=1}^\infty\frac{1}{(1-r^k)^d}\cdot\prod_{l=1}^d\sum_{m=-\infty}^\infty r^{\frac{m(m-1)}{2}}(q_1q_2\dots q_{l-1}z)^m \notag \\
&:= \prod_{k=1}^\infty\frac{1}{(1-r^k)^d} \cdot F(z;q).\notag 
\end{align}
Let us write $r = e^{2\pi i \rho}$, $Q_j = \prod_{i=1}^j q_i = e^{2\pi i T_j}$ for $1 \le j \le d-1$ and $z = e^{2 \pi i \zeta}$, then we have 
\begin{align*}
F(z;q) =& \sum_{k=-\infty}^\infty \sum_{m_1 + \ldots + m_d = k} r^{\sum_{i=1}^d \frac{m_i(m_i-1)}{2}} Q_1^{m_2} \ldots Q_{d-1}^{m_d} z^k \\
=& \sum_{k=-\infty}^\infty \sum_{m_1,\ldots,m_{d-1} \in \Z} r^{\sum_{i \leq j} m_i m_j+\frac{k(k-1)}{2}-k \sum_{i=1}^{d-1} m_i} Q_1^{m_1} \ldots Q_{d-1}^{m_{d-1}} z^k \\
=& \sum_{p=0}^{d-1} \sum_{a=-\infty}^\infty \sum_{m_1,\ldots,m_{d-1} \in \Z} r^{\sum_{i \leq j} m_i m_j+\frac{(ad+p)(ad+p-1)}{2}-(ad+p) \sum_{i=1}^{d-1} m_i} Q_1^{m_1} \ldots Q_{d-1}^{m_{d-1}} z^{ad+p} \\
=& \sum_{p=0}^{d-1} \sum_{a=-\infty}^\infty \sum_{m_1,\ldots,m_{d-1} \in \Z} r^{\sum_{i \leq j} (m_i-a)(m_j-a) + \frac{(ad+p)(ad+p-1)-a^2d(d-1)}{2} -p \sum_{i=1}^{d-1} m_i} Q_1^{m_1} \ldots Q_{d-1}^{m_{d-1}} z^{ad+p} \\
=& \sum_{p=0}^{d-1} r^{\frac{p^2-p}{2}} \sum_{a=-\infty}^\infty r^{\frac{ad(a-1)}{2}+ap} (Q_1 \ldots Q_{d-1})^a z^{ad+p} \\
& \cdot \sum_{m_1,\ldots,m_{d-1} \in \Z} r^{\sum_{i \leq j} m_i m_j} (r^{-p} Q_1)^{m_1} \ldots (r^{-p} Q_{d-1})^{m_{d-1}} \\
=& \sum_{p=0}^{d-1} r^{\frac{p^2}{2}-\frac{p^2}{2d}} (Q_1\ldots Q_{d-1})^{-\frac{p}{d}} \cdot \theta_{d-1}(r^{-p} Q_1,\ldots,r^{-p}Q_{d-1};r) \\
&\cdot\sum_{a=-\infty}^\infty (r^d)^{\frac{1}{2} \left(a+\frac{p}{d}\right)^2} \left( z^d r^{-\frac{d}{2}} (Q_1\ldots Q_{d-1}) \right)^{a+\frac{p}{d}}.
\end{align*}
The last factor is exactly the Riemann theta function $\Theta_1
\begin{bmatrix}
                \frac{p}{d} \\
                \sum_{i=1}^{d-1} T_i -\frac{d \rho}{2}
\end{bmatrix}
(d\xi; d\rho)$. 
This completes the proof. 
\end{proof}

%The parameter $r = e^{2\pi i \rho}$ of the elliptic fiber class is the modular parameter, and it corresponds to the complex parameter of the base elliptic curve $\C^\times/r^\Z$ of the SYZ mirror (as a conic fibration).  

\begin{Cor}
The generic fiber of the elliptic surface $\widetilde{A}_{d-1}\rightarrow \D$ and the base elliptic curve of the SYZ mirror is mirror symmetric 
in the sense that the generic fiber of $\widetilde{A}_{d-1}\rightarrow \D$ has complexified K\"ahler moduli $r=q_1q_2\dots q_d$ 
and the base elliptic curve $\C^\times/r^\Z$ has complex moduli $r$. 
\end{Cor}

The following is a generalization of the root of Yau--Zaslow formula (Corollary \ref{cor:rootYZ}) to the $\widetilde{A}_{d-1}$ surface, 
which involves the multivariable theta function in addition to eta function.
\begin{Cor}
Let $\beta_p$ be the basic disc class corresponding to the primitive generator $(p,1) \in N$, $p=0,\ldots,d-1$. 
Then we have
$$ \sum_{\alpha\in H_2(X,\Z)} n_{\beta_p + \alpha} q^{\alpha} = \frac{e^{\frac{d \pi i \rho}{12} + \pi i p(p-1)\rho - 2\pi i \sum_{l=1}^{p-1} T_l}}{\eta(\rho)^{d}} 
\cdot \vartheta_{d-1} (T_1 - p \rho,\ldots, T_{d-1} - p\rho; \rho) $$
In particular if we restrict to $\alpha$ being multiples of the elliptic fiber class $F$, then
$$ \sum_{k=0}^\infty n_{\beta_p + kF} q^{k} = \left(\frac{q^{\frac{1}{24}}}{\eta(q)}\right)^d$$
which is independent of $p$.
\end{Cor}

Notice that Dedekind eta function appears frequently in Donaldson--Thomas invariants. 
In particular for the Calabi--Yau 3-fold $K_{\widetilde{A}_{d-1}}$ we have
$$ \sum_{k=0}^\infty \mathrm{DT}(kF,0) q^{k} = \left(\frac{q^{\frac{1}{24}}}{\eta(q)}\right)^d$$ %= \sum_{k=0}^\infty n_{\beta_p + kF} q^{k}
where $\mathrm{DT}(kF,0)$ is the Donaldson--Thomas invariants, the virtual count of subschemes $Z$ of $K_{\widetilde{A}_{d-1}}$ with $[Z]=kF$ and $\chi(\mathcal{O}_Z) = 0$. 
The localization technique relates this with the Euler characteristic of $\mathrm{Hilb}^k(K_{\widetilde{A}_{d-1}})$ \cite{Got}. 

\section{Local Calabi--Yau 3-folds of type $\widetilde{A}$} \label{section: 3-fold}
Now we proceed to derive the SYZ mirror of a crepant resolution $X_{(p,q)}$ of the fiber product of the $\widetilde{A}_{p-1}$ surface and the $\widetilde{A}_{q-1}$ surface over $\D$. 
Here we assume the singular fibers of both surfaces occur at $0 \in \D$. 
The fiber product has conifold singularities at $(a,b)$ where $a$ and $b$ are the singular points of the fibers at $0$ of the two surfaces.  
Taking a consistent crepant resolution of all these singularities, we obtain a local Calabi--Yau 3-fold, 
which can be regarded as a partial compactification of (an open subset of) the resolved orbifolded conifold $\widehat{O_{p,q}}$ (see \cite{KL} for instance). 
%This is summarized by the following diagram (also compare this with Figure \ref{fig:FiberProd}). 
$$
\xymatrix{
X_{(p,q)} \ar[r]& \widetilde{A}_{p-1} \times_\D \widetilde{A}_{q-1} \ar[d] \ar[r] \ar@{}[dr]|\Box &\widetilde{A}_{q-1} \ar[d]  \\
& \widetilde{A}_{p-1} \ar[r] & \D. 
}
$$

In \cite{HIV}, Hollowood--Iqbal--Vafa constructed the local Calabi--Yau 3-fold $X_{(1,1)}$ in a heuristic way and argued that its mirror curve is given by 
%\footnote{
%In \cite{HIV}, the authors construct a Calabi--Yau 3-fold by intuitively gluing {\it the infinity} of a conic double conic fibration over $\C^\times$. 
%We identify it with $X_{(1,1)}$, giving a rigorous mathematical definition.}
$$
F(X,Y)=\sum_{k,l \in \Z} e^{(\pi i k(k-1)\rho+\pi i l (l-1)\tau+2\pi i kl \sigma)}X^kY^l=0
$$
in a 2-dimensional torus $(\C^\times)^2/\Z^2$, where $\rho, \tau$ and $\sigma$ represent the complexified K\"ahler parameters of $X_{(1,1)}$.
%such that their imaginary parts are $t_1+t_2, t_2+t_3$ and $t_2$ respectively, and therefore 
%$\Omega:=\begin{bmatrix}
%                \rho   &  \sigma  \\
%                 \sigma& \tau  \\
%\end{bmatrix}$ belongs to the Siegel upper-half plane $\mathfrak{H}_2$.
They gave $3$ supporting arguments for this mirror proposal, namely matrix models, geometric engineering and instanton calculus.

%In order to make sense of the infinite sum $F(X,Y)$, we think of the mirror curve living in a $2$-dimensional torus $T^2$ in accordance with the gluing on the mirror side.  Namely, we consider $X,Y$ as coordinates of the torus $T^2=(\C^\times)^2/\Z^2$, where 
%$$
%(X,Y) \mapsto  (e^{2\pi i \rho}X,e^{2\pi i \sigma}Y), \ \ \ (X,Y)\mapsto (e^{2\pi i \sigma}X,e^{2\pi i \tau}Y). 
%$$
%In this way we can identify $F(X,Y)$ with the Riemann theta function. 
%\begin{Thm}(Doran--Iqbal--Kanazawa \cite{DIK})
%The pair $(T^2,[C])$ is a principally polarized abelian surface.  
%Moreover, for a generic choice of the complexified K\"ahler parameters, it is of the form $(\mathrm{Jac}(C),[C])$, 
%where $\mathrm{Jac}(C)$ is the Jacobian of $C$ and $[C]$ is the Theta divisor. 
%In particular $C$ is a genus $2$ curve. 
%\end{Thm}

In the following we will derive the mirror of $X_{(1,1)}$, and more generally mirrors of $X_{(p,q)}$, via the mathematically rigorous SYZ program. 
We have the same form of mirror as above. 
A crucial advantage of our approach is that additionally our mirrors involve the generating functions of open Gromov--Witten invariants, 
which are analogs of the Dedekind eta function and multivariable theta functions in the surface case.  
%We will conjecture that $\Delta(q)$ is essentially the celebrated Igusa cusp form $\chi_{10}(\Omega)$ of weight $10$. 

%%%%%%%%%%%%%%%%%%%%%%%%%%%%%%%%%%%%%%%%%%%%%%%%%%%%%%%%%%%%%%%%%%%%%%%%%%%%%%%%%%%%%%%%%%

\subsection{Toric geometry}
Similar to the previous section, the geometry can be made more explicit by taking a cover by a toric Calabi--Yau 3-fold of infinite-type. 
Consider the fan $\Sigma_0$ consisting of maximal cones
$$
\langle (i,j,1), (i+1,j,1),(i,j+1,1),(i+1,j+1,1)\rangle, \ \ \ i,j\in\Z,
$$
which is a fiber product of $2$ copies of the fan of the $\widehat{A_{\infty}}$ surface. 
In particular $X_{\Sigma_0}$ is a fiber product of $2$ copies of the $\widehat{A_\infty}$ surface. 
It admits an action by $\Z^2$: $(k,l) \cdot (a,b,c) = (a+k,b+l,c)$ on $N$. 
A crepant resolution is obtained by refining each maximal cone, which is a cone over a square, into $2$ triangles. 
Note that we have $2$ choices for each of the squares, which are related by a flop. 
We make a choice which is invariant under $p\Z \times q\Z$, and the corresponding fan is denoted by $\Sigma$. 
Then $X_{(p,q)}$ is obtained as $X^o_\Sigma / (p\Z \times q\Z)$, where $X^o_\Sigma$ is a toric neighborhood of the toric divisors where $\Z^2$ acts freely.  The natural morphism $\overline{\phi}:X_{(p,q)}\rightarrow \D$ is an abelian surface fibration such that $\overline{\phi}^{-1}(t)=(\C^\times/t^{p\Z})\times(\C^\times/t^{q\Z})$ for $t\ne 0$. 
%This morphism will be useful to construct Landau--Ginzburg mirrors of Riemann surfaces (see Section \ref{}). 

As was discussed in Section \ref{subsection: group action}, in order to have a $\Z^2$-invariant metric, where $\Z^2\cong p\Z \times q\Z$, 
we need to make sure $\Sigma$ has a dual polytope which is invariant under $\Z^2$ up to translation. 
Unlike the situation in the surface case, this imposes a consistency condition on the choice of crepant resolutions of the conifold points 
(see Example \ref{ex:no-metric} for a counterexample that does not admit a $G$-invariant metric). 
Here we simply take the refinement of
$$
\langle (i,j,1), (i+1,j,1),(i,j+1,1),(i+1,j+1,1)\rangle, \ \ \ i,j\in\Z
$$ 
into the $2$ triangles (Figure \ref{fig:FiberProd})
$$
\langle (i,j,1),(i+1,j,1),(i,j+1,1)\rangle, \ \ \ \langle (i+1,j,1),(i,j+1,1),(i+1,j+1,1)\rangle. 
$$ 
\begin{figure}[htbp]
 \begin{center} 
  \includegraphics[width=75mm]{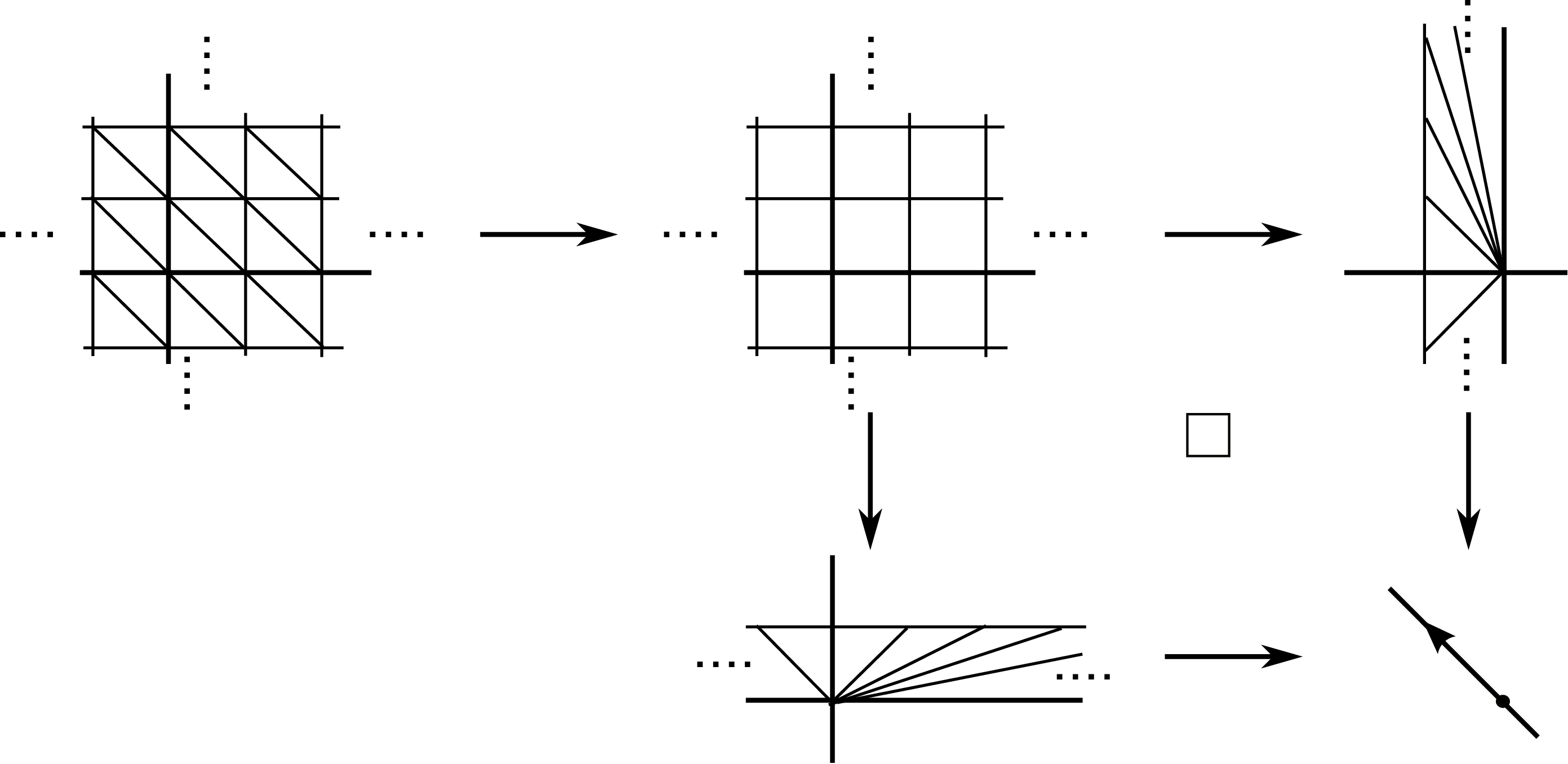}
 \end{center}
 \caption{Toric picture of $\widehat{A_{\infty}} \times_\C \widehat{A_{\infty}}$ and its crepant resolution $X_\Sigma$}
\label{fig:FiberProd}
\end{figure}
It is easy to check that this choice of $\Sigma$ has a $\Z^2$-invariant dual polytope.

\begin{Lem} \label{lem:invP-3fold}
Take $c_{k,l} = -k(k-1)-l(l-1)-kl$ for $k,l\in\Z$.  Then the polytope 
$$P := \bigcap_{(k,l) \in \Z^2} \{ (y_1,y_2,y_3) \in \R^3 \ | \ k y_1 + l y_2 + y_3 \geq c_{k,l} \}$$
is invariant under the $\Z^2$-action generated by 
\begin{align}
(1,0) \cdot (y_1,y_2,y_3) = (y_1,y_2,y_3-y_1) + (-2,-1,2), \notag \\
(0,1) \cdot (y_1,y_2,y_3) = (y_1,y_2,y_3-y_2) + (-1,-2,2). \notag 
\end{align}
\end{Lem}

We label the toric invariant curves as follows.

\begin{Def} \label{def:C^i}
We label the irreducible toric rational curve corresponding to the 2-dimensional cone as follows (the left figure of Figure \ref{fig:Curves}):
\begin{align}
C^1_{(a,b)}&:=\R_{\geq 0} \Conv(\{(a+1,b,1), (a,b+1,1) \}), \notag \\
C^2_{(a,b)}&:=\R_{\geq 0} \Conv(\{(a,b,1), (a,b+1,1) \}), \notag \\
C^3_{(a,b)}&:=\R_{\geq 0} \Conv(\{(a,b,1), (a+1,b,1) \}).\notag
\end{align} 
\end{Def}

\begin{figure}[htbp]
 \begin{center} 
  \includegraphics[width=50mm]{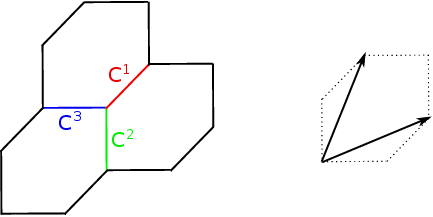}
 \end{center}
 \caption{Toric web diagram and vectors $[\tau,\sigma]^t$ and $[\sigma,\rho]^t$.}
\label{fig:Curves}
\end{figure}

It is straightforward to get the following relations in $H_2(X_{\Sigma},\Z)$ for all $a,b \in \Z$.
$$
C^1_{(a-1,b)}+C^3_{(a-1,b)}=C^1_{(a,b-1)}+C^3_{(a,b)}, \ \ \ 
C^1_{(a-1,b)}+C^2_{(a,b)}=C^1_{(a,b-1)}+C^2_{(a,b-1)}. 
$$

\begin{Lem} \label{Kahler dim}
$H_2(X_{(p,q)},\Z)$ has rank $pq+2$. 
\end{Lem}
\begin{proof}
There are $3pq$ toric rational curves in the fundamental domain of the toric web diagram, namely $C^1_{(a,b)}, C^2_{(a,b)}$ and $C^3_{(a,b)}$ for $0\le a \le p-1$ and $0 \le b \le q-1$. 
We have $2pq$ relations as above, while $2$ of the relations are abundant due to the periodicity. Therefore the rank is $3pq-(2pq-2)=pq+2$. 
\end{proof}

The following relations would be useful to compute the explicit expression of the SYZ mirror.

\begin{Lem} \label{prop: horizontal vertical}
The sum $\sum_{a=0}^{p-1}C^1_{(a,b)}$ is independent of $0\le a \le p-1$ in $H_2(X_{\Sigma},\Z)$. 
Similarly, $\sum_{b=0}^{q-1}C^1_{(a,b)}$ is independent of $0 \le b \le q-1$ in $H_2(X_{\Sigma},\Z)$. 
\end{Lem}
\begin{proof}
By the periodicity, we have $C^1_{(-1,b)}=C^1_{(p-1,b)}$ and thus the sum of the above relations
$$
\sum_{a=0}^{p-1} \left(C^1_{(a-1,b)}+C^3_{(a-1,b)}\right)=\sum_{a=0}^{p-1}\left(C^1_{(a,b-1)}+C^3_{(a,b)}\right)
$$
simplifies to
$$
\sum_{a=0}^{p-1}C^1_{(a,b)}=\sum_{a=0}^{p-1}C^1_{(a,b-1)}.
$$ 
The second assertion follows similarly. 
\end{proof}

It is worth noting that the above toric construction provides a mathematical foundation of the local Calabi--Yau 3-fold $X_{1,1}$ 
heuristically discussed by Hollowood--Iqbal--Vafa \cite{HIV}. 
In particular the topological vertex technique can be justified as $X_{p,q}$ admits the action of the subgroup $\C^\times \times (S^1)^2\subset (\C^\times)^3$.  

%%%%%%%%%%%%%%%%%%%%%%%%%%%%%%%%%%%%%%%%%%%%%%%%%%%%%%%%%%%%%%%%%%%%%%%%%%%%%%%%%%%%%%%%%%

\subsection{SYZ mirror of $X_{(1,1)}$} 
Let us first consider the local Calabi--Yau 3-fold $X_{(1,1)}$. 
By Lemma \ref{Kahler dim}, $\dim H_2(X_{1,1},\R) = 3$. 
The cone of effective curves is given by $\R_{\geq 0} \{C^1,C^2,C^3\}$ where $C^i = C^i_{(a,b)}$ for any $a,b$, where $C^i_{(a,b)}$ are given in Definition \ref{def:C^i}. 
For the purpose of modularity, we define
$$
C_\tau = C^1 + C^2, \ \ \ C_\rho = C^1 + C^3, \ \ \ C_\sigma = C^1
$$
and let $q_\tau = e^{2\pi i \tau}$, $q_\rho= e^{2\pi i \rho}$ and $q_\sigma= e^{2\pi i \sigma}$ be the corresponding K\"ahler parameters respectively 
(the right figure of Figure \ref{fig:Curves}). 
Then we have $q_\tau = q_1 q_2$, $q_\rho = q_1 q_3$ and $q_\sigma = q_1$. 
We will show that $\Omega:=\begin{bmatrix}
                \tau   &  \sigma  \\
                 \sigma& \rho  \\
\end{bmatrix}$ serves as the period matrix of the base abelian surface of the SYZ mirror of $X_{(1,1)}$.  From now on we shall use both $q = (q_\tau,q_\rho,q_\sigma)$ and $\Omega$ to refer to the mirror complex parameters.
%, where
%$$ \mathfrak{H}_m:=\{\Omega \in \mathrm{M}_m(\C) \ | \ \Omega^t=\Omega, \ \Omega^{\mathrm{Im}} \textrm{ is positive definite} \}.  $$ 
%The defintion is in the appendix

\begin{Lem} \label{lem: omega}
 The matrix $\Omega$ belongs to the Siegel upper half-plane $\mathfrak{H}_2$, i.e. $\mathrm{Im}\Omega>0$. 
%Let $\mathrm{Im}\rho$, $\mathrm{Im}\tau$, $\mathrm{Im}\sigma$ be the symplectic areas of the curve classes $C_\rho, C_\tau, C_\sigma$ respectively. 
%Then
%$\Omega^{\mathrm{Im}}=\begin{bmatrix}
%                \tau^{\mathrm{Im}}   &  \sigma^{\mathrm{Im}}  \\
%                 \sigma^{\mathrm{Im}}& \rho^{\mathrm{Im}}  \\
%\end{bmatrix}$ is positive definite.
\end{Lem}
\begin{proof}
Recall that $\mathrm{Im}\rho$, $\mathrm{Im}\tau$ and $\mathrm{Im}\sigma$ serve as the symplectic areas of the curve classes 
$C^1 + C^2, C^1 + C^3$ and $C^1$ respectively. 
%Then it follows immediately from $\tau^{\mathrm{Im}} > \sigma^{\mathrm{Im}}$ and $\rho^{\mathrm{Im}} > \sigma^{\mathrm{Im}}$ since $ C_\tau = C_\sigma + C^2$ and $C_\rho = C_\sigma + C^3$.
Then it follows immediately that $\mathrm{Im}\tau+\mathrm{Im}\rho>0$ and $\mathrm{Im}\tau\mathrm{Im}\rho-(\mathrm{Im}\sigma)^2>0$. 
\end{proof}

Now we combine Theorem \ref{thm:SYZ} and \ref{thm: GS norm} to obtain the SYZ mirror of $X_{(1,1)}$.
\begin{Thm}  \label{thm:CY3/Z^2}
The SYZ mirror of $X_{(1,1)}$ is $uv = F^\open(z_1,z_2;q)$ where
$$
F^\open(z_1,z_2;q) = \Delta(\Omega) \cdot 
\Theta_2
\begin{bmatrix}
                0 \\
                (- \frac{\tau}{2},- \frac{\rho}{2})
\end{bmatrix}
\left(\zeta_1, \zeta_2; \Omega\right). 
$$
Here $z_i = e^{2\pi i \zeta_i}$, $\Theta_2$ is the genus $2$ Riemann theta function, and
\begin{equation} \label{eq:Delta}
\Delta(\Omega)
 = \exp \left(\sum_{j \geq 2} \frac{(-1)^j }{j } 
\sum_{\substack{({\bf l}_i=(l_i^1, l_i^2)\in \Z^{2} \setminus 0)_{i=1}^j  \\ \textrm{with } \sum_{i=1}^j  {\bf l}_i = 0}}
\exp\left(\sum_{k=1}^j \pi i {\bf l}_k\cdot \Omega \cdot {\bf l}^T_k\right) \right).
\end{equation}
%The divisor $\{F^\open=0\}$ endows the ambient abelian surface with the principal polarization. 
Thus the SYZ mirror is a conic fibration over the abelian surface with period $\Omega$  
which degenerates over the genus $2$ curve defined by $F^\open=0$, which give a principal polarization. 
\end{Thm}

\begin{proof}
By Theorem \ref{thm:SYZ} and $\Z^2$-symmetry, we have
$$ F^\open = \Delta(q) \sum_{(p,q) \in \Z^2} q^{C_{(p,q)}} z_1^p z_2^q  $$  
where $\Delta(q) = \sum_{\alpha \in H_2^\mathrm{eff} (X,\Z)} n_{\beta_{(0,1)} + \alpha} q^\alpha$.
Now we need to compute $q^{C_{(p,q)}}$, where
$$
C_{(p,q)} = \beta_{(p,q,1)} - \beta_{(0,0,1)} -  p \cdot (\beta_{(1,0,1)} -  \beta_{(0,0,1)}) - q \cdot (\beta_{(0,1,1)} -  \beta_{(0,0,1)}).
$$
We claim that
$$C_{(p,q)} = \left(\sum_{k=1}^{p-1} k \, (C^1_{p-k,q-1} + C^2_{p-k,q-1}) \right) + p \, C^1_{0,q-1} + \left( \sum_{k=1}^{q-1}  (k \, C^3_{0,q-k} + 
(p+k) C^1_{0,q-k-1}) \right). $$
(See Figure \ref{fig:HoneyComb} for an example.)
One can directly verify this by the intersection numbers:
$$
C_{(p,q)} \cdot D_{(p,q)} = 1, \ \ C_{(p,q)} \cdot D_{(0,0)} = p+q-1, \ \ C_{(p,q)} \cdot D_{(1,0)} = -p, \ \ C_{(p,q)} \cdot D_{(0,1)} = -q,
$$ 
and the intersection number of $C_{(p,q)}$ with any other toric prime divisor is $0$.  

Due to the $\Z^2$-symmetry, the K\"ahler parameters corresponding to $C^i_{(a,b)}$ are independent of $(a,b) \in \Z^2$ and are denoted by $q_i$ for $i=1,2,3$. 
Then we see that
$$
q^{C_{(p,q)}} = q_1^{\frac{(p+q)(p+q-1)}{2}}q_2^{\frac{p(p-1)}{2}}q_3^{\frac{q(q-1)}{2}} = q_\rho^{\frac{p(p-1)}{2}} q_\tau^{\frac{q(q-1)}{2}} q_\sigma^{pq}
$$  
where $q_\rho = q_1 q_2$, $q_\tau = q_1 q_3$ and $q_\sigma=q_1$.  

By Theorem \ref{thm: GS norm}, we observe that 
$$
\log F^\open = \sum_{k=1}^\infty \frac{(-1)^{k-1}}{k} (F^\open - 1)^k
$$
has no $z^0$-term.  
Thus $\log \Delta(q)$ equals to the $z^0$-term of $- \log \left( \sum_{I} q^{C_I} z^I \right)$.  It follows from direct computation that $\Delta$ has the given expression. 
Since $F^\open$ is, up to a nowhere-zero multiple, the genus $2$ Riemann theta function, it endows the ambient abelian surface the principal polarization. 
\end{proof}

\begin{figure}[htbp]
 \begin{center} 
  \includegraphics[width=40mm]{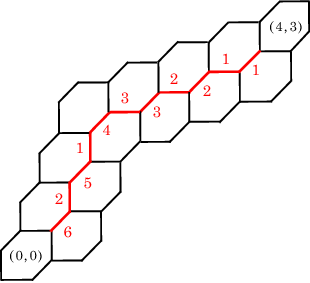}
 \end{center}
 \caption{The toric web diagram where the hexagons are the images of the toric prime divisors. 
 The numbered lines show the curve $C_{(p,q)}$ which connect the divisor $D_{(p,q)}$ to $D_{(0,0)}$.}
\label{fig:HoneyComb}
\end{figure}

%The SYZ mirror is a conic fibration over an abelian surface whose discriminant locus is the genus $2$ curve given $F^\open=0$ in the surface. 
%The genus 2 curve endows the abelian surface a principal polarization.  
%The variable $v$ is the fiber coordinate of the line bundle corresponding to the Riemann theta function, so that the equation is well-defined.

\begin{Lem}
The series 
$$
\log(\Delta(\Omega)):=\sum_{j \geq 2} \frac{(-1)^j }{j } 
\sum_{\substack{({\bf l}_i=(l_i^1, l_i^2)\in \Z^{2} \setminus 0)_{i=1}^j  \\ \textrm{with } \sum_{i=1}^j  {\bf l}_i = 0}}
\exp\left(\sum_{k=1}^j \pi i {\bf l}_k\cdot \Omega \cdot {\bf l}^T_k\right)
$$
converges absolutely and uniformly on compact subsets of the Siegel upper half-plane $\mathfrak{H}_2$. 
\end{Lem}

\begin{proof}
%First we note that 
%$\left| \exp\left( \pi i {\bf l}_k\cdot \Omega \cdot {\bf l}^T_k \right) \right| = \exp \left( - \pi {\bf l}_k\cdot (\mathrm{Im} \Omega) \cdot {\bf l}^T_k \right) $
Since $\mathrm{Im} \Omega>0$, we have
%\begin{align*}
{\small
$$
\sum_{\substack{({\bf l}_i=(l_i^1, l_i^2)\in \Z^{2} \setminus 0)_{i=1}^p \\ \textrm{with } \sum_{i=1}^p {\bf l}_i = 0}}
\left| \exp\left(\sum_{k=1}^p\pi i {\bf l}_k\cdot \Omega \cdot {\bf l}^T_k\right) \right| 
\leq \sum_{({\bf l}_i \in \Z^2)_{i=1}^p} \prod_{k=1}^p \exp\left(-\pi {\bf l}_k\cdot \mathrm{Im} \Omega \cdot {\bf l}^T_k\right)
= \left(\sum_{{\bf l} \in \Z^2} \exp\left(-\pi {\bf l} \cdot \mathrm{Im} \Omega \cdot {\bf l}^T\right) \right)^p.
$$
}
%\end{align*}
There exists an invertible matrix $A \in \mathrm{GL}(2,\R)$ such that 
$\mathrm{Im} \Omega = A \cdot \left( \begin{array}{cc}
r_1 & 0 \\
0 & r_2
\end{array}
\right) A^T \textrm{ for } r_1,r_2>0. $
Then there exist $a_2 > a_1 > 0$ such that $a_1 \|v \| \leq \|v \cdot A\| \leq a_2 \|v\|$ for all $v \in \R^2$.  If $\Omega$ lies in a compact set, then $a_1,a_2,r_1,r_2$ lie in compact sets.
For $\| {\bf l} \| \gg 0$,
$$ \exp \left( - \pi {\bf l} \cdot (\mathrm{Im} \Omega) \cdot {\bf l}^T \right) = \exp -\pi \left(r_1 (\widetilde{l}^1)^2 + r_2 (\widetilde{l}^2)^2 \right) \leq \exp -\pi r a_1 \| {\bf l} \|^2 \leq \exp -(|l^1| + |l^2|) $$
where ${\bf l} = (l^1,l^2)$, ${\bf l} \cdot A = (\widetilde{l}^1,\widetilde{l}^2)$ and $r = \min\{r_1,r_2\}$.  Thus there exists a large $L$ such that
$$ \left(\sum_{|l^1| > L} \sum_{|l^2| > L}  \left| \exp\left( \pi i {\bf l}\cdot \Omega \cdot {\bf l}^T \right) \right| \right)^p \leq \left(\sum_{|l^1| > L} \sum_{|l^2| > L} \exp -(|l^1| + |l^2|)\right)^p = \left(\sum_{|l| > L}\exp -|l| \right)^{2p}$$
for all $\Omega$ lying in the compact set.  We have
$$
\log \left( 1 + \left( 2 \sum_{l>L}\exp (-l)\right)^2 \right) =  \sum_{p \geq 1} \frac{(-1)^p}{p} \left(2 \sum_{l>L}\exp (-l) \right)^{2p}
$$
for sufficiently large $L$ such that $\sum_{l>L}\exp (-l) < \frac{1}{2}$, and hence the RHS is absolutely convergent.
\end{proof}

The function $\Delta$ is an analog of the Dedekind eta function in the surface case, 
and Equation \eqref{eq:Delta} is a higher dimensional analog of the root of Yau--Zaslow formula (Corollary \ref{cor:rootYZ}). 
It satisfies the following modularity properties.

\begin{Prop} \label{prop:modular}
For $A \in \GL(2,\Z)$ and $ B \in \mathrm{M}(2,\Z)$, we have
$$
\Delta(A \Omega A^t) = \Delta(\Omega),\ \ \ \Delta(\Omega + B)= \Delta(\Omega). 
$$
Moreover, on the diagonal $\sigma=0$, we have 
$$
\left.\frac{q^{\frac{1}{24}}_\rho q^{\frac{1}{24}}_\tau}{\Delta(\Omega)}\right|_{\Omega \mapsto -\Omega^{-1}}
= \sqrt{\det (-i \Omega)} \frac{q^{\frac{1}{24}}_\rho q^{\frac{1}{24}}_\tau}{\Delta(\Omega)}. 
$$
\end{Prop}

\begin{proof}
The first and second properties follow easily from Equation \eqref{eq:Delta}. 
The third property follows from the fact that $\Delta|_{\sigma = 0} = \frac{q^{\frac{1}{24}}_\rho q^{\frac{1}{24}}_\tau}{\eta(\tau)\eta(\rho)}$, since $X_{(1,1)}$ degenerates to the fiber product of two $\tilde{A}_1$ surfaces. 
\end{proof}

Unfortunately, at this point we do not know whether or not the third property extends to the off-diagonal $\sigma \not= 0$. 
We anticipate that the function $\Delta$ is closely related to Siegel modular forms of genus $2$. 

It is interesting to note the following.
Consider Schoen's Calabi--Yau 3-fold $S_1 \times_{\PP^1}S_2$, which is the fiber product of $2$ generic elliptic rational surfaces $S_i\rightarrow \PP^1$ $(i=1,2)$. 
Then a degeneration argument shows the Gromov--Witten potential of $S_1 \times_{\PP^1}S_2$ 
equals to that of the product $E \times K3$ of an elliptic curve $E$ and a K3 surface, 
and is conjecturally given by $1/\chi_{10}$ \cite{OP}\footnote{We are grateful to G. Oberdieck for useful communication.}, 
where $\chi_{10}$ is the Igusa cusp form of weight $10$.   
The local Calabi--Yau 3-fold $X_{1,1}$ serves as a local model of a conifold transition of $S_1 \times_{\PP^1}S_2$ which is a crepant resolution of the singular 3-fold $S_1 \times_{\PP^1}S_1$. 
The Igusa cusp form $\chi_{10}$ has the asymptotic behavior similar to $\Delta^{24}$ on the diagonal $\sigma=0$ \cite{Igu}. 
Namely, the cusp form $\chi_{10}(\Omega)$ for $\sigma\rightarrow 0$ is given by
$$
\chi_{10}(\Omega)=(\eta(\tau)\eta(\rho))^{24}(\pi \sigma)^2 +O(\sigma^4). \label{eq: asymptotic}
$$

\subsection{SYZ mirror of $X_{(p,q)}$} \label{section: $X_{p,q}$}
%In this subsection we consider more generally the fiber product between a local $\widetilde{A}_p$ surface and a local $\widetilde{A}_q$ surface over a disc $D$.  It can be defined as the quotient of the toric variety $X_\sigma$ by $\Z^2$, where $\sigma \in \R^3$ is the fan obtained by coning of the squares $\{[i,i+1] \times [j,j+1] \times \{1\} \subset \R^2: i,j \in \Z \}$, and $(a,b) \in \Z^2$ acts on $\sigma$ by $(a,b) \cdot (x,y,z) := (x + a, y + b, z)$.  We can take a refinement $\Sigma$ which consists of maximal cones $\R_{\geq 0} \langle (i,j,1), (i+1,j,1),(i,j+1,1) \rangle$ and $\R_{\geq 0} \langle (i,j,1), (i-1,j,1), (i,j-1,1) \rangle$ for all $i,j \in \Z$.  Then $X_\Sigma$ is a resolution of $X_\sigma$.
Next we consider the local Calabi--Yau 3-fold $X_{(p,q)}$ for $(p,q) \in \N^2$. 
In general we do not have a choice of basis of $H_2(X,\Z)$ such that every effective class is a non-negative linear combination of the basic elements. 
Instead we use the generators
$$
C_{\sigma,(k,l)} = C_{(k,l)}^{1}, \ \ C_{\tau,(k,l)} = C_{(k,l)}^{1} + C_{(k,l)}^{2}, \ \ C_{\rho,(k,l)} = C_{(k,l)}^{1} + C_{(k,l)}^{3}
$$
for $k \in \Z_p$, $l \in \Z_q$, and keep in mind that there are $2pq-2$ relations among them.

%Theorem \ref{thm:CY3/Z^2} has a natural generalization as follows: 

\begin{Thm} \label{mainthm: 3-fold}
The SYZ mirror of $X_{(p,q)}$ is $uv = F^{\open}$ where the open Gromov--Witten potential $F^{\open}$ is given by 
$$\sum_{a,b = 0}^{p-1,q-1} K_{a,b} \cdot \Delta_{a,b} \cdot \Theta_2
\begin{bmatrix}
                \left(\frac{a}{p}, \frac{b}{q}\right)  \\
								\left(\frac{-p \tau}{2} + \displaystyle{\sum_{k=0}^{p-1}} k \tau_{(-1-k,0)}, \frac{-q \rho}{2} + \sum_{l=0}^{q-1} l \rho_{(0,-1-l)}\right) 
\end{bmatrix}
\left(p \cdot \zeta_1, q \cdot \zeta_2; 
\begin{bmatrix}
                p \tau & \sigma \\
               \sigma & q \rho 
\end{bmatrix}\right),
$$
\begin{align*}
\Delta_{a,b} :=& \sum_{\alpha} n_{\beta_{(a,b)} + \alpha} Q^{\alpha}, \\
K_{a,b} :=& Q_\tau^{-\frac{a^2}{2p} + \frac{a}{2}} Q_\sigma^{\frac{-ab}{pq}} Q_\rho^{-\frac{b^2}{2q}+\frac{b}{2}} \left( \prod_{k=0}^{p-1} Q_{\tau,(-1-k,0)}^k \right)^{\frac{-a}{p}} \left( \prod_{l=0}^{q-1} Q_{\rho,(0,-1-l)}^l \right)^{\frac{-b}{q}}\\
& \cdot \left( \prod_{k=0}^{a-1} Q_{\tau, (a-1-k,b)}^k \right) \cdot \left( \prod_{l=0}^{b-1} Q_{\rho, (0,b-1-l)}^l \right) \cdot \left( \prod_{l=0}^{b-1} Q_{\sigma,(0,l)} \right)^{a}.
\end{align*}
In the above expression, we use 
$$
\tau := \sum_{k=0}^{p-1} \tau_{(k,b)}, \ \ \rho := \sum_{l=0}^{p-1} \rho_{(a,l)}, \ \ 
\sigma := \sum_{k,l=0}^{p-1,q-1} \sigma_{(k,l)} = q \sum_{k=0}^{p-1} \sigma_{(k,b)} = p \sum_{l=0}^{q-1} \sigma_{(a,l)}
$$
which are independent of $a$ and $b$ by Lemma \ref{prop: horizontal vertical} (see also Figure \ref{fig:MirrorCorresp3}). 
The divisor $\{F^\open=0\}$ defines a genus $pq+1$ curve and endows the ambient abelian surface with the $(p,q)$-polarization. 
\end{Thm}

\begin{proof}
First, we observe that $\Omega=\begin{bmatrix}
                p \tau & \sigma \\
               \sigma & q \rho 
\end{bmatrix}$ lies in the Siegel upper half-plane as is shown in Lemma \ref{lem: omega}.   
The proof is a complicated version of that of Theorem \ref{thm:CY3/Z^2} and so we shall be brief.  By Theorem \ref{thm:SYZ} and $(p\Z \times q\Z)$-symmetry,
$$F^\open = \sum_{a,b=0}^{p-1,q-1} \Delta_{a,b} \sum_{(c,d) \in \Z^2} Q^{C_{(cp+a,dq+b)}} \cdot z_1^{cp+a} z_2^{dq+b}
$$
where $\Delta_{a,b} = \sum_{\alpha} n_{\beta_{(a,b)} + \alpha} Q^{\alpha}$ are generating functions of open Gromov--Witten invariants, and
$$
C_{(c,d)} = \sum_{k=1}^{c-1} k C_{\tau,(c-1-k,d)} + c \sum_{l=0}^{d-1} C_{\sigma,(0,d-1-l)} + \sum_{l=1}^{d-1} l C_{\rho,(0,d-1-l)}.
$$
Set
\begin{align*}
\epsilon_k = \left\lfloor \frac{a-1-k}{p} \right\rfloor = \left\{ \begin{array}{ll}
0 & (0 \le k \le a-1)\\
-1 & (a \le k \le p-1)\end{array}
\right.,\ \ 
\delta_l = \left\lfloor \frac{b-1-l}{q} \right\rfloor = \left\{ \begin{array}{ll}
0 & (0 \le k \le b-1)\\
-1 & (b \le k \le q-1).\end{array}
\right. \\
\end{align*}
Using $\Z^2$-symmetry, we can check that
%$C_{\tau,(a+cp,b+dq)} = C_{\tau,(a,b)}$ for any $c,d\in\Z$, and the same holds for $C_{\rho,(a,b)}$ and $C_{\sigma,(a,b)}$.  
$C_{(cp+a,dq+b)}$ equals to
{\small
\begin{align*}
%& \sum_{k=1}^{cp+a-1} k C_{\tau,(a-1-k,b)} + (cp+a) \sum_{l=0}^{dq+b-1} C_{\sigma,(0,b-1-l)} + \sum_{l=1}^{dq+b-1} l C_{\rho,(0,b-1-l)}\\
%=& \sum_{k=1}^{p-1} \left( \frac{p (c + \epsilon_{a,k})(c+1+\epsilon_{a,k})}{2} + k(c+1+\epsilon_{a,k}) \right) C_{\tau,(a-1-k,b)} \\
%&+ (cp+a)\left(\frac{d C_\sigma}{p} + \sum_{l=0}^{b-1} C_{\sigma,(0,b-1-l)} \right) \\
%&+ \sum_{l=1}^{q-1} \left( \frac{q (d+\delta_{b,l})(d+1+\delta_{b,l})}{2} + l(d+1+\delta_{b,l}) \right) C_{\rho,(0,b-1-l)}\\=
& \frac{(p c^2 + p c)C_\tau}{2} + \left(c+\frac{a}{p}\right) dC_\sigma + \frac{(q d^2 + q d)C_\rho}{2} + c \sum_{k=0}^{p-1} ( \epsilon_{a,k} p + k ) C_{\tau,(a-1-k,b)} + d \sum_{l=0}^{q-1} (\delta_{b,l} \, q + l) C_{\rho,(0,b-1-l)} \\ 
&+ cp \sum_{l=0}^{b-1} C_{\sigma,(0,l)} + \sum_{k=0}^{p-1} \left(\frac{p \epsilon_{a,k}}{2} + \frac{p \epsilon_{a,k}^2}{2} + k (\epsilon_{a,k}+1)\right) C_{\tau,(a-1-k,b)} + \sum_{l=0}^{q-1} \left(\frac{q \delta_{b,l}}{2} + \frac{q \delta_{b,l}^2}{2} + l (\delta_{b,l}+1)\right) C_{\rho,(0,b-1-l)} \\
&+ a \sum_{l=0}^{b-1} C_{\sigma,(0,b-1-l)}, 
\end{align*}
}
The rest is a rearrangement of the terms in the above expression.
\end{proof}

The Riemann theta functions with characteristics in the above theorem form a basis of the $(p,q)$-polarization 
of the abelian surface $A_\Omega$ with the period $\Omega=\begin{bmatrix}
                p \tau & \sigma \\
               \sigma & q \rho 
\end{bmatrix}$. 
In other words $A_\Omega$ is the quotient $(\C^\times)^2/\Z^2$, where the generators of $\Z^2$ act by
$$
(X,Y) \mapsto (e^{2\pi i \tau}X,e^{2\pi i \sigma/p}Y), \ \ \  (X,Y) \mapsto (e^{2\pi i \sigma/q}X,e^{2\pi i \rho}Y). 
$$
Note that $p$ and $q$ are not necessarily coprime in this article.  
%By pulling back this, by a shifted exponential map, we obtain the desired period matrix $\Omega$ (see Appendix A). 

%%%%%%%%%%%%%%%%%%%%%%%%%%%%%%%%%%%%%%%%%%%%%%%%%%%%%%%%%%%%%%%%%%%%%%%%%%%%%%%%%%%%%%%%%%%%%%%%%%%%%%%
%\subsection{Moduli spaces}
\begin{Prop}
The dimension of the complex moduli space of the SYZ mirror of $X_{p,q}$ is $pq+2$. 
\end{Prop}
\begin{proof}
The complex moduli space can be identified with the complex moduli space of abelian surfaces equipped with divisor which gives a principal polarization.  
Therefore it is the sum of the dimension of $\mathfrak{H}_2$ and the dimension of the linear system of the $(p,q)$-polarizations (which is $pq-1$). 
\end{proof}
Combining with Proposition \ref{Kahler dim}, we observe that the dimension of the complex moduli space of the SYZ mirror matches 
with the dimension of the K\"ahler moduli of the local Calabi--Yau 3-fold $X_{p,q}$. 

Let us take a closer look at the complex moduli space in the genus $2$ case (for $(p,q)=(1,1)$). 
One has the toroidal Torelli map 
$\bar{\mathfrak{t}} : \overline{\mathcal{M}}_2 \rightarrow \overline{\mathcal{A}}_2$, extending the classical Torelli map \cite{Ale}.  
Here $\overline{\mathcal{M}}_2$ is the Deligne--Mumford compactification of the moduli space of genus $2$ curves 
and $\overline{\mathcal{A}}_2$ is the toroidal compactification of the moduli space $\mathfrak{H}_2/\Sp_4(\Z)$ 
given by the secondary Voronoi fan. 
There are $3$ irreducible divisors of the form $\overline{\mathcal{M}}_{1,1} \times \overline{\mathcal{M}}_{1,1} \subset \overline{\mathcal{M}}_{2}$ 
(the left figure of Figure \ref{fig:Degeneration}), given by $\sigma=0$, $\tau=\sigma$, and $\rho=\sigma$. 
Their images under $\bar{\mathfrak{t}}$ are the Humbert surface $H_1 \subset \overline{\mathcal{A}}_2$, 
where the corresponding abelian surface splits into the product of $2$ elliptic curves. 
There the total space of the conic fibration whose discriminant loci is given by the theta divisor has a conifold singularity.  
Thus the Humbert surface $H_1 \subset \overline{\mathcal{A}}_2$ serves as the conifold limits of the complex moduli space. 
On the other hand the large complex structure limit $\Omega = i\infty$ is given by the point $\bar{\mathfrak{t}}(\mathcal{M}_{0,3} \times \mathcal{M}_{0,3})$ 
(the right figure of Figure \ref{fig:Degeneration}).
 \begin{figure}[htbp]
 \begin{center} 
  \includegraphics[width=80mm]{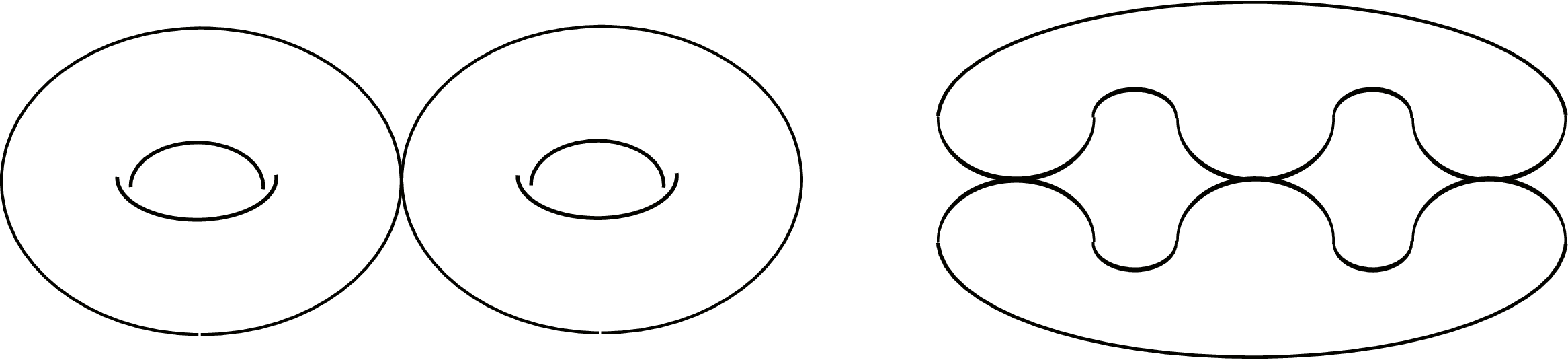}
 \end{center}
  \caption{Conifold loci and large complex structure limit} 
\label{fig:Degeneration}
\end{figure}
Such a degeneration has been studied in Oda--Seshadri \cite{OS} for instance. 
The degeneration limit of abelian surfaces is the union of $2$ copies of $\PP^2$'s glued along 3 $\PP^1$'s \cite[Dual graph (honeycomb) given in Figure 11]{OS}. 
It contains the above stable genus $2$ curve as the theta divisor \cite{Ale}. 
%See the left figure of Figure \ref{fig:LCSL}, which is exactly the square tiling we used in the construction of the local Calabi--Yau 3-fold $X_{(1,1)}$.  

%%%%%%%%%%%%%%%%%%%%%%%%%%%%%%%%%%%%%%%%%%%%%%%%%%%%%%%%%%%%%%%%%%%%%%%%%%%%%%%%%%%%%%%%%%%%%%%%%%%%%%%
\subsection{Fiber-base mirror duality} \label{section:fiber-base mirror dual}

There are various formulation of mirror symmetry for the abelian varieties, for example the work of Golyshev--Lunts--Orlov \cite{GLO} (see also Section \ref{section: speculation}).  
In the surface case, in light of Dolgachev's mirror symmetry \cite{Dol} for the lattice polarized K3 surfaces, 
we formulate mirror symmetry of the abelian surfaces as follows (compatible with \cite{GLO}).  
\begin{Def}
For an algebraic surface $S$, we denote by $NS(S)$ the Neron--Severi lattice and by $T(S)$ the transcendental lattice. 
We call abelian surfaces $A$ and $A'$ mirror symmetric if $NS(A)\oplus U\cong T(A')$ (and thus $NS(A')\oplus U\cong T(A)$). 
\end{Def}
Here $U$ stands for the hyperbolic lattice whose Gram matrix is given by $\left( \begin{array}{cc} 
0 & 1\\ 1 & 0\end{array}\right)$. 
For an abelian surface $A$, we have $H^2(A,\Z)\cong U^{\oplus3}$ as a lattice.  

The large complex structure limit corresponds to the $0$-dimensional cusp in the Bailey--Borel compactification of the period domain.   
That is, a choice of an isotropic vector in the transcendental lattice, giving an orthogonal factor $U$, corresponds to a $0$-dimensional cusp. 
This amounts to a choice of a SYZ fibration, which is an elliptic fibration on the mirror side. 
In our case, there is essentially a unique choice as described above.

\begin{Prop}
A generic fiber of $X_{(p,q)} \rightarrow \D$ and the base $A_\Omega$ of the mirror conic fibration are mirror symmetric. 
%that the transcendental lattice and Neron--Severi lattices are interchanged up to the factor of hyperbolic lattice $U$. 
\end{Prop}
\begin{proof}
A generic fiber of $\pi:X_{(p,q)} \rightarrow \D$ is a $U\oplus\langle -2pq\rangle$-polarized abelian surface\footnote{
Here $\langle n \rangle$ stands for the rank $1$ lattice generated by $a$ with $a^2=n$. 
We say that a compact surface $S$ is $M$-polarized if there is a primitive embedding $M \hookrightarrow NS(S)$ whose image contains an ample divisor}.  
On the other hand, $A_\Omega$ is a $(p,q)$-polarized abelian surface, which is generically a $\langle 2pq\rangle$-polarized abelian surface. 
Thus the transcendental lattice $T(A_\Omega)\cong U^{\oplus2}\oplus \langle-2pq\rangle$. 
\end{proof}

As is the local Calabi--Yau surface $\widetilde{A}_{d-1}$ case, this mirror correspondence (the base-fiber duality) is intuitively clear 
as the period $\Omega$ on one hand represents the symplectic structure of $X_{(p,q)}$ and on the other hand represents complex structure of $A_\Omega$ 
(Figure \ref{fig:MirrorCorresp3}). 
 \begin{figure}[htbp]
 \begin{center} 
  \includegraphics[width=25mm]{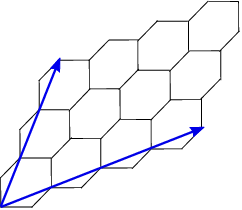}
 \end{center}
  \caption{Toric web diagram and $\Omega$-translation} 
\label{fig:MirrorCorresp3}
\end{figure}
The period matrix $\Omega$ establishes a dictionary of symplectic geometry and complex geometry.

%%%%%%%%%%%%%%%%%%%%%%%%%%%%%%%%%%%%%%%%%%%%%%%%%%%%%%%%%%%%%%%%%%%%%%%%%%%%%%%%%%%%%%%%%%%%%%%%%%%%%%%
\subsection{Mirror symmetry for varieties of general type}

As mentioned in Introduction, the fibration $w:X_{(p,q)} \to \D$ serves as the Landau--Ginzburg mirror of the discriminant locus of the mirror conic fibration \cite{Sei-Spec,AAK}, 
which is a smooth genus $pq+1$ curve.
The critical locus $\mathrm{Crit}(w)$ lies in the central fiber as the singular set. 
It consists of $3pq$ rational curves intersecting at $2pq$ points 
in such a way that exactly $3$ components meet in each point (thus arithmetic genus of $pq+1$). 
This is known as the mirror of a genus $pq+1$ curves. 
More precisely, the mirror of a curve of genus $\ge 2$ is a perverse curve \cite{Rud} 
and indeed $\mathrm{Crit}(w)$ comes equipped with perverse structure given by the sheaf of vanishing cycles of $w:X_{(p,q)} \to \D$.   
\begin{Ex}[Seidel \cite{Sei}]
The set $\mathrm{Crit}(w:X_{(1,1)}  \to \D)$ consists of a union of $3$ rational curves which forms a `$\theta$'-shape.  
This recovers the work of Seidel \cite{Sei} stating that the mirror of a genus $2$ curve is such a Landau--Ginzburg model $w:X_{(1,1)} \to \D$.    
 \begin{figure}[htbp]
 \begin{center} 
  \includegraphics[width=80mm]{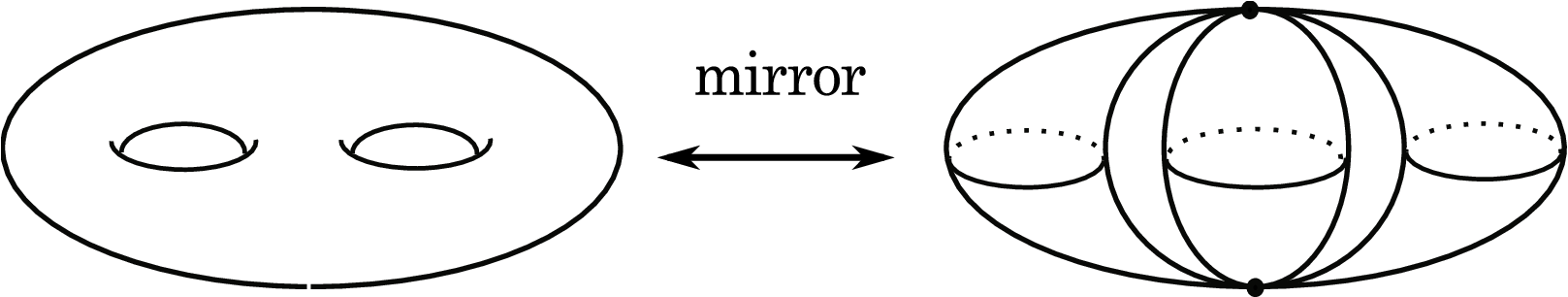}
 \end{center}
  \caption{Genus $2$ curve and mirror perverse curve} 
%\label{fig:MirrorCorresp4}
\end{figure}
\end{Ex}
As we will see in the next section, the SYZ mirror of $X_{(d_1,\ldots,d_l)}$ is a conic fibration over a $(d_1,\ldots,d_l)$-polarized abelian variety, 
with discriminant locus being a hypersurface defined by the $(d_1,\ldots,d_l)$-polarization. 
Then the Landau--Ginzburg model $w:X_{(d_1,\ldots,d_l)} \to \D$ serves as a mirror of this general-type hypersurface.  
%More precisely, the discriminant locus of the conic fibration and critical locus of the potential $w$ are mirror symmetric and share the same skeleton. 
We refer the reader to the work of Gross--Katzarkov--Ruddat \cite{GKR} for proposals of the Landau--Ginzburg mirrors of the varieties of general type,
where they show the interchange of Hodge numbers expected in mirror symmetry. 
This exchange occurs between the Hodge numbers of the discriminant locus 
and certain Hodge numbers associated to a mixed Hodge structure of the perverse sheaf of vanishing cycles on the critical locus $\mathrm{Crit}(w)$. 

%%%%%%%%%%%%%%%%%%%%%%%%%%%%%%%%%%%%%%%%%%%%%%%%%%%%%%%%%%%%%%%%%%%%%%%%%%%%%%%%%%%%%%%%%%%%%%%%%%%%%%%
%%%%%%%%%%%%%%%%%%%%%%%%%%%%%%%%%%%%%%%%%%%%%%%%%%%%%%%%%%%%%%%%%%%%%%%%%%%%%%%%%%%%%%%%%%%%%%%%%%%%%%%

\section{Local Calabi--Yau manifolds of type $\widetilde{A}$ in high dimensions} \label{section: higher-dim}
%%%%%%%%%%%%%%%%%%%%%%%%%%%%%%%%%%%%%%%%%%%%%%%%%%%%%%%%%%%%%%%%%%%%%%%%%%%%%%%%%%%%%%%%%%%%%%%%%%%%%%%
%\subsection{SYZ mirror of $X_{\Sigma}/_{(1,\dots,1)}\Z^l$} 
All the above results have natural generalizations to the local Calabi--Yau $(l+1)$-fold $X_{(d_1,\ldots,d_l)}$. %which is a crepant resolution of the fiber product of $\tilde{A}_{d_j}$ surfaces over a disc. 
%Finally we consider SYZ mirror of the Calabi--Yau manifold associated to a hyperrectangle tiling of $\R^l$ for $l \in \N$.  
The SYZ mirror would be given in terms of genus $l$ Riemann theta functions with characteristics, 
and the generating function of open Gromov--Witten invariants has modular properties similar to Proposition \ref{prop:modular}. 
We shall be brief in this section. 

%\subsection{SYZ mirror of $X_{(d_1,\ldots,d_l)}$} \label{section: Xd}

We realize $X_{(d_1,\ldots,d_l)}$ as a quotient by $\Z^l\cong d_1\Z \times \ldots \times d_l\Z \subset \Z^l$ of a toric Calabi--Yau manifold of infinite-type 
whose fan is obtained from a $\Z^l$-invariant lattice triangulation of the hypercubic tiling
$\left\{\prod_{k=1}^l [i_k,i_k+1]\right\}_{ (i_1,\ldots,i_l) \in \Z^l }$ of $\R^l$. 
For instance, we can take the fan whose dual gives a zonotope tiling of $\R^l$, generalizing the honeycomb tiling in the 3-fold case (Figure \ref{fig: zonotope}) . 
Similar to Lemma \ref{lem:invP-3fold}, we can consider the following invariant polyhedral set in $\R^{l+1}$ (whose projection to $\R^l$ is a zonotope tiling).

\begin{Lem}
Take $c_{k_1,\ldots,k_l} = -\sum_{i=1}^l k_i(k_i-1) - \sum_{i<j} k_i k_j$.  Then the polyhedral set
$$P := \bigcap_{(k_1,\ldots,k_l) \in \Z^l} \left\{ (y_1,\ldots,y_{l+1}) \in \R^{l+1} \ | \ \sum_{i=1}^l k_i y_i + y_{l+1} \geq c_{k_1,\ldots,k_l} \right\}$$
is invariant under the $\Z^l$-action generated by 
$$
e_i \cdot (y_1,\ldots,y_l,y_{l+1}) = (y_1,\ldots,y_l,y_{l+1}-y_i) -2 e_i - \left(\sum_{j\not=i,l+1} e_j\right) + 2 e_{l+1}.
$$
 \begin{figure}[htbp]
 \begin{center} 
  \includegraphics[width=40mm]{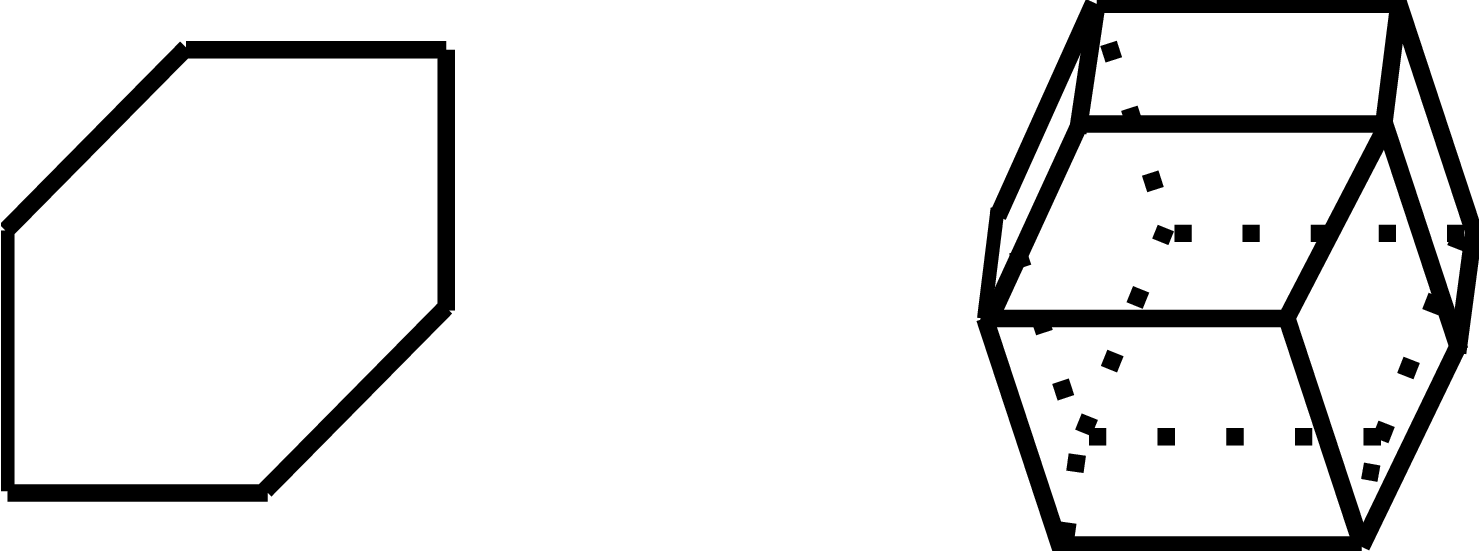}
 \end{center}
  \caption{2 and 3-dimensional zonotopes} 
\label{fig: zonotope}
\end{figure}
\end{Lem}

We note that in contrast to the square case ($3$-fold case), there is no canonical choice of a triangulation of the hypercube in the high dimensions. 
In fact such a choice corresponds to how a family of $(d_1,\ldots,d_l)$-polarized abelian varieties extends 
over the large complex structure limit on the Voronoi compactification of the complex moduli space of the $(d_1,\ldots,d_l)$-polarized abelian varieties. 

We define the curve classes $C_{i,j}^{(v)}$ of the corresponding toric Calabi--Yau manifold of infinite-type for $1 \le i \le j \le l$ and $v \in \Z^l$ as follows. 
For $i=j$, the curve $C_i^{(v)} := C_{i,i}^{(v)}$ is characterized by
$$
C_i^{(v)} \cdot D_{v + 2 e_i} =  C_i^{(v)} \cdot D_{v} = 1; \, C_i^{(v)} \cdot D_{v + e_i} = -2; \, C_i^{(v)} \cdot D_w = 0 \textrm{ for } w \not= v, v+e_i, v+ 2 e_i.
$$
This is analogous to $C_\rho$ and $C_\tau$ in the case of $X_{1,1}$.
For $i < j$, the curve $C_{i,j}^{(v)}$ is characterized by
$$
C_{i,j}^{(v)} \cdot D_{v + e_i + e_j} = D_{v} = 1, \,  C_{i,j}^{(v)} \cdot D_{v + e_i} = C_{i,j}^{(v)} \cdot D_{v+e_j} = -1.
$$
This is analogous to $C_\sigma$ in the case of $X_{1,1}$.
%Due to the $\Z^l$-symmetry, $C_{i,j} = C_{i,j}^{(v)}$ for different $v$ are identified and give curve classes of the quotient.

Note that $C_i$ always represents an effective curve class; while for $i<j$, whether $C_{i,j}$ is an effective curve class depends on the actual crepant resolution. 
In other words for $i<j$, $q_{i,j}^{(v)}$ are local coordinates defined on a punctured neighborhood, which is isomorphic to $(\C^\times)^{\frac{l(l+1)}{2}}$), 
of the large volume limit (which may not extend to the limit).
%The expression of $C_{i,j}$ in terms of effective curve classes depends on the actual triangulation.

%inductively as follows.  For $n=2$, the curve class $C^2_{11}$ is represented by a $(-2)$-curve which is also a toric prime divisor (due to the $\Z$ quotient all toric prime divisors are identified in this dimension).  Now suppose the curve classes $(C^{n-1}_{ij})_{1 \le i \le j \le n-2}$ have been chosen for dimension $n-1$.  For dimension $n$, $C_{ii}$ for $1 \leq i \leq n$ is the curve class whose projection under $\pi_i$ is $C^2_{11}$.  $C_{ij}$ for $1 \leq i < j \leq n-1$ is the curve class whose projection under $\pi_{ij}$ is $C^{n-1}_{ij}$.

%Under the above choice of K\"ahler parameters, the SYZ mirror admits a beautiful expression in terms of Riemann theta functions as in dimension $2$ and $3$. 
%First we consider $X_{1,\ldots,1}$.
\begin{Thm} \label{thm:mir_X_1}
%Let $X_{\Sigma}/_{(1,\dots,1)}\Z^l$ be the Calabi--Yau manifold associate to the hypercubic tiling in $\R^l$. 
The SYZ mirror of the local Calabi--Yau $(l+1)$-fold $X_{1,\ldots,1}$ is given by $uv = F^\open(z_1,\dots,z_l;q)$ where
$$
F^\open(z_1,\ldots,z_l;q)= \Delta_l (q) \cdot 
\Theta_l
\begin{bmatrix}
                0 \\
                \left(- \frac{\tau_{1,1}}{2},\dots,- \frac{\tau_{l,l}}{2}\right)
\end{bmatrix}
\left({\bf \zeta};\Omega\right),
$$
$z_i = e^{2\pi i \zeta_i}$, $q = \left(q_{i,j} = e^{2\pi i \tau_{i,j}} \right)_{i \le j}$, $\Omega$ is the symmetric $l$-by-$l$ matrix $(\tau_{i,j})_{i,j=1}^l$ with $\tau_{j,i} = \tau_{i,j}$, and 
$$
\Delta_{l}(q)
= \exp \left(\sum_{n \geq 2} \frac{(-1)^n}{n} 
\sum_{\substack{({\bf m}_i=(m_i^j)\in \Z^{l} \setminus 0)_{i=1}^n\\ \textrm{with } \sum_{j=1}^n {\bf m}_j =0}}
\exp\left(\sum_{k=1}^n \pi i \, {\bf m}_k\cdot \Omega \cdot {\bf m}^T_k \right) \right).
$$
\end{Thm}

\begin{proof}
The proof is almost identical to the 3-fold case and we shall be very brief. 
The basic idea is to consider 2-dimensional projections and reduce to the 3-fold case (Figure \ref{fig:Cubic Tiling}).   
The key is to express the curve class
$$\cC^{(a_1,\ldots,a_{l})} := \beta_{((a_1,\ldots,a_{l}),1)} - \beta_{(0,1)} -  \sum_{i=1}^{l} a_i \left( \beta_{(e_i,1)} - \beta_{(0,1)} \right) $$
in terms of the curve classes $C_{ij}$ defined above. 
It can be verified that
$$
\cC^{(a_1,\ldots,a_{l})} 
= \sum_{i=1}^{l} \left(\sum_{k=1}^{a_i-1} (a_1 C_{1,i}^{(v_{i,k})} + \ldots  + a_{i-1} C_{i-1,i}^{(v_{i,k})} + k C_{i,i}^{(v_{i,k})}) + \sum_{p=1}^i a_p C^{(w_i)}_{p,i+1}  \right)
$$
where $C^{(v)}_{p,l} := 0$ for all $p$ and $v$, $v_{i,k} = (0,\ldots,0,a_i-k-1,a_{i+1},\ldots,a_{l})$, $w_i = (0,\ldots,0,a_{i+1}-1,a_{i+2},\ldots,a_{l})$.
\begin{figure}[htbp]
 \begin{center} 
  \includegraphics[width=80mm]{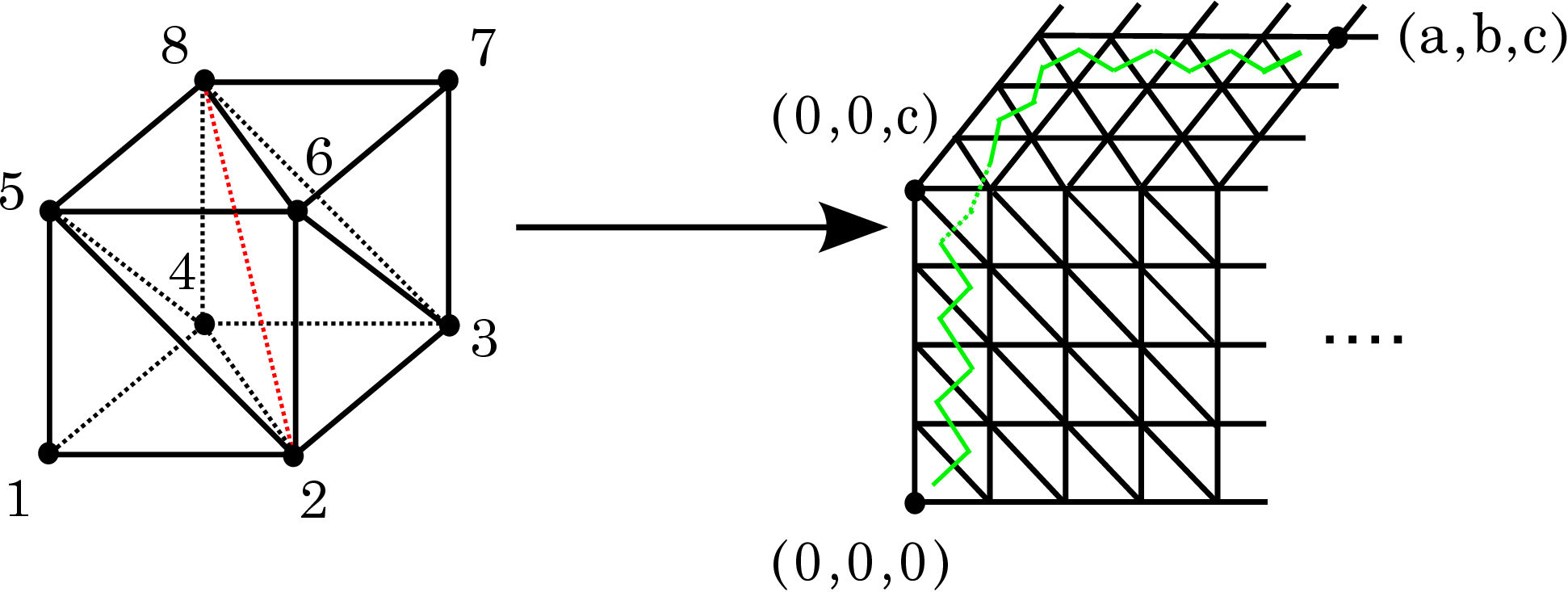}
 \end{center}
 \caption{Cubic tiling and curve connecting the origin and $(a,b,c)$}
\label{fig:Cubic Tiling}
\end{figure}
\end{proof}

%A natural generalization of Conjecture \ref{conj :Delta} is the following: 
%\begin{Conj}
%There exists a polynomial $F_l(q)$ such that   
%$$
%\frac{\left(\prod_{i=1}^{\frac{l(l+1)}{2}}q_i\right)F_l(q)}{\Delta_l(q)^{24}}
%$$
%is a Siegel modular form of genus $l$ with an asymptotic behaviour similar to Eq (\ref{eq: asymptotic}). 
%\end{Conj}
%
%Some candidates for such Siegel modular forms are the cusp form $\chi_{18}(\Omega)$ of genus $3$ and the Schottky form $\chi_8(\Omega)$ of genus $4$. 
%This conjecture particularly implies that there exists in each genus a Siegel modular form which asymptotically behaves as the product of the eta functions on the diagonals. 
%Moreover, the existence of such an interesting asymptotic behaviour is geometrically explained by the fiber product structure of the Calabi--Yau manifold $X_{\Sigma}/_{(1,\dots,1)}\Z^l$ 
%just as the 3-fold case  (see Section \ref{section: 3-fold}). 

%%%%%%%%%%%%%%%%%%%%%%%%%%%%%%%%%%%%%%%%%%%%%%%%%%%%%%%%%%%%%%%%%%%%%%%%%%%%%%%%%%%%%%%%%%%%%%%%%%%%%%%
%We are now in position to generalize the above result.  Let ${\bf d}=(d_1,\dots,d_l) \in \Z_{>0}^l$ and $X_{{\bf d}}:=X_{\Sigma}^o/\Z^l$ be the Calabi--Yau manifold associate to the $(d_1,\dots,d_l)$-hyperrectangular tiling of $\R^l$.  Needless to say, it is a crepant resolution of the intreated fiber product $\widehat{A_{d_1}}\times_\D \dots \times_\D \widehat{A_{d_l}}$. 

Now we state the result for the general case $(d_1,\ldots,d_l) \in \mathbb{N}^l$ but omit the proof.
\begin{Thm} \label{thm:SYZ-genA}
The SYZ mirror of $X_{(d_1,\ldots,d_l)}$ is given by the conic fibration $uv = F^\open$ where 
$$
F^\open = \sum_{a_1,\ldots,a_l = 0}^{d_1-1,\ldots,d_l-1}  K_{(a_1,\ldots,a_l)} \cdot \Delta_{(a_1,\ldots,a_l)} \cdot \Theta_{l}',
$$
$\Delta_{(a_1,\ldots,a_l)} = \sum_\alpha n_{\beta_{a_1,\ldots,a_l} + \alpha} q^\alpha$, $\Theta_l'$ is the Riemann theta function with characteristics 
$$
\Theta_l \begin{bmatrix}
\left(\frac{a_1}{d_1}, \ldots, \frac{a_l}{d_l}\right)  \\
\left(\frac{-d_1 \tau_1}{2} + {\displaystyle\sum_{k=0}^{d_1-1}} k \tau_{1,(-1-k,0,\ldots,0)}, \ldots,\frac{-d_l \tau_l}{2} + {\displaystyle\sum_{k=0}^{d_l-1}} k \tau_{l,(0,\ldots,0,-1-k)}\right) 
\end{bmatrix}
\left(d_1 \cdot \zeta_1,\ldots, d_l \cdot \zeta_l; \Omega\right)
$$
where
$$
\Omega:=\begin{bmatrix}
                d_1 \tau_1 & \sigma_{(1,2)} & \ldots & \sigma_{(1,l)} \\
               \sigma_{(1,2)} & d_2 \tau_2 & \ldots & \sigma_{(2,l)} \\
							 \vdots & \vdots & \ddots & \vdots \\
							 \sigma_{(1,l)} & \sigma_{(2,l)} & \ldots & d_l \tau_l
\end{bmatrix}
$$
and $K_{(a_1,\ldots,a_l)}$ is the following quantity which is independent of $\zeta_i$:
\begin{align*}
K_{(a_1,\ldots,a_l)} =& \left(\prod_{i=1}^l Q_i^{-\frac{a_i^2}{2 d_i} + \frac{a_i}{2}}\right) \cdot \left( \prod_{1\leq i < j \leq l} Q_{(i,j)}^{-\frac{a_i a_j}{d_i d_j}} \right) \cdot \left( \prod_{i=1}^{l} \left(\prod_{k=0}^{d_i-1} Q_{i,(-1-k) \ve_i}^k \right)^{-\frac{a_i}{d_i}} \right) \\
&\cdot \left( \prod_{i=1}^l  \prod_{k=0}^{a_i-1} Q_{i,(a_i-1-k) \ve_i}^k \right) \left( \prod_{0 \leq k < i \leq l} \left(\prod_{j=0}^{a_i-1} Q_{(k,i),(0,\ldots,0,j,a_{i+1},\ldots,a_{l})}\right)^{a_k} \right).
\end{align*}
In the above $\tau_i = \sum_{k=0}^{d_i-1} \tau_{i,(a_1,\ldots,a_{i-1},k,a_{i+1},\ldots,a_l)}$ and 
$$\sigma_{(i,j)} = d_i \sum_{k=0}^{d_j-1} \sigma_{(i,j),(a_1,\ldots,a_{j-1},k,a_{j+1},\ldots,a_l)} = d_j \sum_{k=0}^{d_i-1} \sigma_{(i,j),(a_1,\ldots,a_{i-1},k,a_{i+1},\ldots,a_l)}$$ 
which are independent of $a_1,\ldots,a_l$. 
We set
\begin{align}
Q_i &:= \exp 2\pi i \tau_i, \ \ \ \ \ \ \ \ \ \ Q_{i,(a_1,\ldots,a_l)} := \exp 2\pi i \tau_{i,(a_1,\ldots,a_l)}\notag \\
Q_{(i,j)} &:= \exp 2\pi i \sigma_{(i,j)}, \ \ \ Q_{(i,j),(a_1,\ldots,a_l)} := \exp 2\pi i \sigma_{(i,j),(a_1,\ldots,a_l)}.\notag 
\end{align} 

In particular, the divisor $F^\open(z_1,\dots,z_l;q)=0$ defines the $(d_1,\dots,d_l)$-polarization of the ambient abelian variety. 
\end{Thm}

By Theorem \ref{thm:open-mir-thm}, the generating functions $\Delta_{(a_1,\ldots,a_l)}$ of the open Gromov--Witten invariants can be computed by the mirror map. 
They are higher-dimensional analogs of the Dedekind eta function and multi-variable theta functions (c.f. Proposition \ref{prop:modular}). 
As is the 3-fold case, the function $\Delta_{(a_1,\ldots,a_l)}$ admits an interesting asymptotic behavior given by the product of the Dedekind eta functions. 
We wish that they produce an interesting new class of higher genus Siegel modular forms.

The fiber-base mirror duality still holds in higher dimensions. 
Namely, the generic fiber of $X_{(d_1,\ldots,d_l)}\to \D$, which is the product of isogeneous elliptic curves $\C^\times/t^{d_i \Z}$ $(i=1,\dots,l)$, 
and the base of conic fibration of the SYZ mirror of $X_{(d_1,\ldots,d_l)}$, which is a $(d_1,\ldots,d_l)$-polarized abelian variety, are mirror symmetric. 
For instance, this mirror correspondence has been verified in the work of Golyshev--Lunts--Orlov \cite[Proposition 9.6.1 and Corollary 9.6.3]{GLO}. 

In light of the $3$-fold case $w:X_{(d_1,\ldots,d_l)} \to \D$ should serve as the Landau--Ginzburg mirror 
of the hypersurface in the mirror abelian variety defined by the $(d_1,\ldots,d_l)$-polarization. 
%It is an interesting problem to investigate further this mirror correspondence.   
Mirror symmetry for the varieties of general type is still lurking and only partially explored area, and deserves further explorations.

\begin{comment}
\begin{Prop} \label{Trans}
The generating function of open Gromov--Witten invariants $\Delta_{(a_1,\ldots,a_l)}$ appeared in Theorem \ref{thm:SYZ-genA} is determined by the relations
$$ \Delta_{(a_1,\ldots,a_l)}(q^C \mapsto q^{C^{(-\ve_i)}}) = \Delta_{(a_1,\ldots,a_{i-1},a_i-1,a_{i+1},\ldots,a_l)}$$
for all $i = 1,\ldots, l$ and that the formal expansion of
$$ \log F^\open = \sum_{k=1}^\infty \frac{(-1)^{k-1}}{k} \left(-1 + \sum_{a_1,\ldots,a_l = 0}^{d_1-1,\ldots,d_l-1}  K_{(a_1,\ldots,a_l)} \cdot \Delta_{(a_1,\ldots,a_l)} \cdot \Theta_{l}' \right)^{k} $$
in $(z_1, \ldots, z_l) = 0$ has no constant term, where we set $z_i = \exp 2\pi i \zeta_i$.
\end{Prop}

\begin{proof}
From the translation $\Z^l$-symmetry we have the relation $n_{\beta_{\vh - \ve_i} + \alpha^{(-\ve_i)}} =  n_{\beta_{\vh} + \alpha}$.
Under the change of coordinates $q^C \mapsto q^{C^{(-\ve_i)}}$, we this have
\begin{align*}
\Delta_{(a_1,\ldots,a_l)}(q^C \mapsto q^{C^{(-\ve_i)}}) &= \sum_\alpha\left( q^{\alpha^{(-\ve_i)}} n_{\beta_{\vh} + \alpha} \right) \\
&= \sum_\alpha \left( q^\alpha n_{\beta_{\vh} + \alpha^{(\ve_i)}} \right) \\
&= \sum_\alpha \left( q^\alpha n_{\beta_{\vh - \ve_i} + \alpha} \right).
\end{align*}
The last part of the assertion follows from the Gross--Siebert normalization of the slab function as before \cite{Lau,GS3}. 
\end{proof}
\end{comment}

%%%%%%%%%%%%%%%%%%%%%%%%%%%%%%%%%%%%%%%%%%%%%%%%%%%%%%%%%%%%%%%%%%%%%%%%%%%%%%%%%%%%%%%%%%%%%%%%%%%%%%%
\section{Speculation} \label{section: speculation}

The mirror correspondence studied in this article has natural generalizations. 
We propose the following mirror correspondence, replacing the abelian varieties by more general Calabi--Yau manifolds.   
Let $(X,Y)$ be a mirror pair of Calabi--Yau manifolds. 
For simplicity let us assume that the complex moduli space of $X$ and the K\"ahler moduli space of $Y$ are 1-dimensional\footnote{
In general we consider a $1$-dimensional family of Calabi--Yau manifolds and a polarized mirror Calabi--Yau manifold 
in such a ways that the deformation direction in the complex moduli space corresponds to the polarization direction in the K\"ahler moduli space of the mirror.}. 
We consider a conic fibration $g:\mathcal{Y}\rightarrow Y$ degenerating along a smooth divisor $L$ which gives the ample generator of $\mathrm{Pic}(Y)\cong \Z$, 
and the degeneration family $f:\mathcal{X}\rightarrow \mathbb{D}$ near the large complex structure limit of $X$ where the central fiber $X_0$ is the only singular fiber. 
We anticipate that the total spaces $\mathcal{X}$ and $\mathcal{Y}$ form a mirror pair of Calabi--Yau manifolds\footnote{
Note that the total space $\mathcal{Y}$ is always taken to be a Calabi--Yau manifold.}. 
Moreover, the critical loci $\mathrm{Crit}(f)$ furnished with the perverse structure coming form vanishing cycles (or the Landau--Ginzburg model $(\mathcal{X},f)$) 
and the discriminant loci $L = \mathrm{Disc}(g)$ form a mirror pair of varieties of general type.  

We can make sense of the above conjecture by homological mirror symmetry, 
namely the derived category of sheaves on $\mathcal{X}$ supported at $X_0$ (or the matrix factorization category of the Landau--Ginzburg model $(\mathcal{X},f)$ instead) 
is quasi-equivalent to the (split closure of) derived Fukaya subcategory of $\mathcal{Y}$ generated by the Lagrangians coming from $L$ (or the derived Fukaya category of $L$ instead resp.). 
We conclude by the following table.
%Here we need to make sense of a {\it mirror pair} as there seems no definite formulation of mirror symmetry for the non-compact Calabi--Yau manifolds and the varieties of general type.  In our case for $Y$ being a polarized abelian variety, $\mathcal{X}$ and $\mathcal{Y}$ are mirror symmetric in the sense of SYZ. 
\begin{center}
 \begin{tabular}{c|c|c} 
 $(n+1)$-dim  & total space CY manifold $\mathcal{X}$   & total space CY manifold $\mathcal{Y}$ \\ \hline
 $n$-dim        & fiber CY manifold $X$ &  base CY manifold $Y$ \\ \hline
$(n-1)$-dim & perverse critical loci $\mathrm{Crit}(f)$ & dicsriminant loci $\mathrm{Disc}(g)$\\  %LG model $(\mathcal{X},f)$
 \end{tabular}
 \end{center}
For instance, let $Y$ be a $\langle 2n\rangle$-polarized K3 surface\footnote{
For $1\le n \le 4$, such a K3 surface is generically a complete intersection in a weighted projective space 
$\PP_{(1^3,3)}\cap (6)$, $\PP_{(1^4)}\cap (4)$, $\PP_{(1^5)}\cap (2,3)$ and $\PP_{(1^6)}\cap (2^3)$ respectively. 
For $5\le n \le 9$, Mukai showed that it is generically a complete intersection in a homogeneous space.}  
and $X$ its mirror K3 surface in the sense of Dolgachev \cite{Dol} for $1 \le n \le 4$. 
For the mirror family $f:\mathcal{X}\rightarrow \mathbb{D}$ near the large complex structure limit,  
the configuration $\mathrm{Crit}(f)$ of $\PP^1$'s is given in Figure \ref{fig: K3Graph}, as mirror symmetry for the genus $n+1$ curve $L$ expects \cite{Sei,Efi}. 
\begin{figure}[htbp]
 \begin{center} 
  \includegraphics[width=80mm]{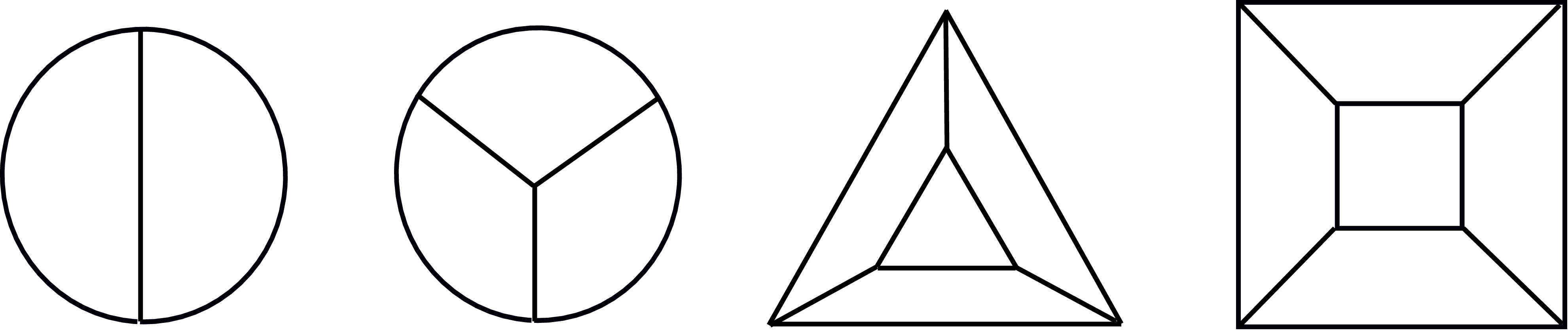}
 \end{center}
 \caption{Configuration of $\PP^1$'s mirror to genus $n+1$ curve $L$ for $(1\le n \le 4)$}
\label{fig: K3Graph}
\end{figure}
The total space $\mathcal{X}$ can also be taken to be a Calabi--Yau 3-fold (the Kulikov model of type III).

%%%%%%%%%%%%%%%%%%%%%%%%%%%%%%%%%%%%%%%%%%%%%%%%%%%%%%%%%%%%%%%%%%%%%%%%%%%%%%%%%%%%%%%%%%%%%%%%%%%%%%%
%%%%%%%%%%%%%%%%%%%%%%%%%%%%%%%%%%%%%%%%%%%%%%%%%%%%%%%%%%%%%%%%%%%%%%%%%%%%%%%%%%%%%%%%%%%%%%%%%%%%%%%
\section{Appendix} \label{abelian varieties}
We consider an $m$-dimensional complex torus $X=\C^m/\Lambda$. 
Here $e_1,\dots, e_m$ be a complex basis of $\C^m$ 
and $\Lambda$ be the lattice generated by the $2m$ independent vectors $\lambda_i=\sum\widetilde{\omega}_{\alpha i} e_\alpha$ in $\R^{2m}\cong \C^m$. 
We define the differentials $dz_\alpha$ and $dx_i$ in such a way that $\int_{e_\beta}dz_\alpha=\delta_{\alpha \beta}$ and $\int_{\lambda_j}dx_i=\delta_{ij}$ hold.  
The $m\times 2m$ matrix $\widetilde{\Omega}=(\widetilde{\omega}_{\alpha i})$ is called the period matrix and the lattice $\Lambda$ is generated by the $2m$ columns of $\widetilde{\Omega}$. 
The Kodaira embedding theorem asserts that the complex torus $X$ is an abelian variety if and only if it admits a Hodge form (an integral closed positive $(1,1)$-form)
$$
\omega=i\sum_{\alpha,\beta}h_{\alpha,\beta}dz_\alpha \wedge d\overline{z}_\beta. 
$$
We may pick a new basis of $\C^m$ and $\Lambda$, not in a unique way, such that
$$
\widetilde{\Omega}=\begin{bmatrix}
                \delta_1   &  & 0   &  \\
                 & \ddots &  &   \Omega\\
                 0 & & \delta_m  &
\end{bmatrix},
$$
where $\Omega=(\omega_{ij}) \in \mathfrak{H}_m$ and integers $\delta_i\ge 1 (1\le i \le m)$ such that $\delta_i|\delta_{i+1}$. 
%(We can take $\delta_1,\dots, \delta_m$ so that $\delta_i|\delta_{i+1}$, but we {\it do not} impose this divisibility conditions in this article for convenience.)
%$\delta_i \ge 1$ satisfy $\delta_i|\delta_{i+1}$. 
%Here $\mathfrak{H}_m$ is the Siegel upper half-space of degree $m$ is the set of $m \times m$ symmetric matrices over $\C$ whose imaginary part is positive definite. 
Here $\mathfrak{H}_m$ is the Siegel upper half-space of degree $m$ defined as
$$
\mathfrak{H}_m:=\{\Omega \in \mathrm{M}_m(\C) \ | \ \Omega^t=\Omega, \ \mathrm{Im}(\Omega)>0 \}.
$$
In these new coordinates, $\omega$ takes of the form
$$
\omega=\sum_\alpha \delta_\alpha dx_\alpha \wedge dx_{m+\alpha}. 
$$
The cohomology class $[\omega]$ of the Hodge form, or equivalently the sequence of integers $(\delta_1,\dots,\delta_m)$, 
provides the so-called $(\delta_1,\dots,\delta_m)$-polarization of the abelian variety $X$.  
The sequence $(\delta_1,\dots,\delta_m)$ is an invariant of the cohomology class $[\omega]$ and independent of the choice of a basis. 
When $\delta_1=\dots=\delta_m=1$, the abelian variety $X$ is called principally polarized.  
By abuse of notation, we {\it do not} impose the divisibility condition $\delta_i|\delta_{i+1}$ in this article, 
but you can always find a new basis with respect which the corresponding sequence $(\delta'_1,\dots,\delta'_m)$ satisfies the divisibility condition.  
%For the standard explanation, we assume $\delta_i|\delta_{i+1}$, but we {\it do not} impose this divisibility conditions in this article. 
  
For $\bfa,\bb \in \R^m$, the genus $m$ Riemann theta function with characteristic
$
\begin{bmatrix}
                \bfa \\
                \bb
\end{bmatrix}$  
is defined by 
$$
\Theta_m
\begin{bmatrix}
                \bfa \\
                \bb
\end{bmatrix}
(\zz; \Omega):=\sum_{\nn \in \Z^m}\exp 2\pi i\left(\frac{1}{2}(\nn+\bfa)\cdot\Omega(\nn+\bfa)+ (\nn+\bfa)\cdot(\zz+\bb)\right), 
$$
where $\zz \in \C^m, \ \Omega\in \mathfrak{H}_m$. 
We allow the shift $\bb$ to be in $\C^m$ for simplicity of notations in this article. 
We also denote $\Theta_m\begin{bmatrix}
                0 \\
                0
\end{bmatrix}
(\zz; \Omega)$ by $\Theta_m(\zz; \Omega)$. 

Let $L$ be the line bundle associated to the Hodge form $\omega$. 
It is known that $H^0(X,L)$ has a basis given by the theta functions 
$$
\Theta_m
\begin{bmatrix}
                (\frac{i_1}{\delta_1},\dots,\frac{i_m}{\delta_m}) \\
                0
\end{bmatrix}
(\zz; \Omega), \ \ \ (0 \le i_k \le \delta_k-1). 
$$
It is also useful to realize $X$ as $(\C^\times)^m/\Z^m$ via the shifted exponential map 
$$
\exp:\C^m \rightarrow (\C^\times)^m, \ \ \  (z_1,z_2,\dots,z_m) \mapsto (e^{2\pi i\delta_1z_1},e^{2\pi i\delta_2z_2},\dots,e^{2\pi i\delta_mz_m}). 
$$ 
Then $X$ can be thought as a quotient of $(\C^\times)^m$ by the equivalent relations for $(y_1,y_2,\dots,y_m) \in (\C^\times)^m$: 
$$
(y_1,y_2,\dots,y_m) \sim (e^{2\pi i\frac{\omega_{i1}}{\delta_1}}y_1,e^{2\pi i\frac{\omega_{i2}}{\delta_2}}y_2,\dots,e^{2\pi i\frac{\omega_{im}}{\delta_m}}y_m), \ \ \ (1 \le i \le m). 
$$ 

\par\noindent{\scshape \small
Department of Mathematics, Kyoto University\\
Kitashirakawa-Oiwake, Sakyo, Kyoto, 606-8502, Japan}
\par\noindent{\ttfamily akanazawa@math.kyoto-u.ac.jp}\\

\par\noindent{\scshape \small
Department of Mathematics and Statistics, Boston University\\
111 Cummington Mall, Boston MA 02215 USA}
\par\noindent{\ttfamily lau@math.bu.edu}

\end{document}